\documentclass[11pt]{amsart}

\usepackage{amssymb,epsf,psfrag,amsthm}
\usepackage[mathscr]{eucal}
\usepackage{graphicx}
\usepackage{color}
\usepackage{amsfonts}
\usepackage{enumerate}
\usepackage{latexsym}
\usepackage{epsfig}
\usepackage{amsthm}
\usepackage{amsmath}
\usepackage{latexsym}
\usepackage[all]{xy}
\usepackage{mathtools}
\usepackage{tikz}
\usepackage{quiver}
\usepackage{soul}
\usepackage[section]{placeins}
\usepackage{float}
\usepackage[colorlinks=true,   urlcolor=blue, pdfpagemode=UseOutlines, pdfstartview=FitH]{hyperref}

\newcommand{\C}{\mathbb{C}}
\newcommand{\Z}{\mathbb{Z}}
\newcommand{\R}{\mathbb{R}}

\DeclareMathOperator{\diag}{diag}

\makeatletter\@addtoreset{equation}{section}\makeatother
\makeatletter\@addtoreset{figure}{section}\makeatother
\makeatletter\@addtoreset{table}{section}\makeatother

\makeatletter\@addtoreset{equation}{section}\makeatother
\makeatletter\@addtoreset{figure}{section}\makeatother
\makeatletter\@addtoreset{table}{section}\makeatother

\newtheorem{theorem}{Theorem}[section]
\newtheorem{proposition}[theorem]{Proposition}

\newtheorem{lemma}[theorem]{Lemma}

\newtheorem{cor}[theorem]{Corollary}

\newtheorem{questionL}{Question}

\newtheorem*{main theorem}{Main Theorem}

\newtheoremstyle{remark}
{9pt}{9pt}{}{}{\bfseries}{}{.5em}{}
\theoremstyle{remark}
\newtheorem{remark}[theorem]{Remark}

\newtheorem{example}[theorem]{Example}

\newtheoremstyle{definition}
{9pt}{9pt}{}{}{\bfseries}{}{.5em}{}
\theoremstyle{definition}
\newtheorem{definition}[theorem]{Definition}
\numberwithin{equation}{section}
\numberwithin{theorem}{section}
\numberwithin{figure}{section}

\newcommand{\purge}[1]{}
\newcommand{\bb}{\emph{b}}
\newcommand{\ha}{\hat{\alpha}}
\newcommand{\hb}{\hat{\beta}}

\begin{document}

\title{On a symplectic generalization of a Hirzebruch problem}

\author{Leonor Godinho, Nicholas Lindsay and Silvia Sabatini}

\maketitle

\begin{abstract}
Motivated by a problem of Hirzebruch, we study $8$-dimensional, closed, symplectic 
manifolds having a Hamiltonian torus action with isolated fixed points and second
Betti number equal to $1$. Such manifolds are automatically positive monotone. Our
main result concerns those endowed with a Hamiltonian $T^2$-action and fourth Betti
number equal to $2$. We classify their isotropy data, (equivariant) cohomology rings
and (equivariant) Chern classes, and prove that they agree with those of certain
explicit Fano $4$-folds with torus actions.

Moreover, under more general assumptions, we prove several finiteness results
concerning Betti and Chern numbers of $8$-dimensional, positive monotone symplectic
manifolds with a Hamiltonian torus action.
\end{abstract}

\tableofcontents

\section{Introduction}
 
A complex manifold $X$ of complex dimension $n$ is called a compactification of $\mathbb{C}^n$, if there exists a complex analytic subset $V \subset X$ such that $X \setminus V$ is biholomorphically equivalent to $\mathbb{C}^n$. A celebrated problem of Hirzebruch aims at classifying compactifications of $\mathbb{C}^n$ with second Betti number $b_2$
equal to one \cite[Problem 27]{Hi}.

A projective compactification $X$ with $b_{2}(X)=1$ is automatically a Fano variety (see \cite[page 2]{PZ}). 
In this case, there has been steady progress on Hirzebruch's problem since its formulation. In particular, combined work of several authors gives a complete solution when $n \leq  3$ \cite{PS,Fu1, P2, Fu2}. Moreover, when $n = 4$, and the Fano index\footnote{The Fano index of a Fano variety $X$ is  the largest integer $r\geq1$ such that $H^{\otimes r}$ is isomorphic to $K^{-1}_X$ for some divisor $H \in \text{Pic} (X)$, where $K^{-1}_X$ is the anti-canonical line bundle of $X$.} is equal to $3$, the problem was solved by Prokhorov \cite{P}. In this case, $X$ is unique up to isomorphism, and there are only four possibilities for $V \subset X$.

A rich subclass of projective compactifications of $\mathbb{C}^n$ is the class of smooth projective varieties admitting a $\mathbb{C}^*$-action with isolated fixed points. The fact that such varieties are compactifications of $\mathbb{C}^n$ follows from the decomposition theorem of Bialynicki-Birula \cite[Theorem 4.4]{BB}. All of the known projective compactifications of $\mathbb{C}^n$ with second Betti number equal to one and $n \leq 4$ admit a $\mathbb{C}^*$-action with isolated fixed points.  Moreover, the restricted action of the compact torus of an algebraic torus action on a smooth Fano variety is Hamiltonian with respect to a certain invariant K\"{a}hler form (see Lemma \ref{invariant kahler} for a proof).

Motivated by these facts, we study Hamiltonian torus actions on closed, positive monotone symplectic manifolds $(M,\omega)$. 
Note that a symplectic manifold $(M,\omega)$ can be regarded  as an almost
complex manifold $(M,J)$, where $J$ is an almost complex structure compatible with $\omega$. Since the space of such structures is
contractible, complex invariants of the tangent bundle of $(M,\omega)$, such as Chern classes and Chern numbers, are well-defined.  Let $c_1$ be the first Chern class of the tangent bundle of a symplectic manifold $(M,\omega)$. We say that $(M,\omega)$ is
  \textbf{positive monotone} if there exists a positive constant $\lambda>0$ such that $c_{1}= \lambda [\omega]$.
Note that a smooth Fano variety is automatically a positive monotone symplectic manifold
and so  this condition is the symplectic analogue of the Fano condition in algebraic geometry.

In the past few years, there has been intensive work trying to determine
whether these two categories are indeed distinct. In \cite{FP10}, Fine and Panov gave an example of a 12-dimensional,   non simply connected, positive monotone symplectic manifold
$(\widetilde{M},\widetilde{\omega})$ 
which, therefore, cannot be homotopy equivalent to a Fano variety, as these are always simply connected. In a subsequent paper \cite{FP15}, they conjecture that the existence of a 
Hamiltonian circle action on a 6-dimensional, closed, positive monotone symplectic manifold implies the existence of a diffeomorphism with a Fano 3-fold. There has already been
intense work towards proving that this conjecture holds (see for instance \cite{T,LP19,ChaT,CK23}). 
However, in higher dimensions, the existence of a Hamiltonian circle action is not
enough to ensure the existence of a diffeomorphism with a Fano variety (e.\ g.\ consider the product of the above $\widetilde{M}$ with $S^2$ with a Hamiltonian circle action on the second factor that fixes the first).
Nevertheless, the problem remains open when the circle action has isolated fixed points. (For some special actions of higher dimensional tori see also \cite{CSS}).

In this paper, motivated by Hirzebruch's problem, we consider $8$-dimensional, closed symplectic manifolds $(M,\omega)$ with \textbf{second Betti number equal to one}, admitting a Hamiltonian torus action with isolated fixed points. Observe that such manifolds are automatically positive monotone, as $c_1$ must be a multiple of $[\omega]$, and by \cite{AB, LO, O, GHS}, this multiple must be positive.

If we further assume that $b_{4}(M)=1$, or equivalently that the action has 5 fixed points (which, in dimension 8, is the minimal number of fixed points), then the fixed point representations have been completely classified \cite{GS,JaTo}. They are all isomorphic to the ones induced by a linear action on $\mathbb{CP}^4$. Note that, in this case, the fixed point representations also determine the equivariant cohomology ring and Chern classes of $(M,\omega)$ \cite{T}.

If $b_4(M)=2$, there are three known $8$-dimensional examples satisfying $b_{2}(M)=1$, which may be distinguished by their {\bf index} $k_{0}$ (the largest integer such that $c_1=k_0\eta$, for some nonzero $\eta\in H^2(M;\Z)$, modulo torsion
elements). These are the Fano-Mukai fourfold $V$ (with $k_0=2$), the quintic del Pezzo fourfold $W$ (with $k_0=3$) and the Grassmannian $Q=Gr(2,4)$ (with $k_0=4$). 
These smooth Fano varieties are described in detail in Section \ref{section examples}.
Some of their invariants are summarized in Table \ref{Table intro}.

 \begin{table}[h]
\begin{center}
\begin{tabular}{ |c|c|c|c|c|c|c|c|c|} 
 \hline
$M$ & $(b_{2}(M),b_{4}(M))$ &  $c_{1}^4[M]$ & $c_{1}^2c_{2}[M]$ & $c_{2}^{2}[M]$ & $c_{1}c_{3}[M]$ &  $k_0$ \\ 
 \hline
 $V$ & $(1,2)$ &  288  & 168  & 98 &  48 & 2 \\ 
\hline
 $W$ & $(1,2)$ &  405 & 198  & 97  & 48 & 3  \\ 
 \hline
 $Q$ & $(1,2)$ & 512 & 224 &  98 & 48 & 4  \\ 
 \hline
\end{tabular}

\end{center}
\caption{Some of the invariants of $V,W$ and $Q$.}
\label{Table intro}
\end{table}

Each of these manifolds admits a Hamiltonian torus action which is carefully described in Section \ref{section examples}. 
The corresponding moment map polytopes, enriched with the moment map images of their isotropy 
submanifolds $M^w$, for every isotropy weight $w$ (see Remark \ref{weights isotropy manifold}), are shown in Figure \ref{vfig0}.
These can be thought of as the ``moment map image" of the associated multigraphs defined in Section \ref{multigraph section}.

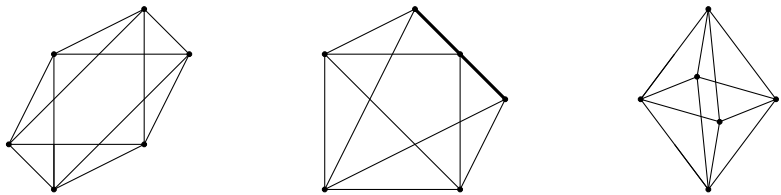
\begin{figure}[h]

 \begin{tikzpicture}[scale = 0.6]

\draw[ black] (-1,-1) -- (-1,-2) -- (1,-1) -- (2,1) -- (1,2) -- (-1,1) -- (-2,-1);

\draw[ black] (-2,-1) -- (1,2);

\draw[ black] (-1,-2) -- (2,1);

\draw[ black] (-2,-1) -- (1,-1);

\draw[ black] (1,-1) -- (1,2);

\draw[ black] (2,1) -- (-1,1);

\draw[ black] (-1,-2) -- (-1,1);

\draw[black] (-1,-2) -- (-2,-1);

\filldraw [black] (-1,-2) circle (1.5pt);

\filldraw [black] (-2,-1) circle (1.5pt);

\filldraw [black] (1,-1) circle (1.5pt);

\filldraw [black] (2,1) circle (1.5pt);

\filldraw [black] (1,2) circle (1.5pt);

\filldraw [black] (-1,1) circle (1.5pt);

\draw (5,1) -- (8,1);

\draw  (5,1) -- (5,-2);

\draw  (5,-2) -- (8,-2);

\draw  (8,-2) -- (8,1);

\draw  (5,-2) -- (7,2);

\draw  (5,-2) -- (9,0);

\draw  (5,1) -- (8,-2);

\draw  (5,1) -- (7,2);

\draw  (8,-2) -- (9,0);

\draw [very thick] (9,0) -- (7,2);


\filldraw [black] (9,0) circle (1.5pt);

\filldraw [black] (7,2) circle (1.5pt);

\filldraw [black] (8,-2) circle (1.5pt);

\filldraw [black] (5,1) circle (1.5pt);

\filldraw [black] (8,1) circle (1.5pt);

\filldraw [black] (5,-2) circle (1.5pt);


\filldraw  (12,0) circle (1.5pt);

\filldraw  (15,0) circle (1.5pt);

\filldraw  (13.75,-0.5) circle (1.5pt);

\filldraw  (13.25,0.5) circle (1.5pt);

\filldraw  (13.5,2) circle (1.5pt);

\filldraw  (13.5,-2) circle (1.5pt);

\draw   (12,0) -- (13.5,2);

\draw  (12,0) -- (13.5,-2);

\draw (15,0) -- (13.5,2);

\draw   (15,0) -- (13.5,-2);

\draw  (13.25,0.5) -- (13.5,2);

\draw   (13.25,0.5) -- (13.5,-2);

\draw   (13.75,-0.5) -- (13.5,2);


\draw  (13.25,0.5) -- (15,0);

\draw  (13.25,0.5) -- (12,0);

\draw   (13.75,-0.5) -- (15,0);

\draw  (13.75,-0.5) -- (12,0);


\draw (13.5,-2) -- (12.75,-1);

\draw  (13.5,-2) --  (13.75,-0.5);




\draw (12,0) -- (12.75,1);






\end{tikzpicture}

\caption{The ``enriched" moment map polytopes corresponding to $V, W$ and $Q$.} \label{vfig0}

\end{figure}

Observe that the first and third examples correspond to Hamiltonian GKM spaces (see Section \ref{Ham GKM spaces}), whereas the
second example does not, as the thick edge is the moment map image of an isotropy submanifold symplectomorphic
to $\C P^2$. Hence, the first and third graphs in Figure~\ref{vfig0} are the GKM graphs associated with the torus action on $V$ and $Q$.

In \cite{GT}, Goldin and Tolman introduce a special basis for the equivariant cohomology ring with $\Z$-coefficients of a Hamiltonian $T$-space
and the elements of this basis are called canonical classes. In particular, the restriction of these classes to the fixed point set, as well as
the associated equivariant structure constants, can be computed explicitly for Hamiltonian GKM spaces admitting an index-increasing component of the
moment map (see Section \ref{section Ham actions}). We use this method to compute the equivariant cohomology ring and Chern classes
of $V$ and $Q$\footnote{Note that the computations for $Q$ are already known, as $Q$ is  the Grassmannian of 
complex planes in $\C^4$ endowed with the standard $3$-dimensional torus action.} (see Theorems \ref{equiv cohomology and chern classes V} and \ref{equiv cohomology and chern classes Q}). 
Even though the torus action on $W$ is not GKM, we can still 
prove the existence of canonical classes and compute its equivariant cohomology ring and Chern
classes (see Theorem \ref{equiv cohomology and chern classes W}).  
As a byproduct,  we obtain the ordinary cohomology rings and Chern classes of $V$ and $W$ (those of $Q$ are well-known).
\begin{cor}\label{cohomology ring V}
Let $V$ be the Fano-Mukai 4-fold of genus 10 and degree 18 described in Section~\ref{subsec V}. Then there exists a basis $\{\zeta_i\}_{i=0,\ldots,4}\cup\{\zeta'_2\}$ of the cohomology ring $H^*(V;\Z)$, with $\zeta_i\in H^{2i}(V;\Z)$,  for   $i=0,\ldots,4$, and  $\zeta'_2\in H^4(V;\Z)$, satisfying the
following properties and relations: 
\begin{align}\label{product V o}
 & \zeta_1^2= 3\zeta_2+3\zeta'_2, \quad  \zeta_1\zeta_3= \zeta_2^2=(\zeta'_2)^2=\zeta_4, \quad   \zeta_1\zeta_2=  \zeta_1\zeta'_2=3\zeta_3,  \quad  \zeta_2\zeta'_2=0\,. 
 \end{align}
 Moreover, the Chern classes of the tangent bundle of $V$ are given by 
 \begin{equation}\label{ord chern classes V}
  c_1=-2 \zeta_1,\quad
  c_2= 7 (\zeta_2+\zeta'_2),\quad
  c_3= - 24 \zeta_3,\quad
  c_4= 6 \zeta_4\,.
 \end{equation}
\end{cor}

 \begin{cor}\label{cohomology ring W}
Let $W$ be the quintic  del Pezzo $4$-fold defined in Subsection \ref{subsec W}. Then there exists a basis $\{\zeta_i\}_{i=0,\ldots,4}\cup\{\zeta'_2\}$ of the cohomology ring $H^*(W;\Z)$,   with $\zeta_i\in H^{2i}(W;\Z)$, for $i=0,\ldots,4$, and  $\zeta'_2\in H^4(W;\Z)$, satisfying the
following properties and relations: 
\begin{align}\label{product W nonequi}
 & \zeta_1^2= 3\zeta_2+2\zeta'_2, \quad \zeta_1\zeta_3 =\zeta_2^2=\zeta_4,\quad   \zeta_1\zeta_2=\zeta_1\zeta'_2=\zeta_3,  \quad (\zeta'_2)^2 = 2\zeta_4,\quad
 \zeta_2\zeta'_2=-\zeta_4. 
 \end{align}
 Moreover, the Chern classes of the tangent bundle of $W$ are given by
 \begin{equation}\label{chern classes W}
 c_1= -3 \zeta_1 ,\quad c_2=13\, \zeta_1 + 9\, \zeta_2^\prime ,\quad c_3= -16 \zeta_3, \quad c_4 = 6\zeta_4.
 \end{equation}
\end{cor}

\subsection{Main results}

We begin by proving several `finiteness' results.
More precisely, in Theorem \ref{thm:ChernNumbers}, we show that, if $(M,\omega)$ is an 8-dimensional, closed, symplectic manifold with $b_2(M)=1$ and index $k_0>1$,
admitting a Hamiltonian circle action with isolated fixed points, then, for each value of $b_4(M)$, there are finitely many possibilities for the Chern numbers of $(M,\omega)$.
As a consequence, for each value of $b_4(M)$, there are finitely many complex cobordism classes of such $(M,\omega)$.
In Theorem \ref{main theorem chern},  we assume that the Euler characteristic is 6 (which is equivalent to the condition $b_2(M)=1$ and $b_4(M)=2$),
and use number theoretical arguments to refine the possibilities for the Chern numbers.
\\$\;$

The Fano manifolds $V$, $W$ and $Q$ are examples of closed, positive monotone symplectic manifolds
of dimension 8, with a Hamiltonian torus action with $6$ fixed points, and so 
 their Euler characteristic is also $6$. Their isotropy data, namely the isotropy weights at each fixed point, are given
respectively in Tables \ref{values c1T V}, \ref{values c1T} and \ref{values c1T Q}.
Our main result says that, in a certain sense, these are
the only possible examples. More precisely, we have the following result.

\begin{main theorem}\label{main theorem}
Let $(M,\omega)$ be an $8$-dimensional, closed, symplectic  manifold admitting a Hamiltonian $T^2$-action with 6 isolated fixed points. Then its index $k_0$ satisfies $2\leq k_0 \leq 4$. Moreover,
the isotropy data, (equivariant) cohomology ring and (equivariant) Chern classes agree with those of
\begin{itemize}
\item[$\bullet$] $V$ with the action described in Section \ref{subsec V}, if $k_0=2$;
\item[$\bullet$] $W$ with the action described in Section \ref{subsec W}, if $k_0=3$;
\item[$\bullet$] $Q$ with a subtorus action of the 3-dimensional torus action described in Section \ref{subsec Q}, if $k_0=4$.

\end{itemize}
\end{main theorem}
Note that, by Lemma \ref{lemma monotone chi six}, $(M,\omega)$ is positive monotone with symplectic form
that can be rescaled to satisfy $c_1=[\omega]$. Moreover, the moment map can be translated so that the weight sum formula \eqref{wsf} holds.
Therefore, our main theorem can be regarded as a classification theorem for positive monotone symplectic manifolds of dimension $8$ and Euler characteristic $6$. 
The weight sum formula also implies that, modulo $GL(2;\Z)$-transformations, the only ``enriched moment map polytopes" are those in Figure \ref{vfig0}.
 
To prove this theorem we use an adaptation of the algorithm introduced in \cite{GS} (see Section \ref{sec:algorithm}). This requires the existence of a combinatorial object (a multigraph)
that describes the action, satisfying a particular positivity condition involving the first equivariant Chern class. This holds, for instance,  for all
positive monotone Hamiltonian
GKM spaces and we prove, in Theorem \ref{existence multigraph dim 8}, that it also holds under the conditions of the main theorem. 

Hence, to prove the main theorem, we run a program based on this algorithm, producing all the possible isotropy data. For $k_0=3$, we
need to use the restrictions on the Chern numbers given by Theorem \ref{main theorem chern}.
Then, it is enough to show that this data is the same as that of $V$, $W$ and $Q$. This is done in Section \ref{sec:class}.

The accompanying software based on the algorithms presented in this paper and all the accompanying files, can be found at 
$$
 \href{https://www.math.tecnico.ulisboa.pt/~lgodin/CircleActions/}{\text{http://www.math.ist.utl.pt/$\sim$lgodin/CircleActions/}}.
$$

Another finiteness result concerns the Betti numbers. Indeed, in Proposition \ref{index betti}, we provide a finite list of the possible Betti numbers of an 8-dimensional, closed, symplectic 
manifold with $k_0>1$, that admits a Hamiltonian circle action  with isolated fixed points, and a multigraph that is
positive w.r.t.\ the first Chern class. More precisely, we obtain
\begin{enumerate}
\item If $k_0=5$ then $b_2=b_4=1$;
\item If $k_0=4$ then $b_2=1$ and $b_4=2$;
\item If $k_0=3$ then $(b_2,b_4)$ is either $(1,2)$ or $(2,3)$;
\item If $k_0=2$ then $1\leq b_2\leq b_4\leq 6$.
\end{enumerate}
Therefore, these restrictions hold for any 8-dimensional, positive monotone Hamiltonian GKM space (see Corollary~\ref{cor GKM spaces betti}). 

All of the above possibilities for the Betti numbers with $k_{0} \geq 3$ are realized by explicit examples. For $k_0=5$, the only possible example is $\C P^4$ (see \cite{GS}).
The cases with $(b_{2},b_{4})=(1,2)$ are discussed in
Section \ref{section examples}. An example with $k_{0}=3$ and  $(b_{2},b_{4})=(2,3)$ is given by the toric Fano $4$-fold $\mathbb{CP}^2 \times \mathbb{CP}^2$. We also note that the bound on $b_{4}$ in the $k_{0}=2$ case is sharp, by considering the toric Fano $4$-fold $(\mathbb{CP}^1)^4$, which has $k_{0}=2$, and $(b_{2},b_{4})=(4,6)$.

Based on these results, we again use the algorithm from \cite{GS} to obtain the following theorem. 
\begin{theorem}\label{thm:add}
 Let $(M,\omega)$ be an 8-dimensional,  closed, symplectic 
manifold with $k_0>1$, that admits a Hamiltonian action of a $2$-dimensional torus with isolated fixed points with a multigraph that is
positive w.r.t.\ the first Chern class.
\begin{itemize}
\item If  $k_0=3$ and $(b_2,b_4) =(2,3)$, then the isotropy data  agrees with that of a subtorus of the standard toric action on $\C P^2 \times \C P^2$. 
\item If $k_0=2$ and $b_2=1$, then $ b_4= 2$  and the isotropy data agrees with that of $V$.
\end{itemize}
\end{theorem}

\smallskip
Motivated by these results, we finish the introduction, with some questions that  we believe  are interesting.
\begin{questionL}
Let $(M,\omega)$ be a closed, positive monotone symplectic manifold, admitting a Hamiltonian action of  a torus with isolated fixed points. 
\begin{enumerate}
\item[i)] Are there finitely many (equivariant) complex cobordism classes  in each dimension?
\item[ii)] Are there finitely many homotopy/homeomorphism/diffeomorphism/symplectomorphism types in each dimension?
\end{enumerate}
\end{questionL}

\section{Preliminaries}

\subsection{Equivariant cohomology: generalities}

We recall that the equivariant cohomology of a topological space $X$ acted on by a torus $T$ of dimension $d$ is given by 
$$
H_T^*(X;R):=H^*(X\times_T ET;R)\,,
$$
where $R$ is a ring with unit and $ET$ is a contractible space on which $T$ acts freely. As $T$ is a torus of dimension $d$, we have that 
$ET$ can be taken to be a product of $d$ copies of $S^\infty\subset \C^\infty$ with the diagonal action of $T\simeq S^1\times \cdots \times S^1$, where each factor acts freely on $S^\infty$. The quotient $BT:=ET/T$
is, therefore, a product of $d$ copies of $\C P^{\infty}$. We remark that $H^*(BT;R)$ is isomorphic to the polynomial ring $R[x_1,\ldots,x_d]$, 
where each $x_i$ has degree 2, and that the constant map $X\to \{pt\}$ gives $H_T^*(X;R)$ the structure of an $R[x_1,\ldots,x_d]$-module.
Moreover, for any subgroup $K$ of $T$, the natural map $M\times_K ET \to M\times_T ET$ sending a $K$-orbit to a $T$-orbit induces a restriction map
$r_K\colon H_T^*(M;R)\to H_K^*(M;R)$ (observe that $ET$ is also a classifying space for $K$). In particular, when $K$ is the trivial group, we obtain 
a restriction map 
$$r\colon H_T^*(M;R)\to H^*(M;R)\,.$$

Let $M^T$ be the set of $T$-fixed  points and assume that $M^T\neq \varnothing$. For $p\in M^T$, consider the (equivariant) inclusion map $\iota_p\colon \{p\} \to M$
and the corresponding restriction $$\iota_p^*\colon H_T^*(M;R)\to H_T^*(\{p\};R)\simeq R[x_1,\ldots,x_d]\,.$$
Given a class $\alpha\in H_T^*(M;R)$, we denote $\iota_p^*(\alpha)$ simply by $\alpha(p)$. 
When $T$ is a circle and $R$ is $\Z$ or $\R$, we have $H_{S^1}^*(\{p\};R)=R[x]$ and so it makes sense to say that a class $\alpha\in H_{S^1}^{2k}(M;R)$ is positive (resp.\ negative). Namely, we say that
\begin{equation}\label{eq:equiv}
\alpha(p)=m\,x^k >0 \,\, \text{(resp. }<0) \quad \text{iff} \quad  m >0\quad \text{(resp.} \;\;m<0 )\,.
\end{equation}

\subsection{Hamiltonian actions and their properties}\label{section Ham actions}

Let $(M,\omega)$ be a closed symplectic manifold of dimension $2n$ that is acted on by a torus $T$ of dimension $d$, preserving the symplectic structure.
Such action is called \textit{Hamiltonian} if it admits a moment map, namely a $T$-invariant map
$\psi\colon M \to \mathfrak{t}^*$ satisfying 
\begin{equation}\label{equation mm}
d\langle \psi, \xi\rangle=-\iota_{\xi^{\#}}\omega, \quad \text{for every }\xi\in \mathfrak{t}\,,
\end{equation}
where $\langle\cdot, \cdot\rangle$ denotes the natural pairing between $\mathfrak{t}^*$ and $\mathfrak{t}$, and $\xi^{\#}$ is the vector field on $M$ corresponding to $\xi\in \frak{t}$. Let us assume that the fixed point set $M^T$ is discrete. 
Consider a $\widetilde{\xi}\in \mathfrak{t}$ such that $\alpha(\widetilde{\xi})\neq 0$ for all $\alpha\in\ell^*$ that occur as weights of
the $T$-action at a fixed point, and the corresponding component of the moment map given by $\varphi:=\langle \psi , \widetilde{\xi}\rangle$, which, for the above properties of $\widetilde{\xi}$, is called \textit{generic}.
We say that a weight $\alpha$ is \textit{positive} (resp. \textit{negative}) if $\alpha(\widetilde{\xi})>0$  (resp. if $\alpha(\widetilde{\xi})<0$). For every $p\in M^T$, let $\lambda_p$ be the number
of negative weights at $p$ and $N_k$ be the number of fixed points with $k$ negative weights.  
 Then it is well known that $\varphi$ is a perfect Morse function
with critical set given by $M^T$. Moreover, 
$H^{2j}(M;\Z)=\Z^{N_j}$  for all $j=0,\ldots,n$ and  $H^j(M;\Z)=0$ for every odd $j$. 
Observe that, since $M$ is compact, $[\omega^j]\in H^{2j}(M;\R)=\R^{N_j}$ is a nontrivial cohomology class for every $j=0,\ldots,n$, and so $N_j>0$ for all $j=0,\ldots,n$.

Once the information on the fixed points is known, computing the cohomology groups is trivial. However, determining the ring structure of
$H^*(M;\Z)$ can be more complicated. Nevertheless, the existence of a Hamiltonian action can help us
understanding the (equivariant) cohomology ring structure (see in particular Section \ref{restrictions and structure constants}).

Let $\Lambda_p^-$ be the product of the negative weights at $p$. 
The following  well-known result by Kirwan \cite{Ki}, provides a basis of $H_T^*(M;\Z)$ related to a generic component $\varphi$ of the moment map. 
This basis depends on an ordering of the fixed points that is induced by $\varphi$.  Namely, we say that 
\begin{equation}\label{ordering fixed points}
p \prec q \quad \text{(resp.   } p\preceq q) \quad \text{iff} \quad \varphi(p)<\varphi(q)\quad\text{(resp.   } \varphi(p)\leq \varphi(q))\,.
\end{equation}

\begin{proposition}[\textbf{Kirwan}]\label{kirwan classes}
Let $(M,\omega)$ be a closed symplectic manifold equipped with a Hamiltonian $T$-action with discrete fixed point set $M^T$
and moment map $\psi\colon M \to \mathfrak{t}^*$. Let $\varphi\colon M \to \R$ be a generic component of the moment map and 
consider the corresponding ordering $\preceq$ on the fixed points. 

Then, for every fixed point $p$, there exists a class $\gamma_p\in H_T^{2\lambda_p}(M;\Z)$, such that 
\begin{enumerate}
\item $\gamma_p(p)=\Lambda_p^-$, and
\item $\gamma_p(q)=0$ for every $q\in M^T\setminus \{p\}$ such that $q\preceq p$\,.
\end{enumerate}
Moreover, for any such classes, the $\{\gamma_p\}_{p\in M^T}$ are a basis for $H_T^*(M;\Z)$ as a module over $H^*(BT;\Z)$
and their restrictions to the ordinary cohomology ring, $\{r(\gamma_p)\}_{p\in M^T}$, are a basis for $H^*(M;\Z)$ as a module over $\Z$.
\end{proposition}
Note that Proposition \ref{kirwan classes} also states that $H_T^*(M;\Z)$ is a free module over $H^*(BT;\Z)$.

A class $\gamma_p$ satisfying (1) and (2) above is called a \textit{Kirwan class at} $p$. Although Kirwan classes always exist, they are not unique. 
For instance, if there exists a fixed point $q\in M^T$ such that $p\prec q$ and $\lambda_p=\lambda_q$, then $\gamma_p+k \gamma_q$ is also a Kirwan class at $p$,
for every $k\in \Z$. 

Following \cite{GT}, we recall another set of classes, called \textit{canonical classes}, that may not exist (see \cite[Example 2.2]{GT}) but, if they do,  are unique
(see \cite[Lemma 2.7]{GT}). 
\begin{definition}\label{canonical class def}
Let $(M,\omega)$ be a closed symplectic manifold equipped with a Hamiltonian $T$-action with discrete fixed point set $M^T$
and moment map $\psi\colon M \to \mathfrak{t}^*$. Let $\varphi\colon M \to \R$ be a generic component of the moment map.

Let $p\in M^T$. Then a class $\tau_p\in H_T^{2\lambda_p}(M;\Z)$ is called a \textbf{canonical class} at $p$ if
\begin{enumerate}
\item $\tau_p(p)=\Lambda_p^-$, and
\item $\tau_p(q)=0$ for every $q\in M^T\setminus \{p\}$ such that $\lambda_q \leq \lambda_p$\,.
\end{enumerate}
\end{definition} 
In \cite[Lemma 2.8]{GT} it is proved that, if a canonical class $\tau_p$ exists, then it is also a Kirwan class. Therefore, if a canonical class $\tau_p$ exists for every fixed point $p$,
the set $\{\tau_p\}_{p\in M^T}$ is a basis of $H_T^*(M;\Z)$ as a module over $H^*(BT;\Z)$
and the restrictions   $\{r(\tau_p)\}_{p\in M^T}$ of these classes to the  ordinary cohomology ring,  give a basis for $H^*(M;\Z)$ as a module over $\Z$.

\begin{remark}\label{trivial cc}
There are two canonical classes that always exist: those corresponding to the minimum $p_0$ and to the maximum $p_{\text{max}}$ of $\varphi$. 
In this case, it follows from Definition \ref{canonical class def}, that $\tau_{p_0}(p)=1$ for every $p\in M^T$, and $\tau_{p_{\text{max}}}(p)=0$ for all $p\in M^T\setminus \{p_{\text{max}}\}$.  
\end{remark}

\subsubsection{Restrictions to the fixed points and equivariant structure constants}\label{restrictions and structure constants}

Let $(M,\omega)$ be a closed symplectic manifold equipped with a Hamiltonian $T$-action with a discrete fixed point set $M^T$
and moment map $\psi\colon M \to \mathfrak{t}^*$. 
Consider a basis of $H_T^*(M;\Z)$ as a module over $H^*(BT;\Z)$, for instance the basis $\{\gamma_p\}_{p\in M^T}$ given by the Kirwan classes 
with respect to a generic component $\varphi\colon M \to \R$ of the moment map
or, when they exist,
that of canonical classes. The \textit{equivariant structure constants} with respect to $\{\gamma_p\}_{p\in M^T}$ are defined as the elements $a^s_{p,q}\in H^*(BT;\Z)$
such that 
\begin{equation}\label{e structure constants}
\gamma_p\gamma_q=\sum_{s\in M^T}a^s_{p,q}\gamma_s\,.
\end{equation}
If we identify the dual lattice of $\mathfrak{t}$ with $\Z\langle x_1,\ldots,x_d\rangle$, then $H^*(BT;\Z)=\Z[x_1,\ldots,x_d]$ and, therefore, $a^s_{p,q}$ is a polynomial of degree $\lambda_p+\lambda_q-\lambda_s$.
In order to obtain the \textit{ordinary structure constants} with respect to the basis $\{r(\gamma_p)\}_{p\in M^T}$, i.e. the constants $b_{p,q}^s\in \Z$ such that
\begin{equation}\label{structure constants}
r(\gamma_p)r(\gamma_q)=\sum_{s\in M^T}b^s_{p,q}r(\gamma_s)\,,
\end{equation}
it is enough to take $b^s_{p,q}=r(a^s_{p,q})$,  the constant term of the polynomial $a^s_{p,q}$. 

Computing the equivariant structure constants with respect to a basis $\{\gamma_p\}_{p\in M^T}$ is indeed equivalent to computing the restriction of each $\gamma_p$ to the fixed point set, and this problem has been
extensively studied in the literature (see for instance \cite{GuZa, GT, ST}).  
In order to show that the set of polynomials given by $\gamma_p(p')$, for all $p,p'\in M^T$, determines the equivariant structure constants, it is enough to prove
that the coefficients of every equivariant cohomology class $\alpha$, with respect to the basis $\{\gamma_p\}_{p\in M^T}$,  are determined by the restrictions $\gamma_p(p')$ and $\alpha(p')$, for all $p,p'\in M^T$.
This is illustrated in the following result, whose proof is left to the reader (see \cite{GuZa} and \cite[Lemma 2.4]{GS}).
\begin{lemma}\label{from restrictions to constants}
Let $(M,\omega)$ be a closed symplectic manifold equipped with a Hamiltonian $T$-action with discrete fixed point set $M^T$
and moment map $\psi\colon M \to \mathfrak{t}^*$. Let $\varphi\colon M \to \R$  be a generic component of the moment map and let  $p_0,p_1,\ldots, p_m$ be the fixed points ordered so that 
$$
p_0\prec p_1\preceq p_2 \preceq \cdots \preceq p_{m-1}\prec p_m\,.
$$
Consider a basis of $H^*_T(M;\Z)$ given by Kirwan classes $\{\gamma_i\}_{i=0}^m$ and let $\alpha=\sum_{i=0}^m \alpha^i\gamma_i$, where $\alpha^i\in H^*(BT;\Z)$. 
Then the coefficients $\alpha^i$ can be computed from the restriction of $\alpha$ to the fixed points. More precisely, they can be computed recursively as
$$
\alpha^i=\frac{\alpha(p_i)-\sum_{h\colon \varphi(p_h)<\varphi(p_i)}\alpha^h\gamma_h(p_i)}{\Lambda_i^-}\,.
$$
\end{lemma}
A set of equivariant cohomology classes, for which the restrictions to the fixed points can be easily computed  from the isotropy data, is that of equivariant Chern classes. 
The next result is a well known fact (see, for example, \cite[Section 8.6]{GuS}).
\begin{lemma}\label{equiv chern classes}
Let $(M,J)$ be a closed almost complex manifold acted on by a torus $T$ preserving $J$ and with discrete fixed point set $M^T$.
Let $w_1,\ldots,w_n\in \ell^*$ be the isotropy weights of the $T$-action at a fixed point $p$, where $\ell^*$ denotes the dual lattice in $\mathfrak{t}^*$. 
Consider the equivariant Chern classes $c_j^T\in H_T^{2j}(M;\Z)$ of the tangent bundle of $(M,J)$. Then for every $p\in M^T$
$$
c_j^T(p)=\sigma_j(w_1,\ldots,w_n)\,,
$$
where $\sigma_j(x_1,\ldots,x_n)$ is the $j$-th elementary symmetric polynomial.
\end{lemma}

\begin{remark}\label{remark chern classes}
From Lemmas \ref{from restrictions to constants} and \ref{equiv chern classes} it follows that, given a basis of Kirwan classes and 
their restrictions to the fixed point set, the coefficients of the equivariant Chern classes with respect to this basis can be easily computed from the isotropy data
at the fixed points.
\end{remark}

\begin{remark}\label{our conventions cc}(\textbf{Convention on canonical classes})
Let $(M,\omega)$ be a closed symplectic manifold equipped with a Hamiltonian $T$-action with discrete fixed point set $M^T$ and
moment map $\psi\colon M \to \mathfrak{t}^*$. Assume that for every $p\in M^T$ a canonical class $\tau_p$ exists.
In this paper we consider a slight modification of the canonical classes, by taking those given by $(-1)^{\lambda_p}\tau_p$. 
We still denote this ``new canonical class'' at $p$ by $\tau_p$. 
Note that, with this convention, $\tau_p(p)$ is a product of positive weights.
\end{remark}

\subsection{Hamiltonian GKM spaces}\label{Ham GKM spaces}
One special subclass of closed symplectic manifolds acted on by a torus in a Hamiltonian way is that of \textbf{Hamiltonian GKM spaces}. 
A Hamiltonian action of a torus $T$ of dimension greater or equal to 2 is said to be GKM (after Goresky-Kottwitz-MacPherson \cite{GKM}) if the fixed points are isolated and, for each fixed point $p$, the weights of the isotropy action at $p$
are pairwise linearly independent. This implies that, for every isotropy weight $\alpha\in \ell^*$, the manifold fixed by the subgroup 
$T_\alpha:=\exp\{\text{Ann}(\alpha)\}$ is a 2-sphere, called an \emph{isotropy sphere}. Each sphere is symplectic
and carries  an effective action of  the quotient circle $T/T_\alpha$,
with two fixed points, $p$ and $q$. Observe that, if $\alpha$ is the weight of the isotropy $T$-action on $S^2$ at $p$, then $-\alpha$ is the weight of
the isotropy $T$-action on $S^2$ at $q$.

One can assign a labeled graph $\Gamma=(V,E)$ to each (Hamiltonian) GKM space, where the vertices are in bijection with the fixed points, and the edges
with the isotropy spheres. If one chooses a direction on each edge, for instance given by a generic component of the moment map, then the label $\eta(p,q)$ of the
edge $e=(p,q)$ (the edge connecting the fixed points $p$ and $q$  on the corresponding $S^2$) is given by the corresponding weight of the $T$-action
at $p$. Since we are assuming that this action is Hamiltonian with moment map $\psi\colon M \to \mathfrak{t}^*$,
 it is natural to see the GKM graph in $\mathfrak{t}^*$, mapping each vertex $v$, corresponding to the fixed point $p$, to $\psi(p)$, and each isotropy sphere to its image
 under the moment map $\psi$. Examples of such graphs are given in Figures \ref{omegaweights}, \ref{vfig1} and \ref{gkm Q} (see the corresponding sections for a
 description of the manifolds and the actions). 

Suppose that there exists a generic component $\varphi$ of the moment map $\psi$ that is \emph{index-increasing}, i.e.
$$p\prec q \implies \lambda_p<\lambda_q \quad \text{for all the edges }e=(p,q) \text{ of the GKM graph.}$$
Then Goldin and Tolman \cite{GT} give a recipe to compute the equivariant cohomology ring with $\Z$ coefficients. This method will be used in Sections
\ref{subsec V} and \ref{subsec Q} to compute the equivariant cohomology rings of the corresponding manifolds.  

\subsection{Monotone Hamiltonian $T$-spaces}\label{monotone Ham spaces}
Let $(M,\omega)$ be a closed symplectic manifold of dimension $2n$ endowed with a Hamiltonian action of a torus $T$. 
We say that $(M,\omega)$ is \textbf{monotone} if there exists $\lambda \in \R$ such that $c_1=\lambda [\omega]$. 
Since the action is Hamiltonian, it is known that $\lambda$ must be positive (see for instance \cite{AB,LO,O,GHS, CSS}).
We can therefore rescale the symplectic form so that $c_1=[\omega]$ holds. 
 
At the level of equivariant cohomology classes, the above equation implies that there exists $a\in \mathfrak{t}^*$ such that
$c_1^T=[\omega-\psi] + a$. Since the moment map is defined up to a constant, by replacing the moment map $\psi$ with the moment
map $\psi-a$, we obtain 
\begin{equation}\label{eq:eqc1}
c_1^T=[\omega-\psi]\,.
\end{equation}
Hence, restricting \eqref{eq:eqc1} to a fixed point $p$, we obtain the so called \textbf{weight sum formula}
\begin{equation}\label{wsf}
\sum_{i=1}^n w_i^p=-\psi(p)\quad \text{for all fixed points  }p\in M^T,
\end{equation}
where $w_1^p,\ldots,w_n^p$ are the weights of the isotropy $T$-action at $p\in M^T$. (See \cite[Section 3.1]{CSS}.)

 \section{The Chern Numbers}

In this section we study the possible Chern numbers of an 8-dimensional, closed, symplectic manifold that can be
endowed with a Hamiltonian circle action and isolated fixed points. 

The \textbf{index} of $(M,\omega)$, which  will play an important role, is
defined as follows. Given an almost complex manifold $(M,J)$, the index of $(M,J)$ is the largest integer such that, modulo torsion,
$c_1=k_0\,\eta$ for some nonzero $\eta\in H^2(M;\Z)$. We then define the index $k_0$ of  a symplectic manifold $(M,\omega)$ by regarding it as an almost
complex manifold $(M,J)$, where $J$ is an almost complex structure compatible with $\omega$. 

If $(M,\omega)$ can be endowed with a Hamiltonian $S^1$-action with isolated fixed points, then its index $k_0$ coincides with the minimal Chern
number\footnote{The minimal Chern number of a closed, symplectic manifold $(M,\omega)$,  with first Chern class $c_1$ which is not torsion, is the non-negative integer $D$ such that $\langle c_1, \pi_2 (M)\rangle =
D \mathbb{Z}$.} and satisfies 
\begin{equation}\label{inequalities index}
1\leq k_0 \leq n+1
\end{equation}
(see \cite[Corollary 1.1]{S}).

The problem of determining which numbers can arise as  Chern numbers of a closed, symplectic manifold
acted on  by a torus in a Hamiltonian way is wide open. Nevertheless, some information  can be recovered from the Todd genus
and the index of $(M,\omega)$ (for an extensive discussion see, for instance, \cite{S}). We recall that the
Todd genus of $(M,\omega)$, simply denoted by $\text{Td}(M)$, is the genus associated to the power series of $\displaystyle\frac{x}{1-e^{-x}}$ and
can be computed in terms of the Chern numbers. For example, in dimension 8 (which is the dimension we are mostly interested in), it is given by

\begin{equation}\label{todd dim 8}
\text{Td}(M)=\frac{-c_1^4+4c_1^2c_2+3c_2^2+c_1c_3-c_4}{720}[M]\,.
\end{equation}
If $(M,\omega)$ is acted on by a circle in a Hamiltonian way, then it is well-known that $\text{Td}(M)=\text{Td}(M_{\text{min}})$,
where $M_{\text{min}}$ denotes the submanifold  on which the moment map attains its minimum (see for instance \cite[Section 5.7]{HBJ}). If the fixed points are
isolated, one  immediately obtains that $\text{Td}(M)=1$.

Let us now consider an 8-dimensional, closed, symplectic manifold that admits a Hamiltonian circle action and isolated fixed points. 
In this case, two of the Chern numbers, namely
$c_4[M]$ and $c_1c_3[M]$, can be expressed as linear combinations of the even Betti numbers (see \cite[Section 3]{GS} and in particular \cite[Corollary 3.1]{GS}). 

\begin{lemma}[Godinho--Sabatini, \cite{GS}]
\label{c4 and c1c3}
Let $(M,\omega)$ be an  $8$-dimensional, closed, symplectic manifold admitting a Hamiltonian $S^1$-action with isolated fixed points. Let $b_{i}(M)$
be the Betti numbers of $M$. Then 
\begin{equation}\label{c1c3c4}
c_4[M]=2+2b_2(M)+b_4(M)\quad\text{and}\quad c_1c_3[M]=44+8b_2(M)-2b_4(M)\,.
\end{equation}
\end{lemma}
Before stating our next results, we also recall the following equations for the other Chern numbers, that are a direct consequence of \cite[Propositions 6.3, 6.4]{S} applied to closed, connected symplectic manifolds that admit a Hamiltonian $S^1$-action with isolated fixed points.

\begin{proposition}[Sabatini, \cite{S}]
\label{chern numbers 8}
Let $(M,\omega)$ be an $8$-dimensional, closed, symplectic, manifold with index $k_0$, admitting a Hamiltonian $S^1$-action with isolated fixed points. 

$\bullet$ If $k_0=5$, then 
\begin{equation}\label{k0=5}
c_1^4[M]=625,\;\; c_1^2c_2[M]=250,\;\;c_2^2[M]=101-2b_2(M)+b_4(M)\,.
\end{equation}

$\bullet$ If $k_0=4$, then 
\begin{equation}\label{k0=4}
c_1^4[M]=512,\;\; c_1^2c_2[M]=224,\;\;c_2^2[M]=98-2b_2(M)+b_4(M)\,.
\end{equation}

$\bullet$ If $k_0=3$, then 
\begin{equation}\label{k0=3}
c_1^2c_2[M]=108+\frac{2}{9}c_1^4[M], \;\; c_2^2[M]=82-2b_2(M)+b_4(M)+\frac{1}{27}c_1^4[M]\,.
\end{equation}

$\bullet$ If $k_0=2$, then 
\begin{equation}\label{k0=2}
c_1^2c_2[M]=96+\frac{1}{4}c_1^4[M],\;\; c_2^2[M]=98-2b_2(M)+b_4(M)\,.
\end{equation}

\end{proposition} 

Let $M$ be a closed, orientable manifold of dimension $4k$, for some $k\in \Z_{>0}$, with chosen orientation class $[M]\in H_{4k}(M;\Z)$ and consider the intersection form 
\begin{equation}\label{eq:intform}
 \langle \cdot , \cdot \rangle \colon H^{2k}(M;\Z)  \times H^{2k}(M;\Z)  \to \Z 
 \end{equation}
 $$
 (\alpha, \beta)  \mapsto \langle \alpha, \beta \rangle  :=(\alpha \cup \beta)[M].
$$
If we use $\R$-coefficients, \eqref{eq:intform} is a bilinear, symmetric form on $H^{2k}(M;\R)$. Let $b_{2k}^+(M)$ (resp.\ $b_{2k}^-(M)$) be the dimension of the maximal subspace on which $\langle \cdot , \cdot \rangle$ is positive (resp.\ negative) definite. 
The signature $\sigma(M)$ of $M$ is the signature of $\langle \cdot , \cdot \rangle$, namely $\sigma(M)=b_{2k}^+(M)-b_{2k}^-(M)$. 
If the manifold is symplectic and admits a Hamiltonian circle action with isolated fixed points,
the signature of $M$ can be computed from its Betti numbers \cite{JR}. The next result, originally proved in \cite{L}, uses this fact to 
obtain a necessary and sufficient condition for  $\langle \cdot , \cdot \rangle$ to be positive definite in dimension $8$.

\begin{proposition} \label{b2onepositive}
Let $(M,\omega)$ be an $8$-dimensional,  closed, symplectic manifold admitting a Hamiltonian $S^1$-action with isolated fixed points. 
Then the intersection form $\langle \cdot , \cdot \rangle$ of $M$ is positive definite if and only if  $b_{2}(M)=1$. 
\end{proposition}
\begin{proof}
By \cite[Theorem 1.1]{JR} and Poincar\'{e} duality, we have 
\begin{equation}\label{eq:sign}
\sigma(M) = 2 - 2b_{2}(M)+ b_{4}(M).
\end{equation}
The condition that the bilinear form is positive definite is equivalent to $\sigma(M)=b_{4}(M)$.  To see this, write $b_{4}(M)$  as $b_{4}^{+}(M) + b_{4}^{-}(M)$ and $\sigma(M)$ as  $b_{4}^{+}(M) - b_{4}^{-}(M)$. This immediately implies that  $\langle \cdot, \cdot \rangle$ is positive definite (i.e. $b_{4}^{-}(M) = 0$)  if and only if  $b_{2}(M)=1$.
\end{proof}

Consider an $8$-dimensional,  closed, symplectic manifold admitting a Hamiltonian $S^1$-action with isolated fixed points. From Proposition \ref{chern numbers 8}, it is clear that,
if the index $k_0$ is either 4 or 5, then the Betti numbers $b_2(M)$ and $b_4(M)$ determine all the Chern numbers.
The cases in which $k_0$ is equal to 2 or 3 are more complicated. In the following, we prove that, once we assume
$b_2(M)=1$ and $k_0>1$, then there are only finitely many possible Chern numbers for each value of $b_4(M)$ and give an explicit range in which they may vary. The key role is played by the fact that, for $b_2(M)=1$, the intersection form on $H^4(M,\Z)$
is positive definite (see Proposition \ref{b2onepositive}).

\begin{theorem}\label{thm:ChernNumbers} 
Let $(M,\omega)$ be an  $8$-dimensional,  closed, symplectic manifold admitting a Hamiltonian $S^1$-action with isolated fixed points. Assume that the index $k_0$ is not one and that $b_2(M)=1$.  
Then, for each value of $b_4:=b_4(M)$, there are finitely many possibilities for the Chern numbers of $(M,\omega)$.

More precisely we have
\begin{equation}\label{c4 c1c3 1}
c_4[M]=4+b_4\,,\quad c_1c_3[M]=52-2b_4\,,
\end{equation}
and for the Chern numbers $c_1^4[M]$, $c_1^2c_2[M]$ and $c_2^2[M]$ we have:
\begin{itemize}
\item If $k_0=4,5$, then they are determined by the Betti numbers and  are given by \eqref{k0=5} and \eqref{k0=4}.
\item If $k_0=3$, then
$$c_2^2[M]=80 + b_4 + \frac{1}{27} c_1^4[M], \quad c_1^2c_2[M]= 108 + \frac{2}{9}c_1^4[M]$$
and
\begin{equation}\label{eq:c1}
\frac{c_1^4[M]}{81}\in \left[ \frac{b_4+32- \sqrt{b_4^2+64\,b_4+448}}{2}, \,\frac{b_4+32+\sqrt{b_4^2+64\,b_4+448}}{2}\right] \cap \Z.
\end{equation}
\item If $k_0=2$, then $$c_2^2[M]=96 + b_4, \quad c_1^2c_2[M]= 96 + \frac{1}{4}c_1^4[M]$$
and
$$
\frac{c_1^4[M]}{16}\in \left[ \frac{(\sqrt{96+b_4} - \sqrt{b_4})^2}{4}, \,\frac{(\sqrt{96+b_4} + \sqrt{b_4})^2}{4}\right] \cap \Z.
$$ 
\end{itemize}

\end{theorem} 
\begin{proof}
Equation \eqref{c4 c1c3 1} is a direct consequence of Lemma \ref{c4 and c1c3}.
Then, since $k_0$ is bounded above by 5 (see \eqref{inequalities index}) and since,  for $k_0=4,5$, the claim follows as a special case of \eqref{k0=5} and \eqref{k0=4}, we only have to consider the cases $k_0=2$ and $3$.

Let $b_4=N$ and let us fix an integral basis of $H^{4}(M,\mathbb{Z}) = \mathbb{Z}^N$.   Let  $X = (x_{1}, \ldots, x_{N})$ and $Y = (y_{1}, \ldots, y_{N} )\in \mathbb{Z}^N$ be the coordinates of  $c_{1}^2[M]$ and  $c_{2}[M]$ with respect to this basis. Since $b_2(M)=1$, we have from Proposition \ref{b2onepositive} that the intersection from on $H^{4}(M,\mathbb{Z})$ is a positive definite inner product. Let $\lVert \cdot\rVert$ be the associated norm. Then we have 
$$\lVert X \rVert^2=c_1^4[M], \quad \langle X , Y \rangle = c_{1}^2 c_2[M] \quad \text{and} \quad \lVert Y\rVert^2 = c_{2}^2[M].$$ 
Using the formula for $c_{1}c_{3}[M]$ in \eqref{c1c3c4} and the fact that $c_{4}[M]$ is the Euler characteristic of $M$, the equation $\text{Td}(M)=1$  and \eqref{todd dim 8}
give
\begin{equation}\label{Todd} 
720 = -\lVert X\rVert^2 + 4 \langle X , Y\rangle + 3 \lVert Y\rVert ^2 + 48 - 3 N.
 \end{equation}

$\bullet$ $k_{0} = 3$: From \eqref{k0=3}, we have
\begin{equation}\label{eq:Silvia2}
\lVert Y\rVert ^2 = 80 + N + \frac{\lVert X\rVert^2}{27} .
\end{equation}
Then, substituting \eqref{eq:Silvia2} into \eqref{Todd}, gives
\begin{equation}\label{eq: Silvia 3}
 9  \langle X ,  Y \rangle = 2 \lVert X\rVert ^2 +972. 
 \end{equation}
Applying the Cauchy-Schwarz inequality we obtain
$$  2 \lVert X \rVert^2 +972 \leq 9 \lVert X\rVert \lVert Y\rVert.$$ 
Squaring both sides of the inequality, we get
$$ 4\lVert X\rVert^4 + 3888 \lVert X\rVert^2  +944784\,  \leq 81 \, \lVert X\rVert^2 \left( 80 + N + \frac{\lVert X\rVert ^2}{27}\right),$$
and so
$$\lVert X\rVert ^4 - \left(81N + 2592\right) \lVert X\rVert^2 + 944784 \leq 0.$$
Hence, for each value of $N$, there are finitely many possibilities for $\lVert X\rVert^2=c_1^4[M]$ and \eqref{eq:c1} holds. Consequently, by \eqref{eq:Silvia2} and  \eqref{eq: Silvia 3}, there are also finitely many possibilities  for $\langle X,Y\rangle=c_1^2c_2[M]$ and $\lVert Y\rVert^2=c_2^2[M]$.

$\bullet$ $k_{0}=2$:
From  \cite[Corollary 6.1]{S}, we have
\begin{equation}\label{eq:Silvia}\lVert Y\rVert ^2 = 96 + N .
\end{equation}  
Substituting  \eqref{eq:Silvia} into  \eqref{Todd},  gives
\begin{equation}\label{eq:c_1^2c2}
4  \langle X,Y \rangle= 384 + \lVert X\rVert ^2,
\end{equation}
and so, from the Cauchy-Schwarz inequality, we get
$$ \lVert X \rVert^2- 4  \lVert X\rVert \lVert Y\rVert   +384 \leq 0.$$

Solving the quadratic inequality and substituting the expression for $\lVert Y\rVert$ in \eqref{eq:Silvia}, we obtain that 
\begin{equation}\label{eq:ineq}
\lVert X \rVert  \in [ 2\,(\sqrt{96+N} - \sqrt{N}), \,2\,(\sqrt{96+N} + \sqrt{N})],
\end{equation}
and so, for each value of $N$, there are finitely many possibilities for $\lVert X\rVert^2$, and consequently, by \eqref{eq:Silvia} and \eqref{eq:c_1^2c2}, for $\langle X,Y\rangle$ and $\lVert Y\rVert^2$. The result then follows.

\end{proof}
As a direct consequence of Theorem \ref{thm:ChernNumbers} we have the following result.
\begin{cor}\label{cor 1 chern}
For each $N\in \Z_{> 0}$, there is a finite number of complex cobordism classes of 8-dimensional, closed, symplectic manifolds that
admit a Hamiltonian $S^1$-action with isolated fixed points, with index $k_0$ not equal to one, $b_2(M)=1$ and $b_4(M)=N$.
\end{cor}
Observe that any  manifold as in Corollary~\ref{cor 1 chern} is automatically positive monotone. Indeed, $b_2(M)=1$ forces $c_1=\lambda [\omega]$ for some $\lambda\in \R$, and the existence of a Hamiltonian
action forces $\lambda$ to be positive (see \cite[Lemma 5.2]{GHS}). 
Therefore, this corollary goes in the direction of proving that 
there are finitely many complex cobordism classes of a positive monotone symplectic manifold admitting a Hamiltonian group action. 

\subsection{The case $\chi(M)=6$}

 In this subsection we analyze the special case in which the Euler characteristic $\chi(M)$ is equal to $6$, which will be further investigated in the next sections.
 We begin with the following result.
 
\begin{proposition}\label{sixfp}

Let $(M,\omega)$ be a closed, $8$-dimensional, symplectic manifold admitting a Hamiltonian $S^1$-action with isolated fixed points. Then the following statements hold:
\begin{enumerate}
\item $\chi(M)=6$ if and only if $b_{2}(M)=1$ and $b_{4}(M)=2$.
\item If $\chi(M)=6$,
then there exists a basis $v_1,v_2$ of $H^4(M;\Z)$ in which the intersection form $\langle \cdot , \cdot \rangle:H^4(M;\Z)\times H^4(M;\Z) \to \Z$ is the standard Euclidean form, i.e.
\begin{equation}\label{bilinear v1v2}
\langle v_i, v_j \rangle = \delta_i^{\,j},\quad\text{for all }i,j=1,2\,,
\end{equation}  
where $\delta_i^{\, j}$ denotes the Kronecker delta.
\end{enumerate} 
\end{proposition}
\begin{proof}

By Lemma \ref{c4 and c1c3}, we have $\chi(M)=c_4[M]=2+2b_2(M)+b_4(M)$, which, together with the fact that $b_{2i}(M)\geq 1$ for all $i=0,\ldots,4$, is equal to $6$ if and only if $b_2(M)=1$ and $b_4(M)=2$. This proves (1).



By Proposition \ref{b2onepositive} and (1) we know that the intersection form is positive definite.  Since there is exactly one isomorphism class of symmetric, unimodular, positive definite, integral bilinear forms of rank $2$ (the class of the Euclidean inner product on $\mathbb{Z}^2$) \cite[Theorem 2.2]{MH}, the  result follows. 
\end{proof}

The next theorem is a refinement of Theorem \ref{thm:ChernNumbers} to the case in which $\chi(M)=6$. Using Proposition \ref{sixfp} (2) and number theoretical arguments,
we are able to restrict further the possible values of the Chern numbers of $(M,\omega)$, when $k_0>1$. 

\begin{theorem}\label{main theorem chern}  
Let $(M,\omega)$ be an $8$-dimensional,  closed, symplectic manifold with $\chi(M)=6$, admitting a Hamiltonian $S^1$-action with isolated fixed points. Then the index $k_0$
satisfies $1\leq k_0 \leq 4$. If  $k_0\neq 1$, then the possible Chern numbers are listed in Tables \ref{spin} and \ref{non spin}.    
More precisely,\\$\;$\\
$\bullet$ If $k_0=2$ or 4, then the possible Chern numbers are listed in Table~\ref{spin}. In particular, when $k_0=4$, the Chern numbers are those listed in  Case 1.\\

\begin{table}[h]
\begin{center}
\begin{tabular}{|c||c|c|c|c|c|}
\hline
  &  $c_1^4[M]$ & $c_1^2c_2[M]$  & $c_2^2[M]$ & $c_1c_3[M]$ & $c_4[M]$\\
  \hline \hline 
  Case 1 &  512 &  224 & 98 & 48 & 6 \\
  \hline
  Case 2 & 288  &  168 & 98 & 48 & 6\\
  \hline
  Case 3 & 400  &  196 & 98 & 48 & 6 \\
  \hline
  \end{tabular}
 \end{center}
 \caption{
 Possible Chern numbers of $(M,\omega)$ if $k_0=2,4$.}
 \label{spin}
 \end{table}

$\bullet$ If $k_0=3$, then the possible Chern numbers are listed in Table~\ref{non spin}.

\begin{table}[h!]
\begin{center}
\begin{tabular}{|c||c|c|c|c|c|}
\hline
  &  $c_1^4[M]$ & $c_1^2c_2[M]$  & $c_2^2[M]$ & $c_1c_3[M]$ & $c_4[M]$\\
  \hline \hline 
  Case 1 & 405  & 198  & 97 & 48 & 6 \\
  \hline
  Case 2 & 648  & 252 & 106 & 48 & 6\\
  \hline
  Case 3 &  729 & 270 & 109 & 48 & 6 \\
  \hline
   Case 4 &  1296 & 396  & 130 & 48 & 6 \\
  \hline
  Case 5 &  1458 & 432 & 136 & 48 & 6 \\
  \hline
  Case 6 &  2025 & 558 & 157 & 48 & 6 \\
  \hline
  Case 7 &  2106 & 576 & 160 & 48 & 6 \\
  \hline
  Case 8 & 2349  & 630 & 169 & 48 & 6 \\
  \hline
  \end{tabular}
\end{center}
 \caption{
Possible Chern numbers of $(M,\omega)$ if $k_0=3$.}
 \label{non spin}
   \end{table}
\end{theorem}
\begin{proof}
First of all we observe that, by Proposition \ref{sixfp} (1) and Lemma \ref{c4 and c1c3}, the Chern numbers $c_1c_3[M]$ and $c_4[M]$ are always respectively equal to $48$ and $6$. 
Since, by \cite[Corollary 1.1]{S}, the index $k_0$ of $(M,\omega)$ satisfies $1\leq k_0\leq 5$, to prove the first statement we need to prove that $k_0$
cannot be 5. This follows immediately from the fact that
$k_0$ should always divide the Chern number $c_1c_3[M]$, which, in this case, is 48. Therefore $1\leq k_0 \leq 4$. 
\\$\;$\\
$\bullet$ \emph{Suppose that $k_0=2$ or 4.} 
Then, by Proposition \ref{chern numbers 8}, we have that $c_2^2[M]=98-2b_2(M)+b_4(M)$, which, together with Proposition \ref{sixfp} (1), implies that $c_2^2[M]=98$. Hence we just have to determine the possible values of $c_1^4[M]$ and $c_1^2c_2[M]$.

If $k_0=4$, then  \eqref{k0=4} and Proposition~\ref{sixfp} (1) imply that the Chern numbers are exactly those of Case 1 in Table \ref{spin}. 

If $k_0=2$, by  \eqref{k0=2} we have
\begin{equation}\label{eq chern}
c_{1}^4[M] + 384 = 4c_{1}^2 c_{2}[M].
\end{equation}
Let $v_1, v_2$ be the basis of $H^4(M;\Z)$ given by Proposition \ref{sixfp} (2). Then,
writing 
$$c_{2} = y_1v_1+y_2v_2$$ 
with $y_1,y_2\in \Z$, we have $y_{1}^2+y_{2}^2 = 98$. Writing $98$ as a sum of squares, it is easy to check that this implies $y_{1}= \pm 7$ and $y_{2} = \pm 7$.

Since the index is 2, we have that $c_{1} = 2 u$, for some (primitive element) $u \in H^{2}(M,\mathbb{Z})$. Hence $c_{1}^2 = 4 u^2$ and writing $u^{2} = k_{1}v_{1} + k_{2}v_{2}$ with $k_1,k_2\in \Z$, gives $c_{1}^2 = 4 k_{1} v_{1} + 4 k_{2} v_{2}$ and $c_{1}^4[M]= 16(k_{1}^2+k_{2}^2)$. Using 
\eqref{eq chern}, we obtain
$$ 16(k_{1}^2+k_{2}^2) + 384 = 16( y_{1} k_{1}   +  y_{2}  k_{2}),$$
and so, by Lemma \ref{numlem}, we have $(|k_{1}|,|k_{2}|) = (3,3), (3,4), (4,3)$ or $(4,4)$, implying that $c_{1}^4[M]$ is either $288$, $400$ or $512$. The corresponding values of $c_{1}^2c_{2}[M]$ can be obtained using \eqref{eq chern}. $\;$\\

$\bullet$ \emph{Suppose that $k_0=3$.} 
We will first prove that there exists $(a,b,c,d) \in \mathbb{Z}^4$ such that:
\begin{equation}\label{chern index 3}
c_{1}^4[M] = 81(c^2+d^2), \quad c_{1}^2c_{2}[M] = 9(ac+bd) \quad \text{and}\quad c_{2}^{2}[M] = a^2+b^2,
\end{equation}
where $(a,b,c,d)$ satisfy
\begin{enumerate}
\item[(i)] $a^2+b^2 - (82 + 3c^2+ 3d^2)=0$ \
\item[(ii)] $ac+bd = 12 + 2(c^2+d^2)$.
\end{enumerate}
We begin by expressing $c_{2}$ and $c_{1}^2/9$ in terms of the integral basis given by Proposition \ref{sixfp} (2). Let $a,b\in \Z$ be such that $c_{2} = av_1+bv_2$. Since $k_0=3$, we have $c_{1}= 3 u$, for some $u \in H^{2}(M,\mathbb{Z})$ and $c_{1}^2 = 9 u^2$. Writing $u^{2} = cv_{1} + dv_{2}$ with $c,d\in \Z$, gives $c_{1}^2 = 9 c v_{1} + 9 d v_{2}$.

On the other hand, by Proposition \ref{sixfp} (1) and  \eqref{k0=3}, we have  
$$c_2^2[M]=82 + \frac{c_1^4[M]}{27}\quad\text{and}\quad c_1^2c_2[M]=108+\frac{2}{9}c_1^4[M].$$ 
Substituting the above expressions in terms of $v_1$ and $v_2$, we obtain (i) and (ii). 

To finish the proof we just need to determine the possible integer solutions $(a,b,c,d)$ of (i) and (ii). This is done in Lemma \ref{number theory 2}. 
The possible values of the Chern numbers
$c_1^4[M], c_1^2c_2[M]$ and $c_2^2[M]$ can then be obtained using \eqref{chern index 3}. 
\end{proof}

\begin{lemma} \label{numlem} Let $x,y\in \Z$ and  $y_{1},y_{2} \in \{-7,7\}$  be such that
$$x^2+ y^2 + 24 =  y_{1} x   +  y_{2}  y.$$
Then, $(|x|,|y|)$ is one of $(3,3), (3,4), (4,3)$ or $(4,4)$.
\end{lemma}

\begin{proof}
First we note that we can reduce to solving 
\begin{equation}\label{eq:main} x^2+ y^2 + 24 = 7 x   +  7  y.
\end{equation}
Indeed,  if $(x,y) \in \mathbb{Z}^2$ is a solution of 
$$ x^2+ y^2+ 24 = y_{1} x   +  y_{2}  y$$
for some $y_1,y_2\in \{\pm 7\}$,
then any pair of the form $(x',y') = (\pm x, \pm y)$ is also a solution of \eqref{eq:main} for  some other values of $y_1,y_2\in \{\pm 7\}$.

Since the left hand side of \eqref{eq:main}  is positive, we have
$$ (x^2+ y^2) + 24 = 7 \lvert x   +  y \rvert \leq 7\left(\lvert x\rvert +\lvert y\rvert\right)$$
and so
\begin{equation}\label{eq:two} 
(x^2-7|x|) +(y^2-7|y|) \leq -24.
\end{equation}
Suppose, without loss of generality, that  $(x^2-7|x|) \leq  (y^2-7|y|)$. Then by \eqref{eq:two} we have $x^2-7|x| +12 \leq 0$.
The only integer values of $|x|$ satisfying this inequality are $3$ or $4$ and that for both of these values  we have $x^2-7|x|=-12$. Consequently, by \eqref{eq:two}, we get that $ (y^2-7|y|) \leq -12$. The same argument then gives $|y| = 3$ or $4$.

\end{proof}

\begin{lemma}\label{number theory 2}
Let $u:=(a,b)\in \Z^2$ and $v:=(c,d)\in \Z^2$, where we endow $\Z^2$ with standard scalar product. Then there are finitely many integer vectors $u$ and $v$ satisfying
 \begin{equation}\label{eq:vectors}
\left\{ \begin{array}{l}
\lVert u \rVert^2 =  3 \lVert v \rVert^2 + 82 \\ \\
 \langle u, v \rangle  =  2 \lVert v \rVert^2  + 12.
 \end{array} \right.
\end{equation}
Moreover the possible values of $\| u\|^2, \langle u, v \rangle$ and $\| v \|^2$ are given in Table~\ref{tab:chern}:
\begin{table}[h!]
\begin{center}
\begin{tabular}{|c||c|c|c|c|c|c|c|c|}
\hline
$\,\, \lVert u \rVert^2  $    & 97 & 106 & 109 & 130 & 136 & 157 & 160 & 169 \\ \hline
$\langle u, v \rangle  $ & 22 & 28  & 30  & 44  & 48  & 62  & 64  & 70  \\ \hline
$\,\, \lVert v \rVert^2  $    & 5  & 8   & 9   & 16  & 18  & 25  & 26  & 29  \\ \hline
\end{tabular}
\end{center}
\label{values u v}
\caption{Possible values for $\lVert u \rVert^2$, $\langle u, v \rangle$ and $\lVert v \rVert^2 $ }\label{tab:chern}
\end{table}

\end{lemma}
\begin{proof}

By the Cauchy-Schwarz inequality and the first equation of \eqref{eq:vectors} we have
\begin{equation}
\langle u, v \rangle^2  \leq   \lVert u  \rVert^2      \lVert v  \rVert^2 = \left(3 \lVert v \rVert^2 + 82 \right)    \lVert v  \rVert^2,
\end{equation}
and so, by the second equation of \eqref{eq:vectors}, we obtain 
$$
 \langle u, v \rangle^2 =  (2 \lVert v \rVert^2  + 12)^2 \leq 3 \lVert v \rVert^4 + 82 \lVert v  \rVert^2.
$$
We conclude that 
$$
\left(\lVert v \rVert^2 - 17\right)^2 \leq 145,
$$
and then $\lVert v \rVert^2  \in [5,29] \cap \mathbb{Z}$. Moreover, since $\lVert v \rVert^2=c^2+d^2$ and $\lVert u \rVert^2=a^2+b^2$  are  sums of two squares and, by \eqref{eq:vectors}, we have 
$\lVert u \rVert^2 = 3 \lVert v \rVert^2 + 82$,  we conclude that
\begin{equation}\label{eq:list}
\lVert v \rVert^2  \in \{5,8,9,13,16,18,25,26,29\}.
\end{equation}
Note that $ \lVert v \rVert^2$ cannot admit the values $10, 17$ and $20$ despite the fact that they are sums of two squares, since the corresponding values of $ \lVert u \rVert^2$ would be $112, 133$ and $142$ and these are not sums of two squares\footnote{The nonnegative integers that are sums of two squares were  characterized by Euler \cite[p.230]{D} and form the sequence A001481 in OEIS. The first elements of this sequence are 0, 1, 2, 4, 5, 8, 9, 10, 13, 16, 17, 18, 20, 25, 26, 29, 32, 34, 36, 37, 40, 41, 45, 49, 50, 52, 53, 58, 61, 64, 65, 68, 72, 73, 74, 80, 81, 82, 85, 89, 90, 97, 98, 100, 101, 104, 106, 109, 113, 116, 117, 121, 122, 125, 128, 130, 136, 137, 144, 145, 146, 148, 149, 153, 157, 160,...}.

We further note that $\lVert v \rVert^2 \neq 13$. Indeed, if $\lVert v \rVert^2 =13$, then, by \eqref{eq:vectors}, we have  $\lVert u \rVert^2 =121$ and  $\langle u,v\rangle=38$. However, since $13=9+4$ and $121=11^2+0$ are  the only possible partitions of $13$ and $121$ as a sum of two squares, we would have that  $\langle u,v\rangle$ is either $\pm 33$ or $\pm 22$.

It is easy to check that all other values of $\lVert v \rVert^2$ in \eqref{eq:list} are possible and so we conclude that  there are only eight possibilities for $\lVert v \rVert^2$ and consequently for $\lVert u \rVert^2$  and $\langle u,v \rangle$ (see Table~\ref{values u v}).

\end{proof}


\section{Multigraphs describing the action and their positivity properties}\label{multigraph section}

Let $(M,J)$ be a closed almost complex manifold of dimension $2n$ acted on by a torus $T$ of dimension $k$ that preserves $J$. Let $\mathfrak{t}$ (resp.\ $\mathfrak{t}^*$) be the Lie algebra (resp.\ the dual of the Lie algebra) of $T$
and $\ell\subset \mathfrak{t}$ the integral lattice (resp.\ $\ell^*\subset \mathfrak{t}^*$ the dual integral lattice). Let $M^T$ be the set of fixed points of the action. 
In this section we review the notion of multigraph describing the $T$-action on $(M,J)$, assuming that there are finitely many fixed points, and describe some of its properties.
 This concept is not new and has been introduced and studied by several authors (see for instance \cite{T, GS, JaTo, Ja}). 

To begin with, it is well-known that around each fixed point $p$ there exist complex coordinates $(z_1,\ldots,z_n)$ and elements $w_1,\ldots,w_n\in \ell^*$, called the \textbf{(isotropy) weights of the action at $p$}, that model the
action around $p=(0,\ldots,0)$. Namely, for  $ \lambda = \exp(\xi)\in T$, we have
\begin{equation}\label{eq weights}
\lambda * (z_1,\ldots,z_n)=(e^{2\pi i w_1(\xi)}z_1,\ldots,e^{2\pi i w_n(\xi)}z_n).
\end{equation}
If $p$ is an isolated fixed point, then none of the weights at $p$ can be zero. 

Let $W$ be the multiset of all the weights at the fixed points of the $T$-action on $(M,J)$. Note that if a weight $w\in \ell^*$ appears $m$ times as an isotropy weight at (possibly different) fixed points, it will appear $m$ times in $W$. This integer $m$ is called the \textbf{multiplicity} of $w$ and will be denoted here by $N(w)$. The next result is a straightforward generalization of \cite[Proposition 2.11]{Ha} and, although it has been proved in \cite[Proposition 2.2]{Ja},  we present  here a different proof.
\begin{lemma}\label{weights in pairs}
Let $(M,J)$ be a closed almost complex manifold equipped with an  effective action of a torus $T$ of dimension $k$ with isolated fixed points. 
Let $W$ be the multiset of weights and let $N(w)$ be the multiplicity of a weight $w$ in $W$.
Then 
\begin{equation}\label{multiplicities weights}
N(w)=N(-w)\quad \text{for every }w\in W\,.
\end{equation}
\end{lemma}
\begin{proof}
In \cite[Proposition 2.11]{Ha} Hattori proves \eqref{multiplicities weights} for circle actions with isolated fixed points. Hence, in order to prove it for torus actions, it is sufficient to show that,
if \eqref{multiplicities weights} does not hold  for the action of a torus $T$ of dimension $k$ with isolated fixed points, then there exists a circle subgroup of $T$ 
acting with isolated fixed points whose weights do not satisfy \eqref{multiplicities weights}.

We begin with the following observations:
\begin{enumerate}
\item[$(i)$] Let $G$ be a circle subgroup of $T$ and assume that $\xi$ is a primitive element of the lattice $\ell_G$ of $\text{Lie}(G)$ that determines the orientation of $G$, i.e. $G=\exp\langle \xi\rangle$. Then 
the set $M^G$ of fixed points  of $G$ contains $M^T$ and, for each $p\in M^T$, the weights of the $G$-action at $p$ are given by the integers $w_i(\xi)$, for  $i=1,\ldots,n$, where $w_1,\ldots,w_n$
are the weights of the $T$-action at $p$. 
Moreover, if $N(w(\xi))$ denotes the multiplicity of the weight $N(w(\xi))$ of the $G$-action at $p\in M^T$, then
$$
N(w(\xi))\geq N(w),
$$
and equality holds iff $w(\xi)\neq w'(\xi)$ for all the weights $w'$ of the $T$-action different from $w$. 
\item[$(ii)$]  The subset $D\subset \mathfrak{t}$ of $\xi\in \frak{t}$ for which $h\xi\in \ell$ for some $h\in \Z$ and such that the corresponding circle subgroup $G:=\{\exp(t\xi), t\in \R\}$ satisfies $M^G=M^T$ is dense in $\mathfrak{t}$.
\end{enumerate}
Suppose that there exists $\widetilde{w}\in W$  such that $N(\widetilde{w})\neq N(-\widetilde{w})$. The  set 
$$
U:=\{\xi \in \mathfrak{t}\mid \widetilde{w}(\xi)\neq \pm w(\xi)\text{ for all }w\in W\setminus\{\widetilde{w}\}\}
$$
is open in $\mathfrak{t}$ and is invariant under multiplication by $h\in \R$. As  $D$ is dense in $\mathfrak{t}$ and invariant under multiplication by $h\in \Z$, we have that $D\cap U \cap \ell\neq \varnothing$. Consider a primitive element $\xi$ in this set and the action of the corresponding circle subgroup $G$. Then, by $(i)$ and the definition of $U$, it is easy to check that 
$N(\widetilde{w}(\xi))=N(\widetilde{w})$ and $N(-\widetilde{w}(\xi))=N(-\widetilde{w})$. As we assumed that $N(\widetilde{w})\neq N(-\widetilde{w})$, the existence of such $\widetilde{w}$ would then
contradict \cite[Proposition 2.11]{Ha}.
\end{proof}

\begin{remark}\label{weights isotropy manifold}
For any fixed point $p\in M^T$ and every  weight $w\in \ell^*$ of the $T$-action at  $p$,  the set  $M^w$ of fixed points of the subgroup 
\begin{equation}\label{eq:Hw}
H_w:=\exp\{\xi \in \mathfrak{t}\mid w(\xi)\in \Z\}
\end{equation}
of $T$  is a closed almost complex submanifold, invariant under the action of $T$.
The weights of the $T$-action on $M^w$ at the points in $M^T\cap M^w$ are all integral multiples of $w$.  
By Lemma \ref{weights in pairs} applied to $M^w$,   we have that for every $p\in M^T$ and every weight $w$ at $p$, there exists 
another fixed point $q\in M^T$ such that one of the weights at $q$ is $-w$, and such that $p$ and $q$ belong to $M^w$. 
Moreover, let $p\in M^T\cap M^w$ such that $k\,w$ is a weight at $p$
for some $k\in \Z$. Since $M^{k\,w}\subseteq M^w$ for every $k\in \Z\setminus\{0\}$, the connected component $N$ of $M^{k\,w}$ containing $p$ is contained in $M^w$.
Therefore there exists another fixed point
$q\in N\subset M^w$ such that  $-k\,w$ is a weight at $q$.
\end{remark}

As a consequence of Lemma \ref{weights in pairs}, the multiset $W$ can be written as a disjoint union
\begin{equation}\label{disjointunion} 
W= W^+  \sqcup W^- 
\end{equation}
of two multisets $W^+$ and $W^-$ of the same cardinality, where 
$w\in W^+$ if and only if $-w\in W^-$. 
 
 In the following, the triple $(M,J,T)$ denotes a closed almost complex manifold equipped with an effective action of a torus $T$ preserving $J$
with isolated fixed points.
\begin{definition}\label{def:multigraph}
We say that a \textbf{multigraph $\Gamma=(V,E)$ describes $(M,J,T)$} if
\begin{itemize}
\item The vertex set $V$ is in bijection with the fixed point set $M^T$;
\item There exists a bijection $g\colon W^+\to W^-$, such that $g(w)=-w$ for every $w\in W^+$  so that,
for each edge $e\in E$ with endpoints $p,q$, there exists a weight $w_{p,e}$ at $p$ and a weight $w_{q,e}$ at $q$, such that $g(w_{p,e})=w_{q,e}$ or $g(w_{q,e})=w_{p,e}$. 
Moreover, $p$ and $q$ must belong to the same connected component of $M^w$, where $w$ is  $w_{p,e}$ and  $M^w$ is the fixed point set of the subgroup $H_w$ defined  in \eqref{eq:Hw}.
\end{itemize}
\end{definition}
Note that the edges of the multigraph are \emph{undirected} and that, in the notation above, $M^{w_{p,e}}=M^{w_{q,e}}$.

\begin{remark}\label{rmk:circle1}
When $T\simeq S^1$ is a circle the weights at fixed points are integers and so we can actually take the multisets $W^+$ and $W^-$ to be the multisets of positive and negative weights. Moreover, for any weight $w$ of $S^1$ we have 
that $M^{\Z_{\lvert w\rvert}}:=M^w $ is the submanifold of $M$ formed by all of the points that are fixed by the subgroup $\Z_{\lvert w \rvert}$ of $S^1$. 
\end{remark}

\begin{remark}\label{rmk multigraph hamiltonian}
Let $(M,\omega)$ be a closed, symplectic manifold admitting a symplectic action of a torus $T$ with isolated fixed points. In this case,
we choose $J$ to be a compatible, invariant almost complex structure and say that a multigraph describing $(M,J,T)$ is a multigraph describing $(M,\omega,T)$. If the action is Hamiltonian with moment map $\psi$, we say that a multigraph,   describing $(M,J,T)$, describes $(M,\omega,\psi)$. 
\end{remark}

\begin{remark}\label{graph GKM}
Observe that, for Hamiltonian GKM spaces, there is exactly one multigraph describing the action and that this multigraph is given
exactly by the GKM graph (see Section \ref{Ham GKM spaces}).
\end{remark}

Lemma \ref{weights in pairs} and Remark \ref{weights isotropy manifold} assert that every action of a torus on a closed almost complex manifold with isolated fixed points can be described by a multigraph.
\begin{remark}\label{restriction multigraph}
For every $w\in W$, consider the almost complex manifold $(M^w,J)$ with the restriction of the $T$-action. Then it is clear that any multigraph describing $(M,J,T)$ restricts to a multigraph   describing $(M^w,J,T)$. 
\end{remark}
The next result is important for the definition of the subclass of multigraphs that we are interested in.

\begin{lemma}\label{integral multigraph lemma}
 Let $\Gamma=(V,E)$ be a multigraph
describing $(M,J,T)$. Let $e\in E$ be en edge with endpoints $p,q$ and let $w:=w_{p,e}=-w_{q,e}$.
Then, for every class $\alpha\in H^2_T(M;\Z)$, we have
\begin{equation}\label{condition multi}
\alpha(p)-\alpha(q)=  k \,w, \quad\text{for some }k\in\Z.
\end{equation}
\end{lemma}

\begin{proof}
Each class $\alpha\in H^2_T(M;\Z)$ defines a (unique) $T$-equivariant line bundle $V$ on $M$ such that $c_1^T(V)=\alpha$ (see for instance \cite[Theorem A.1]{HL}). 
Moreover, for each fixed point $p\in M^T$, the complex line $V(p)$ inherits a representation of $T$ with weight given exactly by $\alpha(p)\in \ell^*$. 
By definition of multigraph, the endpoints $p,q$ of each edge  are in the same connected component of $M^w$ and so the $H_w$-representations at
$V(p)$ and $V(q)$ are isomorphic (see Remark \ref{weights isotropy manifold}), thus implying \eqref{condition multi}. 

%
\end{proof}

Consider the field of fractions $Q(\mathfrak{t}^*)$ of  the symmetric algebra $\mathbb{S}(\mathfrak{t}^*)$ of $\mathfrak{t}^*$. Then, a priori, using  the above notation,
$$
\frac{\alpha(p)}{w_{p,e}}+\frac{\alpha(q)}{w_{q,e}}\in Q(\mathfrak{t}^*).
$$  
However, an immediate consequence of Lemma~\ref{integral multigraph lemma} is the following result.
\begin{cor}\label{real edge}
Let $\Gamma=(V,E)$ be a multigraph describing $(M,J,T)$. Then, for each edge $e\in E$ with endpoints $p,q$ and each cohomology class $\alpha\in H^2_T(M;\Z)$, we have
\begin{equation}\label{defining volume alpha}
\alpha[e]:=\frac{\alpha(p)}{w_{p,e}}+\frac{\alpha(q)}{w_{q,e}}\in \Z\,.
\end{equation}
\end{cor}
\begin{proof}
This is an immediate consequence of Lemma \ref{integral multigraph lemma}, observing that $w_{p,e}=-w_{q,e}$.
\end{proof}

\begin{definition}\label{evaluation alpha} Let $\Gamma=(V,E)$ be a multigraph describing $(M,J,T)$ and let 
$\alpha\in H^2_T(M;\Z)$. 
For each edge $e\in E$, the integer defined in \eqref{defining volume alpha} is called the \textbf{evaluation of $\alpha$ on the edge $e$}.
\end{definition}

\begin{remark}
Results similar to those in Lemma \ref{integral multigraph lemma} and Corollary \ref{real edge} hold for equivariant cohomology classes with real coefficients. In particular,
given $\alpha\in H^2_T(M;\R)$, an edge $e$ with endpoints $p,q$, and  $w=w_{p,e}=-w_{q,e}$, one has 
\begin{equation}\label{real alpha}
\alpha(p)-\alpha(q)=  k \,w, \quad\text{for some }k\in\R,
\end{equation} 
and so the evaluation $\alpha[e]$ of a 2-form $\alpha\in H_T^2(M;\R)$ on $e$ can be defined as
\begin{equation}\label{evaluation real}
\alpha[e]:=\frac{\alpha(p)}{w_{p,e}}+\frac{\alpha(q)}{w_{q,e}}\in \R\,.
\end{equation}
A simple proof of \eqref{real alpha} can be obtained using the Cartan model for
equivariant cohomology, which is the idea behind the introduction of polynomial assignments in \cite[Section 3]{GSZ}.
\end{remark}

\begin{remark}\label{geometric interpretation}
If there exists a $J$-invariant 2-sphere $S^2$ in $(M,J)$, which is $T$-invariant and such that the $T$-action
 on $S^2$ has exactly $p$ and $q$ as fixed points, respectively  with weights $w_{p,e}$ and $w_{q,e}$, then there is a geometric interpretation of the evaluation of a class $\alpha\in H^2_T(M;\R)$ on the edge $e$ with endpoints $p$ and $q$. Indeed, by the Atiyah-Bott-Berline-Vergne Localization Theorem \cite{AB,BV}, the evaluation of $\alpha$
on $e$ is exactly the integral of $\alpha$ on $S^2$, i.e.
$$
\int_{S^2}\alpha=\frac{\alpha(p)}{w_{p,e}}+\frac{\alpha(q)}{w_{q,e}}=\alpha[e]\,.
$$
\end{remark}
Remark~\ref{geometric interpretation} leads to a generalization of the concept of toric $1$-skeleton as defined in \cite[Definition 4.10]{GHS} for Hamiltonian $S^1$-spaces.
\begin{definition}\label{toric one skeleton} Let $(M,J)$ be a closed almost complex manifold equipped with an  effective action of a torus $T$ with isolated fixed points and
let $\Gamma=(V,E)$ be a multigraph describing $(M,J,T)$. 
Assume that for each edge $e\in E$ with endpoints $p,q$, there exists a smoothly embedded  $J$-holomorphic sphere $S^2_e$ which is $T$-invariant  and
such that the fixed points are exactly $p,q$ with weights (respectively) $w_{p,e}, \,w_{q,e}$.  In this case, we say that $\Gamma$ \textbf{admits a toric 1-skeleton}, and
 the set $\mathcal{S}=\{S^2_e\}_{e\in E}$  is the \textbf{toric 1-skeleton} associated to $\Gamma$\footnote{See also \cite{Cha} for a 
generalization of  toric $1$-skeletons to complexity-one spaces whose fixed point sets are not necessarily discrete. }.
\end{definition}
\begin{example}\label{example spaces toric one skeleton}
There are several examples of spaces  that admit a toric $1$-skeleton (see \cite[Lemma 4.11]{GHS}). For instance,
$(M,\omega,\psi)$ can be described by a multigraph that admits a toric $1$-skeleton if the $T$-action is, or extends to,
\begin{itemize}
\item[(i)] a GKM action and, in particular, to a symplectic toric action (see Section \ref{Ham GKM spaces}),
\item[(ii)] an $S^1$-action with isolated fixed points such that at each fixed point the weights are pairwise coprime and not equal to 1,
\item[(iii)] an $S^1$-action with isolated fixed points and $\dim M \leq 4$.
\end{itemize}
\end{example}

An important geometric consequence of the existence of a toric  $1$-skeleton is the following result, which is a straightforward generalization of \cite[Lemma 4.13]{GHS}.
\begin{lemma}\label{poincare dual} Let $(M,J)$ be a closed almost complex manifold of dimension $2n$ equipped with an  effective action of a torus $T$ with isolated fixed points and
let $\Gamma=(V,E)$ be a multigraph describing $(M,J,T)$.  Consider the natural restriction map 
$$
r\colon H^2_T(M;\Z)\to H^2(M;\Z).
$$
Then, for any equivariant cohomology class $\alpha\in H^2_T(M;\Z)$, we have
\begin{equation}\label{poincare dual alpha}
\sum_{e\in E}\alpha[e]=\int_M r(\alpha)\cup c_{n-1}\,.
\end{equation}
Moreover, if  $r$ is surjective and $\Gamma$ admits a toric 1-skeleton $\mathcal{S}=\{S^2_e\}_{e\in E}$, then 
the Chern class $c_{n-1}\in H^{2(n-1)}(M;\Z)$ is the Poincar\'e dual to the class of $\mathcal{S}\in H_2(M;\Z)$.

\end{lemma}
\begin{proof}
This is a generalization of \cite[Lemma 4.13]{GHS}, where the only fact used about the existence of a Hamiltonian $T$-action is the surjectivity of $r$. The details are left to the reader. 
\end{proof}
\begin{remark}
The condition that  $r\colon H^2_T(M;\Z)\to H^2(M;\Z)$ is surjective is equivalent to the equivariant formality of the $T$-action. For a nice account of these facts see \cite[Section 7]{GZ}.
\end{remark}

\begin{definition}\label{evaluation alpha and positive multigraph}
Let $(M,J)$ be a closed almost complex manifold equipped with an  effective action of a torus $T$ with isolated fixed points and let $\Gamma=(V,E)$  be a multigraph describing $(M,J,T)$.
Let $\alpha\in H^2_T(M;R)$, where $R$ is either $\Z$ or $\R$.  We say that an \textbf{edge $e\in E$ is positive w.r.t.\ $\alpha$} if the evaluation of $\alpha$ on $e$ is positive, i.e. if
\begin{equation}\label{positive edge}
\alpha[e]=\frac{\alpha(p)}{w_{p,e}}+\frac{\alpha(q)}{w_{q,e}}>0,
\end{equation}
where $p,q\in V$ are the endpoints of $e$.
The \textbf{multigraph}  $\Gamma=(V,E)$ is said to be \textbf{positive w.r.t.\ $\alpha$} if all the edges are positive w.r.t.\ $\alpha$.
\end{definition}
\begin{remark}\label{positivity for symplectic}
The condition for an edge to be positive w.r.t.\ an equivariant cohomology class $\alpha$ can be interpreted as follows:
An edge $e$ with endpoints $p,q$ and weights $w_{p,e}, w_{q,e}$ is positive w.r.t.\ $\alpha$ if and only if 
\begin{center}$\alpha(p)-\alpha(q)$, regarded as a vector in $\mathfrak{t}^*$, is a positive multiple of $w_{p,e}$.
\end{center} 
Note that this condition is equivalent to saying that $\alpha(q)-\alpha(p)$ is a positive multiple of $w_{q,e}$. 
Therefore the concept of ``positive edge'' (w.r.t.\ $\alpha$) belongs to undirected edges.
\end{remark}
\begin{example}
For Hamiltonian $T$-actions with isolated fixed points, the positivity of an edge $e$ with endpoints $p,q$ w.r.t.\, the class of the equivariant symplectic form $[\omega-\psi]$ says that
the vector $\psi(q)-\psi(p)$ is a positive multiple of $w_{p,e}$.
\end{example}

\begin{remark}\label{rmk:circle2} \textbf{(Orientation on $\Gamma$)} 
In the particular case of a circle action and, taking the sets $W^+$ and $W^-$ to be respectively those of positive and negative weights (see Remark~\ref{rmk:circle1}), we can 
direct the edges of any multigraph $\Gamma$ describing $(M,J,S^1)$, constructed using a bijection $g:W^+\to W^-$ as in Definition~\ref{def:multigraph}. 
Namely, we orient an edge $e$ with endpoints $p,q$ from $p$ to $q$ if $g(w_{p,e})=w_{q,e}$, therefore if $w_{p,e}>0$. In this case, we define 
$i(e)$ to be the initial point (also known as \textbf{tail}) of $e$  and $t(e)$ to be the terminal point (also known as \textbf{head}) of $e$. 

Note that, according to Definition~\ref{evaluation alpha and positive multigraph}, an oriented multigraph $\Gamma=(V,E)$ describing $(M,J,S^1)$ is said to be positive w.r.t.\ some $\alpha\in H^2_{S^1}(M;R)$, where $R$ is either $\Z$ or $\R$, if and only if for every oriented edge $e\in E$, we have
$$
\alpha(i(e)) > \alpha(t(e))
$$
(see \eqref{eq:equiv}).

When $T$ is a torus of higher dimension, 
we can orient the graph by considering a generic $\xi\in \mathfrak{t}$ (namely $w(\xi)\neq 0$ for any isotropy weight $w\in \ell^*$)
 such that $\xi$ generates a circle $S^1=\exp\langle \xi \rangle$, and use this circle to orient the edges of $\Gamma$, as described above.
\end{remark}

There is another natural consequence of the existence of a toric $1$-skeleton for symplectic manifolds with a Hamiltonian action.
\begin{proposition}\label{symplectic positive}
Let $(M,\omega)$ be a Hamiltonian $T$-space with isolated
fixed points and moment map $\psi\colon M \to \mathfrak{t}^*$. Assume that there exists a multigraph $\Gamma$
describing $(M,\omega,\psi)$ which admits a toric $1$-skeleton.  Then $\Gamma$ is positive w.r.t.\ $[\omega-\psi]$.

\end{proposition}

\begin{proof}
Since the spheres of the toric $1$-skeleton are, by definition, $J$-invariant, they are symplectic surfaces in $M$, and therefore $\int_{S^2} \omega>0$. The result then
follows from Remark \ref{geometric interpretation} and from the fact that, for degree reasons, we have $\int_{S^2}\omega=\int_{S^2}(\omega-\psi)$.
\end{proof}
\begin{remark}\label{rmk:existence positive multigraph}
Let $(M,\omega)$ be a Hamiltonian $T$-space with isolated fixed points and moment map $\psi\colon M \to \mathfrak{t}^*$. To the best of our knowledge, every known example of such $(M,\omega,\psi)$ admits a multigraph describing it which is positive w.r.t.\ $[\omega-\psi]$. If a multigraph describing $(M,\omega,\psi)$ admits a toric $1$-skeleton (which is the case, for instance, for Hamiltonian GKM spaces, and, in particular, for symplectic toric manifolds, see Example \ref{example spaces toric one skeleton}), then its positivity w.r.t.\ $[\omega-\psi]$ comes from Proposition \ref{symplectic positive}. However, for other examples of 
$(M,\omega,\psi)$ for which the existence of a $1$-skeleton is not known, the existence of a positive multigraph w.r.t.\ $[\omega-\psi]$ can be checked directly.  
\end{remark}
Given the remark above, we believe that the following question is interesting.
\begin{questionL}
Let $(M,\omega)$ be a Hamiltonian $T$-space with isolated fixed points and moment map $\psi\colon M \to \mathfrak{t}^*$. 

Does there always exist a
multigraph describing $(M,\omega,\psi)$ that is positive w.r.t.\ $[\omega-\psi]$? 

Given the existence of such multigraph, is there a toric $1$-skeleton associated to it?
\end{questionL}
Before stating the next lemma, we introduce the following notation. We say that a subtorus $K$ of a torus $T$  is \textit{generic} if, for every $w\in \mathfrak{t}^*$ that is a weight of the isotropy $T$-action at some fixed point, we have
\begin{equation}\label{eq:neq0}
w\neq 0 \quad \text{iff} \quad \pi_K(w)\neq 0, 
\end{equation}
where $\pi_K\colon \mathfrak{t}^*\to \mathfrak{k}^*$ is the dual of the inclusion  map $\iota_K\colon \mathfrak{k}\to \mathfrak{t}$, and $\mathfrak{k}$
is the Lie algebra of $K$. Note that, since  $K$ is generic, we have $M^K=M^T$.
The projection $\pi_K$ can be extended to a map from $\mathbb{S}(\mathfrak{t}^*)$ to $\mathbb{S}(\mathfrak{k}^*)$ by imposing that it is an algebra morphism.

\begin{lemma}\label{positive subactions} Let $(M,J)$ be a closed almost complex manifold equipped with an  effective action of a torus $T$ with isolated fixed points and let $\alpha\in H^2_T(M;R)$, where $R$ is either $\Z$ or $\R$. 
If $K$ is a subtorus of $T$, let $r_K\colon H_T^2(M;R)\to H_K^2(M;R)$ be the restriction map in equivariant cohomology.
Then the following statements are equivalent:
\begin{itemize}
\item[(i)] There exists a multigraph describing $(M,J,T)$ that is positive w.r.t.\ $\alpha$;
\item[(ii)] For all generic subtori $K\subset T$, there exists a multigraph describing $(M,J,K)$ that is positive w.r.t.\ $r_K(\alpha)$;
\item[(iii)] There exists a generic $K\subset T$ and a multigraph describing $(M,J,K)$ that is positive w.r.t.\ $r_K(\alpha)$.
\end{itemize}  
\end{lemma}
\begin{proof}
The equivalences follow easily from the facts that $(r_K(\alpha))(p)=\pi_K(\alpha(p))$, that if $w$ is a weight of the $T$-representation at $p\in M^T$, then $\pi_K(w)$ is the weight of the $K$-representation at $p\in M^K$ and that
$$R\ni \frac{\alpha(p)}{w_{p,e}}+\frac{\alpha(q)}{w_{q,e}} = \pi_k\left( \frac{\alpha(p)}{w_{p,e}}+\frac{\alpha(q)}{w_{q,e}} \right)\,.$$
\end{proof}

The following result will be useful in the next section.
\begin{lemma}\label{lemma graph submanifolds}
Let $(M,J)$ be a closed almost complex manifold equipped with an  action of a torus $T$ with isolated fixed points and let $\alpha\in H^2_T(M;R)$, where $R$ is either $\Z$ or $\R$. Let $W$ be the multiset of all isotropy weights and, for any $w\in W$, let $M^w$ be  the fixed point set of the group $H_w$ defined in \eqref{eq:Hw}.
Let $\Gamma=(V,E)$ be a multigraph describing $(M,J,T)$ and, for every $w\in W$, let $\Gamma_w=(V_w,E_w)$ be the restriction of $\Gamma$ to $M^w$, which is a multigraph
describing $(M^w,J,T)$ (see Remark \ref{restriction multigraph}).

Then $\Gamma$ is positive w.r.t.\ $\alpha$, if and only if, for every $w\in W$, $\Gamma_w$ is positive w.r.t.\ $\alpha$.
\end{lemma}
\begin{proof}
In order to prove the statement it is enough to observe that any edge $e\in E$ is an edge of $E_w$, for some $w\in W$. 
\end{proof}

\subsection{Existence of positive multigraphs}$\;$\\
The main goal of this section is to prove Theorem \ref{existence multigraph dim 8}. We begin with the following observations.
\begin{lemma}\label{lemma monotone chi six}
Let $(M,\omega)$ be an 8-dimensional,   closed,  symplectic manifold 
with a Hamiltonian action of a torus $T$ with isolated fixed points. Let $\psi\colon M \to \mathfrak{t}^*$ be the corresponding
moment map.

If $\chi(M)=6$,
then the symplectic form can be rescaled and the moment map can be translated so that 
\begin{equation}\label{c1 equal omega}
c_1^T=[\omega-\psi].
\end{equation} 
In particular, with the aforementioned rescaling, the weight sum formula \eqref{wsf} holds and we have $c_1=[\omega]$.
\end{lemma}
\begin{proof}
By Proposition \ref{sixfp} (1), the second Betti number is one. Therefore $c_1=\lambda [\omega]$ for some $\lambda\in \R$ and, by \cite[Lemma 5.2]{GHS} applied to a generic subcircle $S^1\subset T$, we have that $\lambda >0$. Therefore the symplectic form can be rescaled so that $c_1=[\omega]$. As a consequence, if we consider the equivariant extensions of $c_1$ and $\omega$, namely $c_1^T$ and $[\omega-\psi]$, there must be an element $a$ of $\mathfrak{t}^*$ such that $c_1^T=[\omega-\psi]+a$. 
If we replace the moment map $\psi$ with $\psi':=\psi-a$,  we have that $c_1^T=[\omega-\psi']$. Restricting the previous equality to a fixed point $p$, yields the weight sum formula \eqref{wsf}.
\end{proof}
\begin{remark}
Henceforth, if the manifold we are working with satisfies the hypotheses of Lemma \ref{lemma monotone chi six}, we will always assume that the symplectic form is rescaled and the moment map is
translated so that \eqref{c1 equal omega} holds. 
\end{remark}
In order to prove the main result of this section we need the following results. For circle actions, we orient the multigraph as in Remark \ref{rmk:circle2}. 
\begin{lemma}[Jang, Tolman \cite{JaTo}]\label{lemma jang tolman}
Let $(M,J)$ be a closed almost complex manifold acted on by a circle $S^1$ with isolated fixed points. 
Then there exists an oriented multigraph $\Gamma=(V,E)$ describing
$(M,J,S^1)$ such that, for each $e\in E$, the index of $i(e)$ in the isotropy submanifold $M^{\Z_{w(e)}}$ is two less than the index
of $t(e)$ in $M^{\Z_{w(e)}}$. 
\end{lemma}
The following is an immediate consequence of \cite[Lemma 3.1]{T}.
\begin{lemma} \label{hamindex}
Suppose that $(M,\omega)$ is a closed symplectic manifold equipped with a Hamiltonian $S^1$-action with isolated fixed points and moment map $\psi: M \rightarrow \mathbb{R}$. 
Let $p_{\min}$ be the fixed point where $\psi$ attains its minimum. If $q$ is a fixed point such that there is no fixed point in $\psi^{-1}(\psi(p_{\min}), \psi(q))$, then $\lambda(q)=1$. Furthermore, every fixed point  $q^\prime$ with $\psi(q^\prime)=\psi(q)$ has $\lambda(q')=1$.
\end{lemma} 
\begin{remark}\label{rmk:reverse}
Suppose that the dimension of $M$ is equal to $2n$. 
Reversing the circle action, it is easy to check that if  $p_{\max}$ is the fixed point where $\psi$ attains its maximum and $q$ is a fixed point of the original $S^1$-action for which there is no fixed point in $\psi^{-1}(\psi(q),\psi(p_{\max}))$, then $\lambda(q)=n-1$. Furthermore, every fixed point  $q^\prime$ with $\psi(q^\prime)=\psi(q)$ must have $\lambda(q')=n-1$.
\end{remark}
When the manifold is $8$-dimensional and the number of fixed points is $6$, then the Euler characteristic of $M$ is $6$ and so, by Proposition \ref{sixfp} (1) we have exactly one fixed point $p$ with $\lambda(p)=1$ and two fixed points with $\lambda(p)=2$. Then Lemma~\ref{hamindex} and Remark~\ref{rmk:reverse} imply the following statement.
\begin{lemma} \label{indexnondec}
Let  $(M,\omega)$ be an $8$-dimensional, closed, symplectic manifold equipped with a Hamiltonian $S^{1}$-action with $6$ isolated fixed points and moment map $\psi: M \rightarrow \mathbb{R}$. 
Name the fixed points  $p_{0},p_1,p_2,p_{2}' , p_3,p_{4}$  so that   $\lambda(p_{i}) =i$ for each $i$ and $\lambda(p_{2}')=2$. Then we have
$$\psi(p_{0}) < \psi(p_{1}) < \psi(p_{2}) \leq \psi(p_{2}') <\psi(p_{3}) < \psi(p_{4}). $$
\end{lemma}

Suppose that $(M,\omega,\psi)$ is a Hamiltonian $T$-space of dimension $2n$ with $b_2(M)=1$ and discrete fixed point set $M^T$.
Consider a generic element $\xi$ of $\mathfrak{t}$ generating a circle $\exp\langle \xi \rangle = S^1$. Then, by the perfection of the
associated moment map $\varphi=\langle \psi , \xi \rangle$, there is exactly one fixed point $p_1$ with $\lambda(p_1)=1$ and, by Poincar\'e duality, 
there is exactly one fixed point $p_{n-1}$ with $\lambda(p_{n-1})=n-1$.
The next result illustrates another easy feature of multigraphs associated to Hamiltonian actions of this kind. 

\begin{lemma}\label{multigraph b2=1}
Let $(M,\omega,\psi)$ be a Hamiltonian $T$-space of dimension $2n$ with $b_2(M)=1$ and discrete fixed point set $M^T$.
Let $\xi$ be a generic element of $\mathfrak{t}$ generating a circle $\exp\langle \xi \rangle = S^1$.
Let $p_i$ be the only fixed point with $\lambda(p_i)=i$, for $i\in \{0,1,n-1,n\}$. 

Then there exists a multigraph describing $(M,\omega,\psi)$ that has one edge between $p_0$ and $p_1$, and one edge between $p_{n-1}$ and $p_n$.
\end{lemma}

\begin{proof}
Consider the moment map $\varphi =\langle \psi , \xi \rangle$ of the $S^1$-action on $(M,\omega)$. 
As there is exactly one negative isotropy weight at $p_1$, by the definition of multigraph, there is exactly one edge $e\in E$ with $t(e)=p_1$. 
By Lemma \ref{lemma jang tolman}, there exists a multigraph such that the index of $i(e)$ is zero on the isotropy submanifold of $M^{\Z_{w(e)}}$,
therefore $i(e)$ is a minimum for the restriction of $\varphi$ to $M^{\Z_{w(e)}}$.
Since $p_0$ is the only fixed point of the $S^1$-action on $M$ with $\varphi(p_0)<\varphi(p_1)$, it follows that $i(e)=p_0$.

A similar argument applies to the unique fixed point $p_{n-1}$ with $\lambda(p_{n-1})=n-1$.
\end{proof}

\begin{theorem}\label{existence multigraph dim 8}
Let $(M,\omega)$ be an 8-dimensional,  closed,  symplectic manifold with an effective Hamiltonian action of a torus $T$ with moment map $\psi:M\to \frak{t}^*$ and isolated fixed points. 
Assume that $\chi(M)=6$. 
Then $(M,\omega,\psi)$ can be described by a multigraph $\Gamma=(V,E)$ that is positive w.r.t. $[\omega-\psi]$ if 
\begin{itemize}
\item[(a)] the dimension of $T$ is at least 2, or
\item[(b)] the dimension of $T$ is one and there is no 6-dimensional isotropy submanifold with $6$ fixed points.
\end{itemize}
Moreover, let  $\varphi: M\to \R$  be a generic component of the moment map  with minimum  $p_0$ and maximum $p_4$. Then, there exists a multigraph associated to $(M,\omega,\psi)$ that has one edge connecting  $p_{0}$ to the unique fixed point $p_{1}$ with $\lambda(p_{1})= 1$, and one edge connecting the unique fixed point $p_{3}$ with 
$\lambda(p_{3})=3$ to  $p_{4}$.  
\end{theorem}
\begin{proof}
In both cases (a) and (b), the strategy used is the same: For every weight $w\in W$, consider
an edge $e$ with endpoints $p,q$ such that $w_{p,e}=w$ or $-w$. This edge is also an edge for the restriction $\Gamma^w$ of the
multigraph $\Gamma$ on the isotropy submanifold $M^w$ with the restricted action of $T$ (see Remark \ref{restriction multigraph}).
Therefore, it is enough to prove that the edge $e$, regarded as an edge of $\Gamma^w$, is positive w.r.t. the restriction of $[\omega-\psi]$
on $M^w$. 
(See Lemma \ref{lemma graph submanifolds}.)

We first observe that $M^w=M$ only if $\dim(T)=1$ and $w=\pm 1$. 
This follows from the effectiveness of the action of $T$, the fact that $M$ is connected, and that $M^w$ is fixed by $H_w$, which is the trivial subgroup only if $\dim(T)=1$ and $w=\pm 1$. Let us  assume first that $w=\pm 1$ and consider the corresponding edge $e\in \Gamma$. 
Then the fact that $e$ is positive w.r.t.\ $[\omega-\psi]$ follows from Lemmas \ref{lemma jang tolman} and \ref{indexnondec}.

Now suppose that $H_w$ is not the trivial subgroup, and therefore $M^w$ is a proper submanifold of $M$. 
Consider the effective action of $T':=T/H_w$ on $M^w$. This action is Hamiltonian w.r.t. the restriction of the symplectic form, which we still call $\omega$, and the
restriction of the moment map $\psi$ to $M^w$ (here we identify $(\text{Lie}(T'))^*$ with $\text{Ann(Lie}(H_w))\subset \mathfrak{t}^*$). Moreover, it has isolated fixed points.

If $\dim(M^w)\leq 4$, then $(M^w,\omega,T')$ admits a toric $1$-skeleton (see Example \ref{example spaces toric one skeleton} (iii)), and then, by Proposition \ref{symplectic positive}, it can be described by a positive  multigraph w.r.t. $[\omega-\psi]$. 

If $\dim(M^w)=6$, then we claim that the number of fixed points of the $T'$-action is exactly $4$. 
Indeed, we first note that, since the action of $T^\prime$ on  $M^w$ is Hamiltonian, the number of fixed points is at least $4$. Moreover, by Poincar\'e duality, the number of fixed points must be even. Since we are assuming that $\chi(M)=6$, we just need to prove that, in case (a), the action cannot have $6$ fixed points.
Therefore assume that $\dim(T)\geq 2$. In this case $H_w$ is not discrete, so $\dim (T^\prime)=1$ and the image of $M^w$ by the $T'$-moment map is on an affine line in $\mathfrak{t}^*$ (see identifications above). If the number of fixed points of the $T'$-action on $M^w$ were six, then  the images of all the $T$-fixed points by the $T$-moment map  would be on the same affine line  in $\mathfrak{t}^*$.
However, as $\psi(M)$ is the convex hull of the images of these points, this would imply that the $T$-moment polytope were $1$--dimensional, which is impossible since we are assuming
$\dim(T)\geq 2$ and the $T$-action is effective. It follows that the number of fixed points of the $T'$-action on the $6$-dimensional isotropy submanifold $M^w$ is 4. 
Now consider a generic subcircle $S^1$ of $T'$ acting on $M^w$. As this $S^1$-action has 4 fixed points and $M^w$ is $6$-dimensional, the possible weights were classified by Tolman in \cite{T}. In all of the possible cases the manifold $M^w$ and the $S^1$-action can be described by a multigraph that is positive w.r.t. (the restriction of) $[\omega-\psi]$. Therefore, by Lemma \ref{positive subactions}, the action of $T'$
on $M^w$ can be described by a multigraph that is positive w.r.t. (the restriction of) $[\omega-\psi]$ and therefore so does the action of $T$ on $M^w$. 

In order to prove the last claim, it is enough to observe that, by Proposition \ref{sixfp} (1), we have $b_2(M)=1$, and so the conclusion follows from Lemma \ref{multigraph b2=1}.
\end{proof}
\begin{remark}
If $c_1=\lambda[\omega]$ for some $\lambda>0$, then 
the existence of a multigraph which is positive w.r.t. $[\omega-\psi]$
is equivalent to that of a multigraph which is positive w.r.t. $c_1^T$. Therefore, by Lemma \ref{lemma monotone chi six}, this is the case when 
$\chi(M)=6$.

If $T$ is a circle $S^1$, then, for every fixed point $p$, we have $c_1^{S^1}(p)=\Sigma(p)\,x$ where $\Sigma(p)\in \Z$ denotes the sum of the weights at $p$ and $H_{S^1}(\{p\};\Z)=\Z[x]$. 
We conclude that, if $\chi(M)=6$, a multigraph $\Gamma=(V,E)$ describing $(M,\omega,\psi)$ is positive w.r.t. $[\omega-\psi]$ if and only if 
\begin{equation}\label{positive circle cdt}
\Sigma(i(e))> \Sigma(t(e))\,\quad\text{for every edge  }e\in E\,.
\end{equation}
(see also Remark \ref{rmk:circle2}).
\end{remark}

The next result shows that the only case in Theorem \ref{existence multigraph dim 8} in which the existence of a positive multigraph cannot be ensured 
concerns circle actions in which the isotropy submanifold of dimension 6 containing all the fixed points is fixed by $\Z_{2^j}$, for some $j\in \Z_{\geq 1}$. 
\begin{proposition} \label{isoall}
Let $(M,\omega)$ be a closed symplectic manifold of dimension $2n$ equipped with an effective Hamiltonian $S^{1}$-action with discrete fixed point set $M^{S^1}$.
If there exists a connected isotropy submanifold $M^{{\Z}_m}$ of dimension $2n-2$ containing $M^{S^1}$, then $m=2^j$ for some $j\in \Z_{\geq 1}$.
\end{proposition}
\begin{proof}
First of all we observe that, for every prime $m'$ dividing $m$, the isotropy submanifold $M^{\Z_{m'}}$ contains
$M^{{\Z}_m}$, and therefore it contains $M^{S^1}$. Then, in order to prove the claim, it is enough to prove that $m'$ can only be 2.

Suppose that $M^{S^{1}}$ is contained in a connected component of  $M^{\mathbb{Z}_{m}}$ of dimension $2n-2$ for some prime integer $m \geq 2$. Since the $S^1$-action is effective, exactly one of the isotropy weights at the minimum is not divisible by $m$.  Let  $w$ be this isotropy weight. Then, by \cite[Proposition 2.11]{Ha}, there exists another fixed point which has $-w$ as an isotropy weight. Since, by assumption, both fixed points are in the same connected component of $M^{\mathbb{Z}_{m}}$, we have  by \cite[Lemma 2.6]{T} that
$$ w=-w \mod m.$$ 
It follows that $2w = 0 \mod m$, and then, since $m$ is prime and $w$ is not divisible by $m$, we have  that $m=2$ as required.

\end{proof}

\begin{remark} Here we give an example of a Hamiltonian $S^1$-action on a $6$-dimensional symplectic manifold with isolated fixed points, and with a connected isotropy submanifold of dimension $4$ containing all the fixed points. Consider the $S^1$-action on  $M=Gr^{+}_{\mathbb{R}} (2,5)$ induced by the linear circle action on $\mathbb{R}^5 = \mathbb{R} \oplus \mathbb{C}^2$ given by
$$ z \cdot  (x,w_1,w_2) = (x, z^{-1}w_1, z w_{2}).$$
It is easy to check that this action has $4$ isolated fixed points: $\mathcal{L}\{(0,w_1,0)\}$, $\mathcal{L}\{(0,0,w_2)\}$, $\mathcal{L}\{(0,\overline{w}_1,0)\}$ and  $\mathcal{L}\{(0,0,\overline{w}_2)\}$.
Moreover, the subgroup $\mathbb{Z}_{2} \subset S^1$ fixes all the oriented planes in the subspace $\{0\} \times \mathbb{C}^2 \subset \mathbb{R}^5$, and so $M^{\mathbb{Z}_{2}} = Gr^{+}_{\mathbb{R}} (2,4) \subset Gr^{+}_{\mathbb{R}} (2,5) $ is $4$-dimensional and contains the four fixed points.
\end{remark}

\subsection{Positive multigraphs and Betti numbers}
In this section we study  the consequences of the existence of a positive multigraph on an 8-dimensional, closed, symplectic manifold admitting a Hamiltonian $S^1$-action with finitely many fixed points, proving two finiteness results, one for the Betti numbers (Proposition \ref{index betti}) and other for the Chern numbers (Corollary \ref{finitely many chern numbers}). 

We begin with the following result.
\begin{lemma}\label{lemma betti numbers}
Let $(M,\omega)$ be a closed symplectic manifold of dimension $2n$ with index $k_0$, and let $\mathbf{b}=(b_0,\ldots,b_{2n})$ be the
vector of its even Betti numbers. Suppose that $(M,\omega)$ admits a Hamiltonian circle action with moment map $\psi:M\to \R$ with isolated fixed points and consider the integer $C(k_0,n,\mathbf{b})$ defined as 
$$
C(k_0,n,\mathbf{b}):=
\begin{cases}
\displaystyle\sum_{k=1}^{\frac{n}{2}}\Big[12k^2-n(k_0+1)\Big]b_{n-2k}- \frac{n}{2}(k_0+1)\,b_n\,, \quad \mbox{ for }n\mbox{ even}\\
& \\
\displaystyle\sum_{k=1}^{\frac{n-1}{2}}\Big[12k(k+1)+3-n(k_0+1)\Big]b_{n-1-2k}- \big[n(k_0+1)-3\big]b_{n-1}, \mbox{ for }n\mbox{ odd.}
\end{cases}
$$
If $(M,\omega,\psi)$ can be described by a multigraph $\Gamma=(V,E)$ that is positive w.r.t.\ $c_1^{S^1}$, then 
$C(k_0,n,\mathbf{b})$ is a \emph{non-negative} multiple of $k_0$. 

Moreover $C(k_0,n,\mathbf{b})$ vanishes if and only if $c_1^{S^1}[e]=k_0$ for all $e\in E$. 
\end{lemma}

\begin{proof}
This result is a straightforward generalization of \cite[Corollary 5.5]{GHS}. 

We first note that, under the assumption of this lemma, we have that  $c_1^{S^1}[e]$ is a positive multiple of the index $k_0\in \Z_{>0}$, for every $e\in E$.
 Indeed, by definition of index, there exists a (primitive) class $\eta\in H^2(M;\Z)$ such that $c_1=k_0\eta$. Since the restriction map $r\colon H_{S^1}^2(M;\Z)\to H^2(M;\Z)$
is surjective, we can consider an equivariant extension $\eta^{S^1}$ of $\eta$. Then 
\begin{equation}\label{c1 eta}
c_1^{S^1}=k_0 \eta^{S^1}+ \beta, \text{   for some   } \beta \in H^2_{S^1}(\{pt\};\Z)\,.
\end{equation}
 Since, by assumption, $\Gamma$ is positive w.r.t.\  $c_1^{S^1}$ and $k_0>0$, it follows that $\Gamma$ is positive w.r.t. $\eta^{S^1}$. Therefore, $\eta^{S^1}[e]$
 is a positive integer for every $e\in E$, and, from \eqref{c1 eta}, it follows that $c_1^{S^1}[e]$ is a positive multiple of $k_0$, for every $e\in E$. Therefore, 
 \begin{equation}\label{c1 edges}
 \sum_{e\in E}c_1^{S^1}[e]-k_0|E|=\sum_{e\in E}c_1^{S^1}[e]-k_0\frac{n}{2}\chi(M)
 \end{equation}
 is a non-negative multiple of $k_0$ and it is zero precisely if $c_1[e]=k_0$ for all $e\in E$. 
 
 Then, by \eqref{poincare dual alpha}, we have $\sum_{e\in E}c_1^{S^1}[e]=\int_M c_1c_{n-1}$ which, by \cite[Corollary 3.1]{GS}, can be expressed as a linear combination of even Betti numbers.
 Combining this expression with the definition of $\chi(M)$, it is easy to see that $C(k_0,n,\mathbf{b})$ is exactly $ \sum_{e\in E}c_1^{S^1}[e]-k_0|E|$, and the conclusion follows. 
\end{proof}

\begin{remark}\label{special values C}
The possible values of $C(k_0,4,\mathbf{b})$ are:
$$
C(k_0,4,\mathbf{b})=
\begin{cases}
6(6-b_4) & \text{if  }k_0=2\\
4(8-b_2-2b_4) & \text{if  }k_0=3\\
2(14-4b_2-5b_4) & \text{if  }k_0=4\\
12(2-b_2-b_4) & \text{if  }k_0=5
\end{cases}
$$
(see \cite[Table 1]{GHS}).
\end{remark}
We are now ready to prove the following result.
\begin{proposition}\label{index betti}
Let $(M,\omega)$ be an 8-dimensional,  closed, symplectic manifold admitting a Hamiltonian $S^1$-action with isolated fixed points and moment map $\psi\colon M \to \R$.
 Suppose that $(M,\omega,\psi)$ can be described by a multigraph that is positive with respect to $c_1^{S^1}$.  
If the index $k_0$ of $(M,\omega)$ is not $1$, then there are finitely many possibilities for the Betti numbers of $M$. More precisely,
\begin{enumerate}
\item If $k_0=5$ then $b_2=b_4=1$;
\item If $k_0=4$ then $b_2=1$ and $b_4=2$;
\item If $k_0=3$ then $(b_2,b_4)$ is either $(1,2)$ or $(2,3)$;
\item If $k_0=2$ then $1\leq b_2\leq b_4\leq 6$.
\end{enumerate}
\end{proposition}
\begin{proof}
Since $(M,\omega)$ is a closed symplectic $8$-manifold admitting a Hamiltonian $S^1$-action with isolated fixed points, its even Betti numbers are unimodal \cite[Theorem 1.3]{CK}, that is $b_{2}(M) \leq b_{4}(M)$. 
Then the result follows from Lemma \ref{lemma betti numbers} and the expressions for $C(k_0,4,\mathbf{b})$ in Remark \ref{special values C}. 

\end{proof}

\begin{remark}
Note  that, even when $k_{0}=1$, we can retrieve some non-trivial information  from Lemma \ref{lemma betti numbers}. 
Indeed $C(1,4,\mathbf{b})$ is given by $4(10+b_2-b_4)$, and then, if $(M,\omega,\psi)$ can be described by a multigraph that is positive w.r.t.\ $c_1^{S^1}$, Lemma \ref{lemma betti numbers} gives that 
$$b_{4}- b_{2} < 10.$$ 
Moreover, since the Betti numbers are unimodal \cite[Theorem 1.3]{CK}, we have $b_4-b_2\geq0$, and so  $b_4-b_2\in [0,9]$. Therefore, the (interesting) questions of whether there exist upper bounds for $b_{2},b_{4}$ or $\chi(M)$  are equivalent. 
\end{remark}

Let $(M,\omega)$ be a closed symplectic manifold which can be endowed with a Hamiltonian $T$-action with moment map $\psi$ and isolated fixed points. 
If $(M,\omega,\psi)$ admits a toric $1$-skeleton w.r.t.\ a multigraph $\Gamma$ describing $(M,\omega,\psi)$, and if $(M,\omega)$ is monotone, namely $c_1=\lambda [\omega]$ for some positive $\lambda$, 
then, by Proposition \ref{symplectic positive}, $\Gamma$ is also positive w.r.t.\ $c_1^T$, as the equation $c_1=\lambda [\omega]$ implies $c_1^T=\lambda [\omega-\psi]+ \beta$, for some
$\beta\in H_T(\{ pt \};\R)$. Since (Hamiltonian) GKM spaces always admit a toric $1$-skeleton, the above argument proves the following result.
\begin{cor}\label{cor GKM spaces betti}
Let $(M,\omega)$ be an 8-dimensional,  closed,  symplectic manifold which can be endowed with a Hamiltonian GKM action with moment map $\psi$. Assume that $(M,\omega)$ is monotone and that 
the index $k_0$ of $(M,\omega)$ is not $1$. Then there are finitely many possibilities for the Betti numbers of $M$. More precisely,
\begin{enumerate}
\item If $k_0=5$ then $b_2=b_4=1$;
\item If $k_0=4$ then $b_2=1$ and $b_4=2$;
\item If $k_0=3$ then $(b_2,b_4)$ is either $(1,2)$ or $(2,3)$. 
\item If $k_0=2$ then $1\leq b_2\leq b_4\leq 6$.
\end{enumerate}

\end{cor}

We will prove in Section~\ref{subsec:index3}  that there are no Hamiltonian GKM spaces of dimension $8$ with $k_0=3$ and $(b_2,b_4)=(1,2)$.

By Theorem \ref{thm:ChernNumbers}, if an $8$-dimensional,  closed, symplectic manifold with $b_2(M)=1$ admits a Hamiltonian $S^1$-action with isolated fixed points and has index $k_0\neq 1$, then,
for each value of $b_4:=b_4(M)$, there are finitely many possibilities for the Chern numbers of $(M,\omega)$. Combining this with Proposition~\ref{index betti}, we immediately get the following result.

\begin{cor}\label{finitely many chern numbers}
Let $(M,\omega)$ be an 8-dimensional, closed, connected, symplectic manifold admitting a Hamiltonian $S^1$-action with moment map $\psi$ and isolated fixed points. 
Suppose that $(M,\omega,\psi)$ can be described by
a multigraph that is positive with respect to $c_1^{S^1}$. 
If $b_{2}(M)=1$ and the index $k_0$ of $(M,\omega)$ is not $1$, then there are finitely many possibilities for the Chern numbers of $(M,\omega)$.
\end{cor}

From the argument preceding Corollary \ref{cor GKM spaces betti}, we immediately obtain the following result.

\begin{cor}\label{GKM chern numbers}
Let $(M,\omega)$ be a closed 8-dimensional, symplectic manifold which can be endowed with a Hamiltonian GKM action with moment map $\psi$. 
Assume that $(M,\omega)$ is monotone, that the index $k_0$ of $(M,\omega)$ is not $1$ and that $b_2=1$. 
Then there are finitely many possibilities for the Chern numbers of $(M,\omega)$.
\end{cor}

\begin{remark}
If the index of $(M,\omega)$ is equal to $2$ and $b_2=1$, then \eqref{eq:Silvia}, \eqref{eq:c_1^2c2} and \eqref{eq:ineq} in the proof of Theorem~\ref{main theorem chern}
 imply that 
$$
c_2^2[M] = 96 + b_4, \quad c_1^4[M]=16 k  \quad \text{and} \quad c_1^2c_2[M]=96 + 4 k\,.
$$ 
In Table~\ref{values of k} we can see the possible values of $k$, depending on the value of $b_4$. 
\begin{table}[h!]
\begin{center}
\begin{tabular}{| c || c |}
\hline
$b_4= 2 $   & $k \in \{18, 25, 32\}$ \\ \hline
$b_4=3 $ & $ k \in \{17, \ldots, 34\} $  \\ \hline
$b_4 = 4$  & $  k \in \{16, \ldots, 36\} $ \\ \hline
$b_4=5 $ & $k \in \{16, \ldots, 37\} $  \\ \hline
$b_4 = 6$  & $  k \in \{15, \ldots, 39\} $ \\ \hline
\end{tabular}
\end{center}
\label{values of k}
\caption{Index 2: Possible values of $k$ s.t. $c_1^4[M]=16 k$ and  $c_1^2c_2[M]=96 + 4 k$}
\end{table}
(Note that, since $k_0=2$, the Chern number $c_1^4[M]$ must be a multiple of $2^4$. Moreover, when $b_4=2$, we have $\chi(M)=6$ and so we can use the values in Table~\ref{spin}.)
 \end{remark}

\section{Examples}\label{section examples}

In this section we  explicitly describe the torus action and 
compute the equivariant and ordinary cohomology ring and Chern classes of some special 8-dimensional, positive
monotone symplectic manifolds, $V, W $ and $Q$ that  are exactly the examples  in our classification result, Main Theorem (see page \pageref{main theorem}). 
They are all smooth algebraic varieties with  holomorphic $\mathbf{T}:= (\mathbb{C}^*)^k$-actions  and, as we will see in Lemma~\ref{invariant kahler}, they also have K\"{a}hler structures  which  are invariant under the actions of $(S^1)^k < \mathbf{T}$ and for which the $(S^1)^k$ -actions are Hamiltonian.

In order to describe their equivariant cohomology rings and Chern classes, we will follow the strategy described in Section \ref{section Ham actions}: prove the existence of canonical classes and compute their equivariant structure constants. Similar computations will be done for the equivariant Chern classes. 

We begin with the following observation, which is useful when the manifold satisfies $b_2(M)=1$. This condition holds for the aforementioned
manifolds $V,W$ and $Q$.
\begin{lemma}\label{tau 1}
Let $(M,\omega,\psi)$ be a Hamiltonian $T$-space with $b_2(M)=1$ and discrete fixed point set $M^T$ and let $k_0$ be the index of $(M,\omega)$.
Consider a generic element $\xi$ of $\mathfrak{t}$
and the corresponding component of the moment map $\varphi:=\psi^\xi$.
Let $p_i$ be the only fixed point of index $i$ w.r.t.\ $\varphi$, for $i=0,1$. 
Let $\Gamma=(V,E)$ be a multigraph describing $(M,\omega,\psi)$ and let $e$ (resp.\ $e'$) be the only edge between $p_0$ and $p_1$ (resp.\ the only edge between $p_{n-1}$
and $p_n$) see Lemma \ref{multigraph b2=1}. 
Then, there exists a canonical class $\tau_1\in H^2_T(M;\Z)$ w.r.t. $\xi$ and it is given by 
$$
\tau_1=\frac{c_1^T(p_0)-c_1^T}{k_0}\,.
$$
Moreover,  $k_0=c_1^T[e]=c_1^T[e']$.
\end{lemma}
\begin{proof}
First  we observe that, by \cite[Lemma 3.1]{T}, a fixed point $P$ of index greater than 1 satisfies $\varphi(P)>\varphi(p_1)>\varphi(p_0)$. 
Since there is only one fixed point of index $1$, it follows that the Kirwan class $\tau_1$ at $p_1$ is  a canonical class. Since Kirwan classes are a basis of $H^*_T(M;\Z)$ as an
$H_T^*(pt;\Z)$-module, for degree reasons we have that 
\begin{equation}\label{eq c1T}
c_1^T=a\tau_1+b\,, \quad\text{for some }a\in \Z \text{ and }b\in H^2_T(\{pt\};\Z)\,.
\end{equation}
Since $\tau_1(p_0)=0$, it follows that $b=c_1^T(p_0)$. 

Let $-w_1$ be the weight in the negative normal bundle at $p_1$. Since $\tau_1(p_1)=w_1$ (see Remark \ref{our conventions cc}), and since there is an edge between $p_0$ and $p_1$,  \eqref{eq c1T} gives that 
$
a=-c_1^T[e]\,.
$

Since $r(\tau_1)$ generates $H^2(M;\Z)\simeq \Z$, it must be a primitive element, and by definition of index, it follows that $k_0=\lvert c_1^T[e] \rvert$. Since
it is easy to check that $c_1^T[e]$ is always positive, we have $k_0=c_1^T[e]$.  In order to prove that $c_1^T[e]=c_1^T[e']$, it is sufficient to consider the component of the moment map corresponding to $-\xi$. The details are left to the reader. 
\end{proof}

\begin{remark}\label{rmk:indexmultigraph}
If $(M,\omega, \psi)$ is as in Lemma~\ref{tau 1}, the symplectic form can be rescaled and the moment map can be translated in such a way that the weight sum formula in \eqref{wsf} holds. In this case, the  index $k_0$ of $(M,\omega)$ can be retrieved from the multigraph embedded in $\frak{t}^*$ through the moment map, since it is equal to the affine length of the edge $e$ from $p_0$ to $p_1$. 
\end{remark}

The next lemma explains the link between algebraic torus actions and Hamiltonian $T$-actions on smooth Fano varieties.

\begin{lemma} \label{invariant kahler}
Let $\mathbf{T}= (\mathbb{C}^*)^k$ be an algebraic torus acting algebraically on a smooth Fano variety $X$.  Then there exists a $(S^1)^k$-invariant K\"{a}hler form $\omega$ on $X$  satisfying $c_{1}(X) = [\omega]$,  for which the action of $(S^1)^k < \mathbf{T}$ is Hamiltonian.
\end{lemma}

\begin{proof}

 As the tangent bundle $TX$ of $X$ is $\mathbf{T}$-invariant, the anticanonical line bundle
$$
L=-K_X=\det TX=\wedge^n(TX)
$$
is also $\mathbf{T}$-invariant. Then there exists $m>0$ such that $L^{\otimes m}$ is $\mathbf{T}$-linearizable\footnote{A line bundle $L$ over a complex algebraic manifold $X$ with an algebraic $G$-action is called $G$-linearizable if the $G$-action on $X$  lifts to an action on $L$ which is linear on the fibers of $L$ \cite[Definition 3.2.3]{Br}.} (see \cite[Theorem 5.2.1]{Br}). Moreover, since $X$ is Fano, we have that $L$ is ample and so $L^{\otimes m}$ is also ample (as $(c_{1}(L^{\otimes m}))^\ell = m^\ell c_{1}(L)^\ell$ for each $\ell>0$). 

By \cite[Proposition 3.2.6]{Br}  we have that there exists a $\mathbf{T}$-equivariant embedding $i:X\to \mathbb{CP}^N$ into some projective space $\mathbb{CP}^N$, for some linear action on $\mathbb{CP}^N$ (Hamiltonian with respect to the Fubini-Study K\"ahler form).
Hence, $X$ has a $(S^1)^k$-invariant K\"ahler structure $\omega$ induced by the Fubini-Study K\"ahler form in $\mathbb{CP}^N$ and the $(S^1)^k$-action on $X$ is Hamiltonian.

The pull-back of the Fubini-Study metric satisfies $[\omega] = c_{1}(L^{\otimes \ell})$ for some integer $\ell>0$, therefore $[\omega] = \lambda c_{1}(X)$ for some positive constant $\lambda$. 
Since the form remains invariant and K\"{a}hler after rescaling by a positive constant, the result follows.

\end{proof}

\subsection{Index 2: The Fano-Mukai fourfold $V$ and its equivariant cohomology ring}\label{subsec V}

In this section, we describe an example of a closed, monotone  symplectic  manifold of dimension $8$ with index $2$, which is acted
on in a Hamiltonian way by a 2-dimensional torus with 6 fixed points. In particular,
we give the isotropy weights at each fixed point and compute its equivariant cohomology ring and Chern classes.

\subsubsection{Torus actions on the generalized flag variety $\Omega:= G_{2}/P$}
$\;$\\

Let $G_2$ be the complexification of the exceptional simple Lie group (of complex dimension $14$) and let $P$ be a parabolic subgroup (of complex dimension $9$) associated to a long root.
The generalized flag manifold 
$$
\Omega:=G_2/P \subset \mathbb{P}(\frak{g}_2)\subset \mathbb{CP}^{13}
$$
is known to be GKM for a $T^2$-action with GKM graph as in Figure~\ref{omegaweights} (see \cite{GHZ}). 
From this graph and from Lemma~\ref{invariant kahler} we obtain the following result.

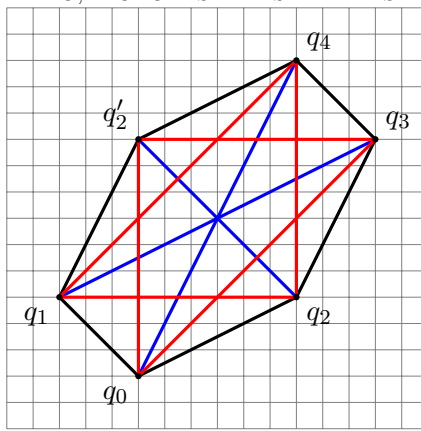
\begin{figure}[h]
 \begin{tikzpicture}[scale = 0.35]

\draw[step=1cm,gray,very thin] (-8,-8) grid (8,8);

\draw[very thick, black] (-6,-3) -- (-3,-6) -- (3,-3) -- (6,3) -- (3,6) -- (-3,3) -- (-6,-3);

\draw[very thick, blue] (-3,-6) -- (3,6);

\draw[very thick, blue] (-6,-3) -- (6,3);

\draw[very thick, blue] (-3,3) -- (3,-3);

\draw[very thick, red] (-6,-3) -- (3,6);

\draw[very thick, red] (-3,-6) -- (6,3);

\draw[very thick, red] (-6,-3) -- (3,-3);

\draw[very thick, red] (3,-3) -- (3,6);

\draw[very thick, red] (6,3) -- (-3,3);

\draw[very thick, red] (-3,-6) -- (-3,3);

\draw[fill] (-3,-6) circle [radius=0.1];

\draw[fill] (-6,-3) circle [radius=0.1];

\draw[fill] (3,-3) circle [radius=0.1];

\draw[fill] (6,3) circle [radius=0.1];

\draw[fill] (3,6) circle [radius=0.1];

\draw[fill] (-3,3) circle [radius=0.1];

\node [below left] (p_1) at (-6,-3) {$q_1$};

\node [below left] (p_1) at (-3,-6) {$q_0$};

\node [below right] (p_1) at (3,-3) {$q_2$};

\node [above left] (p_1) at (-3,3) {$q_2'$};

\node [above right] (p_1) at (6,3) {$q_3$};

\node [above right] (p_1) at (3,6) {$q_4$};

\end{tikzpicture}

\caption{GKM graph of  $\Omega$ in $\frak{t}^*$.}
\label{omegaweights}
\end{figure}


\begin{theorem} \label{gtwoflag}
There exists a K\"{a}hler form $\omega$ on $\Omega$ for which $c_{1}(\Omega) = [\omega]$ and a holomorphic $(\C^*)^2$-action on $\Omega$  such that the action of $T^2<(\C^*)^2$ is Hamiltonian with respect to $\omega$ and has $6$ isolated fixed points,  with weights given in Table~\ref{table:weightsOmega}.
\begin{table}[h]
\centering
\begin{tabular}{|c||l|}
\hline
      & \multicolumn{1}{c|}{Weights}       \\ \hline

$q_0$ & $ \{(2,1),(1,1), (0,1), (-1,1), (1,2)\}$      \\ \hline
$q_1$ & $  \{(1,-1),(1,0), (1,1), (1,2),(2,1)\} $      \\ \hline
$q_2$ & $ \{(1,2),(0,1), (-1,0), (-2,-1), (-1,1)\}   $      \\ \hline
$q_2'$ & $  \{(-1,-2),(0,-1), (1,0), (2,1),(1,-1)\}  $      \\ \hline

$q_3$ & $   \{(-1,1),(-1,0), (-1,-1), (-1,-2), (-2,-1) \} $      \\ \hline

$q_4$ & $   \{(-2,-1),(-1,-1), (0,-1), (1,-1), (-1,-2) \}$      \\ \hline
\end{tabular}\caption{Weights for the $T^2$-action on $\Omega$}\label{table:weightsOmega}
\end{table}
\end{theorem}
\noindent (See also  Section 5.3.1 in \cite{GHS}, and, in particular, Proposition 5.24, for more information on how to construct the GKM graph corresponding to a coadjoint orbit with $c_1=[\omega]$.)

\begin{remark}
Note that, by Remark~\ref{rmk:indexmultigraph},  the index $k_0=3$ of $\Omega$ can be retrieved from the  multigraph in Figure~\ref{omegaweights}, as we can choose a generic component of the moment map in a such a way that $q_0$ and $q_1$ are respectively the minimum and the index-$2$ critical point.
\end{remark}
\FloatBarrier

%

\FloatBarrier

\subsubsection{The Fano-Mukai $4$-fold $V$ and its torus action.}\label{sec:5.1.2}
$\;$\\

Let us now consider  smooth hyperplane sections 
$$
V_{18}:= \Omega \cap \mathbb{CP}^{12} \subset \mathbb{CP}^{13} 
$$
of $\Omega$, known as Fano-Mukai fourfolds of genus $10$ and degree $18$. As shown by Prokhorov and Zaidenberg \cite[Remark 13.4]{PZ}, the moduli space of these manifolds is $1$-dimensional. It contains a unique element $V_{18}^
a$ with an action of $\C\times \C^*$ and a unique element $V_{18}^s$ with an effective action of $GL(2,\C)$. A general element $V_{18}^g\notin\{V_{18}^a,V_{18}^s\}$ admits an action of $(\C^*)^2$ \cite[Theorem 1.3]{PZ} which is the restriction to $V_{18}^g$ of the $(\C^*)^2$-action on $\Omega$ \cite[Theorem 1.2]{PZ}.  Let $V$ be one of these general elements. Then we have the following result.


\begin{theorem}
\label{thm weights V}
There exists a K\"{a}hler form $\omega$ on $V$  for which $c_{1}(V) = [\omega]$ and a holomorphic $(\C^*)^2$-action, such that the action of $T^2<(\C^*)^2$ is Hamiltonian with respect to  $\omega$ and has $6$ isolated fixed points, with weights given in Table~\ref{table:weightsV}.

\begin{table}[h]
\centering
\begin{tabular}{|c||l|}
\hline
      & \multicolumn{1}{c|}{Weights}       \\ \hline

$q_0$ & $ \{(2,1),(1,1), (0,1), (-1,1)\}$      \\ \hline
$q_1$ & $  \{(1,-1),(1,0), (1,1), (1,2)\} $      \\ \hline
$q_2$ & $ \{(1,2),(0,1), (-1,0), (-2,-1)\}   $      \\ \hline
$q_2'$ & $  \{(-1,-2),(0,-1), (1,0), (2,1)\}  $      \\ \hline

$q_3$ & $   \{(-1,1),(-1,0), (-1,-1), (-1,-2) \} $      \\ \hline

$q_4$ & $   \{(-2,-1),(-1,-1), (0,-1), (1,-1) \}$      \\ \hline
\end{tabular}\caption{Weights for the $T^2$-action on $V$}\label{table:weightsV}
\end{table}
\end{theorem}
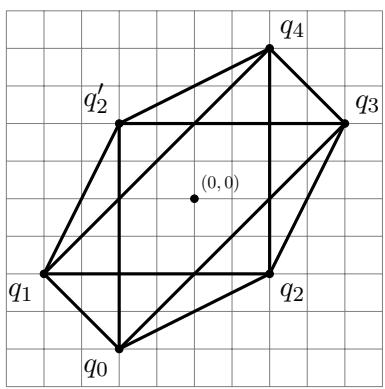
\begin{figure}[h] 

 \begin{tikzpicture}[scale = 0.5]

\draw[step=1cm,gray,very thin] (-5,-5) grid (5,5);

\draw[very thick, black] (-4,-2) -- (-2,-4) -- (2,-2) -- (4,2) -- (2,4) -- (-2,2) -- (-4,-2);

\draw[very thick, black] (-4,-2) -- (2,4);

\draw[very thick, black] (-2,-4) -- (4,2);

\draw[very thick, black] (-4,-2) -- (2,-2);

\draw[very thick, black] (2,-2) -- (2,4);

\draw[very thick, black] (4,2) -- (-2,2);

\draw[very thick, black] (-2,-4) -- (-2,2);

\draw[fill] (-2,-4) circle [radius=0.1];

\draw[fill] (-4,-2) circle [radius=0.1];

\draw[fill] (2,-2) circle [radius=0.1];

\draw[fill] (4,2) circle [radius=0.1];

\draw[fill] (2,4) circle [radius=0.1];

\draw[fill] (-2,2) circle [radius=0.1];

\draw[fill] (0,0) circle [radius=0.1];

\node [below left] (p_1) at (-4,-2) {$q_1$};

\node [below left] (p_1) at (-2,-4) {$q_0$};

\node [below right] (p_1) at (2,-2) {$q_2$};

\node [above left] (p_1) at (-2,2) {$q_2'$};

\node [above right] (p_1) at (4,2) {$q_3$};

\node [above right] (p_1) at (2,4) {$q_4$};

\node [above right, scale = 0.615] (p_1) at (0,0) {$(0,0)$};
\end{tikzpicture}

\caption{ The multigraph describing the $T^2$-action on $V$.}
\label{vfig1}
\end{figure}

\FloatBarrier

\begin{proof}

Let $\omega$ be the invariant K\"{a}hler form on $\Omega$ from Theorem \ref{gtwoflag}. Since $V \subset \Omega$ is a $T^2$-invariant symplectic submanifold of $\Omega$, it follows that $V$ is GKM for the restricted $T^2$-action and  that the corresponding GKM graph of $V$ is a subgraph of the GKM graph of $\Omega$. 
Consequently, to obtain this last GKM graph we just have to determine which edges in the GKM graph of $\Omega$ (corresponding to  $T^2$-invariant spheres in $\Omega$) are not in the GKM graph of $V$.

Note that any $T^2$-invariant sphere $S\subset \Omega$ that is not in $V$ must  intersect $V$ transversely in two fixed points. Moreover, as $\Omega$ is a Fano variety of index $3$, we have that $c_1(\Omega)=3\tau$, where $\tau$ is a generator of $H^2(\Omega,\Z)$. Since $V$ is a hyperplane section of the embedding associated to a line bundle with first Chern class equal to $\tau$, we have that $\tau$ is Poincar\'e dual to $[V]\in H_8(\Omega,\Z)$.   From the fact that $c_1=[\omega]$, we obtain
$$
S\cdot [V] = \int_S i^* \tau = \frac{1}{3} \int_S i^* \omega\,,
$$
where $i:S\to \Omega$ is the inclusion map. 
Therefore $S\cdot [V]=2$ if and only if $\int_S i^* \omega=6$.

Given a GKM graph of a Hamiltonian $T$-space, and considering its image in $\text{Lie}(T)^*$ via the moment map, using the ABBV formula it is easy to see that
the symplectic area of an invariant sphere corresponding to an edge $e$ is exactly the integral length of $e$ (i.e., the length of $e$ measured with respect to the dual lattice). 
From Figure~\ref{omegaweights}, we can see that the only $T^2$-invariant spheres with area equal to $6$ are exactly the ones corresponding to the blue edges. We conclude that the GKM graph of $V$ is the one obtained by deleting the blue edges in the GKM graph of $\Omega$ and the result follows.
\end{proof}

\begin{remark}
Note that, by Remark~\ref{rmk:indexmultigraph},  the index $k_0=2$ of $V$ can be retrieved from the  multigraph in Figure~\ref{vfig1}, as we can choose a generic component of the moment map in a such a way that $q_0$ and $q_1$ are respectively the minimum and the index-$2$ critical point.
\end{remark}

Let us now  compute the equivariant cohomology ring of $V$ and its Chern classes.
Let $x:=(1,0)$ and $y:=(0,1)$ be a basis of $\ell^*$, the dual lattice of $T$. Then the weights of the action  in Theorem \ref{thm weights V}, as well as the values of the equivariant
 first Chern class at the fixed points, are written in Table \ref{values c1T V}.

Note that we have chosen a moment map $\psi\colon V \to \mathfrak{t}^*$ so that $c_1^T=[\omega-\psi]$, thus implying the weight sum formula \eqref{wsf}. 

\renewcommand{\arraystretch}{1.2}
\begin{table}[h]
\begin{center}
  \begin{tabular}{| c || c  c  c  c | c | }
    \hline
     &  &  & & & $c_1^T(q)=-\psi(q)$ \\ \hline \hline
    $q_0$ & $-x+y$, & $y$, & $x+y$, & $2x+y$ & $2x+4y$ \\ \hline
    $q_1$ & $x-y$, & $x$, & $x+y$, & $x+2y$ & $4x+2y$\\ \hline
    $q_2$ &  $-2x-y$, &  $-x$, & $y$, & $x+2y$ & $-2x+2y$\\ \hline
   $q'_2$ & $-x-2y$, & $-y$, & $x$, & $2x+y$ & $2x-2y$\\ \hline
  $q_3$ &$-x-2y$, & $-x-y$, & $-x$, & $-x+y$ & $-4x-2y$\\ \hline
   $q_4$ & $-2x-y$, & $-x-y$, & $-y$, & $x-y$ & $-2x-4y$ \\ 
    \hline
  \end{tabular}
 \end{center}
  \caption{
Weights at the fixed points of the effective $T$-action on $V$ and the corresponding value of $c_1^T$.}
 \label{values c1T V}
\end{table}

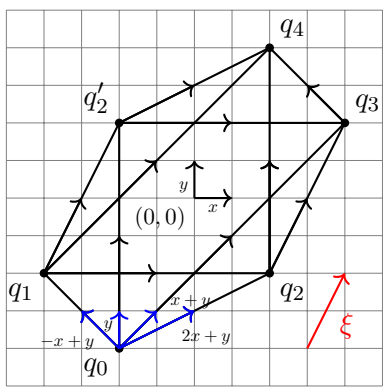
\begin{figure}[h] 

\begin{tikzpicture}[scale = 0.5]

\draw[step=1cm,gray,very thin] (-5,-5) grid (5,5);

\draw[thick, black] (-4,-2) -- (-2,-4) -- (2,-2) -- (4,2) -- (2,4) -- (-2,2) -- (-4,-2);

\draw [->, thick] (-2,-4) -- (-3,-3);

\draw [->, thick] (-2,-4) -- (0,-3);

\draw [->, thick] (2,-2) -- (3,0);

\draw [->, thick] (4,2) -- (3,3);

\draw [->, thick] (-2,2) -- (0,3);

\draw [->, thick] (-4,-2) -- (-3,0);

\draw[thick, black] (-4,-2) -- (2,4);

\draw [->, thick] (-4,-2) -- (-1,1);

\draw[ thick, black] (-2,-4) -- (4,2);

\draw [->, thick] (-2,-4) -- (1,-1);

\draw[ thick, black] (-4,-2) -- (2,-2);

\draw [->, thick] (-4,-2) -- (-1,-2);

\draw[thick, black] (2,-2) -- (2,4);

\draw [->, thick] (2,-2) -- (2,1);

\draw[ thick, black] (4,2) -- (-2,2);

\draw [->, thick] (-2,2) -- (1,2);

\draw [->, thick] (2,-2) -- (2,1);

\draw[ thick, black] (-2,-4) -- (-2,2);

\draw [->, thick] (-2,-4) -- (-2,-1);

\draw[fill] (-2,-4) circle [radius=0.1];

\draw[fill] (-4,-2) circle [radius=0.1];

\draw[fill] (2,-2) circle [radius=0.1];

\draw[fill] (4,2) circle [radius=0.1];

\draw[fill] (2,4) circle [radius=0.1];

\draw[fill] (-2,2) circle [radius=0.1];

\node [below left, thick, scale = 0.8] (p_1) at (0,0) {$(0,0)$};

\draw [->, thick] (0,0) -- (1,0);

\draw [->, thick] (0,0) -- (0,1);

\draw [->, thick, blue] (-2,-4) -- (-3,-3);

\draw [->, thick, blue] (-2,-4) -- (-2,-3);

\draw [->, thick, blue] (-2,-4) -- (-1,-3);

\draw [->, thick, blue] (-2,-4) -- (0,-3);

\draw [->, thick, red] (3,-4) -- (4,-2);
\node [above right, scale = 1, red] (p_1) at (3.6,-4) {$\xi$};

\node [left, scale = 0.6] (p_1) at (0,0.3) {$y$};

\node [below right, scale = 0.6] (p_1) at (0.2,0) {$x$};

\node [below left, scale = 0.6] (p_1) at (-2.5,-3.5) {$-x+y$};

\node [left, scale = 0.6] (p_1) at (-2,-3.4) {$y$};

\node [above right, scale = 0.6] (p_1) at (-0.8,-3.1) {$x+y$};

\node [right, scale = 0.6] (p_1) at (-0.5,-3.7) {$2x+y$};

\node [below left] (p_1) at (-4,-2) {$q_1$};

\node [below left] (p_1) at (-2,-4) {$q_0$};

\node [below right] (p_1) at (2,-2) {$q_2$};

\node [above left] (p_1) at (-2,2) {$q_2'$};

\node [above right] (p_1) at (4,2) {$q_3$};

\node [above right] (p_1) at (2,4) {$q_4$};
\end{tikzpicture}

\caption{The oriented GKM graph describing the action on $V$ and its weights.}
\label{vfig}
\end{figure}
Let $\xi:=(1,2)\in \mathfrak{t}$ be the generic vector that we use to orient the GKM graph of the $T$-action on $V$, as in Figure \ref{vfig}. 
Then it is easy to see that the GKM graph associated to this action is index-increasing, i.e.
$$
p\prec q \implies \lambda_p<\lambda_q \quad \text{for all the edges }e=(p,q) \text{ of the GKM graph.}
$$
Therefore, as an immediate consequence of \cite[Theorem 1.6]{GT}, we have that, for each fixed point $q$ in $V^T$, there exists a canonical class $\tau_q\in H^{2\lambda(q)}_T(V;\Z)$
(see Remark \ref{our conventions cc} for our conventions on canonical classes). Moreover,
its values at the fixed points can be computed by the recipe given in the same theorem. 
(In our case we will only need to use  \cite[Theorem 1.6]{GT} to compute the values of the restrictions of the canonical classes
at $q_2$ and $q'_2$, for reasons that will be clear from the following.)

Note that we have named the fixed points in such a way that $q_i$ has index $i$ and $q'_2$ has index $2$. 
Let $\tau_i\in H^{2i}_T(V;\Z)$ be the canonical class at $q_i$ and $\tau'_2\in H_T^4(V;\Z)$ the canonical class at $q'_2$.

Using Lemma \ref{tau 1} and the values of $c_1^T$ in Table \ref{values c1T V}, we obtain that the canonical class $\tau_1$ of $V$ (with respect to the chosen $\xi$) is given by
\begin{equation}\label{tau 1 V}
\tau_1=\frac{2x+4y - c_1^T}{2}\,.
\end{equation}
Its values at all fixed points can therefore be computed from Table \ref{values c1T V} and can be found in the first column of Table \ref{canonical classes V}.
Next we compute the restrictions to the fixed points of the canonical classes of degree 4, namely $\tau_2$ and $\tau'_2$. For these computations we follow the recipe in \cite[Theorem 1.6]{GT}.
In order to do so, for any $\eta\in \mathfrak{t}^*$, consider the projection $\rho_\eta\colon \mathfrak{t}^*\to \xi^{\perp}$ given by 
$$
\rho_\eta(\alpha):=\alpha-\displaystyle \frac{\langle \alpha, \xi \rangle }{\langle \eta , \xi \rangle}\eta
$$
and extend it to $\mathbb{S}(\mathfrak{t}^*)$, by imposing that it is an (algebra) endomorphism.  For every edge $e=(p,q)$ of the GKM graph of $V$ such that $\lambda(q)-\lambda(p)=1$, we need to compute the following
\begin{equation}\label{def theta}
\Theta(p,q)=\frac{\rho_{\eta(p,q)}(\Lambda_p^-)}{\rho_{\eta(p,q)}\left( \frac{\Lambda_q^-}{\eta(p,q)}\right)}
\end{equation}
which, by \cite[Theorem 1.6]{GT}, is a non-zero integer. Here $\eta(p,q)$ denotes the weight labeling the edge $(p,q)$ of the GKM graph.
According to the same theorem, in order to compute the restrictions to the fixed points of $\tau_2$ and $\tau'_2$, we need to compute $\Theta(q'_2,q_3)$, $\Theta(q_3,q_4)$ and $\Theta(q_2,q_3)$.
We begin with $\Theta(q'_2,q_3)$. One can check that $$\eta(q'_2,q_3)=x,\;\; \Lambda_{q'_2}^-=y(x+2y)\;\;\text{ and }\;\;\Lambda_{q_3}^-/x=(x+y)(x+2y).$$ 
Therefore, since $\xi=(1,2)$, using \eqref{def theta} it is easy to check that
$\Theta(q'_2,q_3)=1$. Analogously it can be checked that $\Theta(q_3,q_4)=1$ and $\Theta(q_2,q_3)=3$. To determine $\tau_2(q_3)$ it is enough to use \cite[Equation (4.1)]{GT}, which gives immediately that $\tau_2(q_3)=3x(x+y)$.
In order to compute $\tau_2(q_4)$ we observe that, using the notation in \cite[Theorem 1.2]{GT}, the only path from $q_2$ to $q_4$ is the one with edges $(q_2,q_3)$ and $(q_3,q_4)$. Therefore, by \cite[Equation (1.1)]{GT}, we have
$$
\tau_2(q_4)=\Lambda_{q_4}^-\cdot \frac{\psi(q_3)-\psi(q_2)}{\psi(q_4)-\psi(q_2)}\cdot \frac{\Theta(q_2,q_3)}{x+2y}\cdot \frac{\Theta(q_3,q_4)}{-x+y}=(2x+y)(x+y)\,.
$$

Similar computations can be done for $\tau'_{2}$, and its values at the fixed points can be found in Table \ref{canonical classes V}.

In order to compute the values of $\tau_3$, it is not necessary to use \cite[Theorem 1.6]{GT}. Indeed, since the only fixed points where $\tau_3$
does not vanish are $q_3$ and $q_4$, and there is a smooth invariant sphere in $V$ that has $q_3$ and $q_4$ as fixed points, the canonical class $\tau_{3}$
must be the equivariant Poincar\'e dual to this sphere and  we obtain the values in Table \ref{canonical classes V}. 
In this table we do not list the values of $\tau_0$ and $\tau_4$, as computing their restrictions to the fixed points is trivial (see Remark \ref{trivial cc}).
\renewcommand{\arraystretch}{1.1}
\begin{table}[h]
\begin{center}
  \begin{tabular}{| c || c | c | c | c |}
    \hline
               & $\tau_1$ & $\tau_2$ & $\tau'_2$ & $\tau_3$  \\ \hline \hline
    $q_0$            &  0                             & 0 &  0 &0  \\ \hline
    $q_1$            &    $-x+y$   &  0&  0 & 0  \\ \hline
    $q_2$            &      $2x+y$               &  $x(2x+y)$ & 0 & 0  \\ \hline
    $q'_2$           &  $3y$ &  $0$ & $y(x+2y)$ &0   \\ \hline
    $q_3$            &  $3x+3y$ &  $3x(x+y)$ & $(x+y)(x+2y)$  & $x(x+y)(x+2y)$  \\ \hline
    $q_4$            &  $2x+4y$ & $(2x+y)(x+y)$ & $3y(x+y)$ &$y(2x+y)(x+y)$   \\ \hline

      \end{tabular}
 \end{center}
  \caption{
Restrictions to the fixed points of the ``non-trivial'' canonical classes on $V$.}
 \label{canonical classes V}
\end{table}

\begin{theorem}\label{equiv cohomology and chern classes V}
Let $V$ be the Fano-Mukai 4-fold of genus 10 and degree 18 endowed with the Hamiltonian $T$-action described in Theorem \ref{thm weights V}. 
Let $x,y$ be a $\Z$-basis of the dual lattice $\ell^*$.
Then there exists a basis of the equivariant cohomology ring $H_T^*(V;\Z)$ given by canonical classes $\{\tau_i\}_{i=0}^4\cup \{\tau'_2\}$, with  $\tau_i\in H_T^{2i}(V;\Z)$, for  $i=0,\ldots,4$, and $\tau'_2\in H_T^4(V;\Z)$, satisfying
the following properties and relations:
\begin{align}
 \begin{aligned}\label{product V}
 & \tau_1^2= (-x+y)\tau_1+3\tau_2+3\tau'_2\,,  & \tau_1\tau_2 & = (2x+y)\tau_2+3\tau_3, \\
 & \tau_1\tau'_2= 3y \,\tau'_2+3\tau_3\,, & \tau_1\tau_3 & = 3(x+y)\tau_3+\tau_4, \\
 & \tau_2^2=x(2x+y)\tau_2+3x\,\tau_3+\tau_4\,,  & \tau_2\tau'_2 & =3(x+y)\tau_3, \\
& (\tau'_2)^2 = y(x+2y)\tau'_2+(x+2y)\tau_3 + \tau_4\,. \\ 
 \end{aligned}
 \end{align}
 Moreover, the equivariant Chern classes $c_j^T\in H_T^{2j}(V;\Z)$ of the tangent bundle satisfy 
 {\small \begin{align}
 \begin{aligned}\label{equiv chern classes V}
 & c_1^T=-2\tau_1+2x+4y\\
 & c_2^T=7(\tau_2+\tau'_2)-7(x+y)\tau_1-x^2+6xy+6y^2\\
 & c_3^T=-24 \tau_3 + 8(x+2y)\tau_2+8(2x+y)\tau'_2-2(3x^2+7xy+3y^2)\tau_1-2(x^3+x^2y-3xy^2-2y^3)\\
 & c_4^T=6 \tau_4-6(x+2y)\tau_3+(x^2+7xy+7y^2)\tau_2+(7x^2+7xy+y^2)\tau'_2-(x^3+5x^2y+5xy^2+y^3)\tau_1 + \\
 &\;\;\;\;\;\;\;\; (-x+y)y(x+y)(2x+y).\\
 \end{aligned}
 \end{align}}
\end{theorem}

\begin{proof}
We have already proved that canonical classes exist and their values at the fixed points are given in Table \ref{canonical classes V}. We can then  compute the coefficients of the equivariant Chern classes with respect to this basis  by the procedure described in Section \ref{restrictions and structure constants}. Similarly, the coefficients of the equivariant Chern classes can be computed from their restrictions to the fixed points, which just depend on the isotropy weights (see Remark \ref{remark chern classes}). The details are left to the reader.
\end{proof}

Corollary \ref{cohomology ring V}  is a consequence of Theorem \ref{equiv cohomology and chern classes V} and \eqref{structure constants}.


\subsection{Index 3: The quintic del Pezzo fourfold $W$ and its equivariant cohomology ring}\label{subsec W}
$\;$\\

Let us consider a del Pezzo $4$-fold of degree-$5$ obtained as the intersection 
$$
W=W_5 = Gr (2,5) \cap L
$$
of the Grassmanian $Gr (2,5)$ (under the Pl\"{u}cker embedding in $\mathbb{P} (\Lambda^2 \C^5)=\C \mathbb{P}^9$) and a codimension-$2$ projective subspace $L$ of $\C \mathbb{P}^9$. 

Identifying $Gr (2,5)$  with the set of lines in $\mathbb{CP}^4$, we see that there are two types of projective planes in $Gr (2,5)$. Namely, the vertex type planes $\sigma_{3,1}$ consisting of lines in $\mathbb{CP}^4$ through a given point and  contained in a fixed projective hyperplane of $\mathbb{CP}^4$, and the non-vertex type planes $\sigma_{2,2}$ consisting of lines contained in a fixed projective plane of $\mathbb{CP}^4$.

The manifold $W$ contains exactly one $\sigma_{2,2}$ plane $\Xi$ and a $1$-parameter family of $\sigma_{3,1}$ planes. Each of these planes intersects $\Xi$ along a line tangent to a fixed conic $C\subset \Xi$. Let $R$ be the union of all the $\sigma_{3,1}$ planes in $W$. It is shown in \cite[Proposition 3.4]{PZ0} that $R$ is the intersection of $W$  with a projective hyperplane and that $R$ is singular along $\Xi$. In particular, we have 
$$
C\subset \Xi \subset R\subset W,
$$
with $\Xi \simeq \mathbb{CP}^2$ and $\dim_{\C} R=3$. 

Let $Bl_{\Sigma} (W)$ be the blow up of $W$ along  $\Xi$. It is shown in \cite[Proposition 2.2]{PZ} that $Bl_{\Sigma} (W)$ can be identified with the blow up $Bl_{\Gamma}(\mathbb{CP}^4)$ of  $\mathbb{CP}^4$ along a twisted cubic curve $\Gamma\subset \mathbb{CP}^4$ and that we have the following commutative  diagram,
\vspace{.3cm}
\begin{equation} \label{diagram}\begin{tikzcd}
	& {\widetilde{H} \subset Bl_{\Xi}(W)\simeq Bl_{\Gamma}(\mathbb{CP}^4)\supset \widetilde{R}} \\
	{\Xi\subset R \subset W} && {\mathbb{CP}^4 \supset H =\langle \Gamma \rangle\supset\Gamma}
	\arrow["\rho", from=1-2, to=2-1]
	\arrow["\varphi"', from=1-2, to=2-3]
	\arrow["\tau"', dashed, from=2-1, to=2-3]
\end{tikzcd}\end{equation}

\vspace{.3cm}
\noindent where $\rho : Bl_{\Xi}(W) \to W$ and $\varphi:Bl_{\Gamma}(\mathbb{CP}^4) \to \mathbb{CP}^4$ are  the blow up maps and $\tau:W\subset L \simeq \mathbb{CP}^7 \to \mathbb{CP}^4$ is the projection with center $\Xi$ (a birational map). Let $\widetilde{H}:=\rho^{-1}(\Xi)$ and $\widetilde{R}:=\varphi^{-1}(\Gamma)$ be the corresponding exceptional divisors. Then  $\widetilde{H}$ can  be identified with the proper transform of the projective hyperplane $H =\langle \Gamma \rangle \simeq \mathbb{CP}^3$ spanned by $\Gamma$, and $\widetilde{R}$ can be identified with the proper transform of $R$  in $Bl_{\Xi}(W)$ (see \cite[Proposition 2.2]{PZ}). 

It is shown in \cite[Corollary 2.2.2]{PZ} that $W$ is a compactification of $\C^4$ and that
$$
W\setminus R \simeq Bl_{\Xi}(W) \setminus \left( \widetilde{R} \cup \widetilde{H} \right) \simeq \mathbb{CP}^4 \setminus H \simeq\mathbb{CP}^4 \setminus \mathbb{CP}^3 \simeq \C^4. 
$$
Moreover, by \cite[Corollary 2.2.4]{PZ}, the automorphism group $Aut(W)$ of $W$ leaves $\Xi$ invariant and can be identified with the group of automorphisms of $\mathbb{CP}^4$ that leaves $\Gamma$ invariant
$$
Aut(W)=Aut(W,\Xi)=Aut(\mathbb{CP}^4,\Gamma).
$$
Note that $R$ is also $Aut(W)$-invariant  and that the action of  $Aut(W)$ on $R$ is effective (cf \cite[Proposition 5.6]{PZ}). Let $T^2$ be a maximal torus in $Aut(W)$. In the following result we  show that the  action of $T^2$ on $W$ has $6$ fixed points and we compute the corresponding weights.
\begin{theorem}\label{Wweight}
Let $W$ be the quintic del Pezzo fourfold defined as above and let  $T^2$ be a maximal torus of  $Aut(W)$. Then the action of $T^2$ on $W$ has $6$ isolated fixed points and  the weights at these points  are listed in Table~\ref{table:weightsW}

\begin{table}[h!]
\centering
\begin{tabular}{|c||l|} 
\hline
      & \multicolumn{1}{c|}{Weights}       \\ \hline

$p_0$ & $ \{(3,0),(2,1),(1,2),(0,3)\}$      \\ \hline
$p_1$ & $ \{(-3,0),(-1,1),(1,2),(0,3)\}$      \\ \hline
$p_2$ & $ \{(3,0),(1,-1),(2,1),(0,-3)\}  $      \\ \hline
$p_2'$ & $  \{(-1,-2),(-2,-1),(-1,1),(-2,2)\} $      \\ \hline

$p_3$ & $    \{(-3,0),(1,-1),(-1,1),(0,-3)\}$      \\ \hline

$p_4$ & $   \{(-1,-2),(2,-2), (-2,-1), (1,-1)\}$      \\ \hline
\end{tabular}\caption{Weights for the $T^2$-action on $W$}\label{table:weightsW}
\end{table}

Moreover, this action is Hamiltonian with respect to  a K\"{a}hler form $\omega$ on $W$ satisfying $c_{1}(W) = [\omega]$. The image of the moment map, together with the invariant spheres, are given in Figure \ref{wfig}.

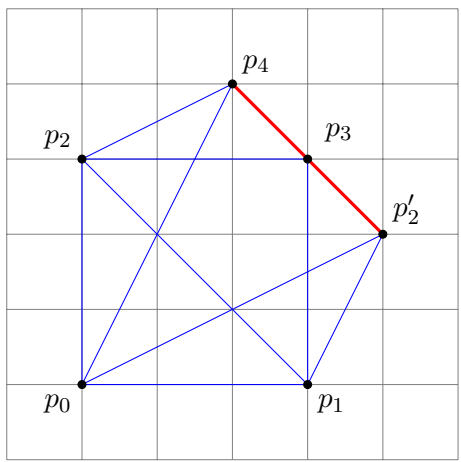
\begin{figure}[h] 
\centering
\begin{tikzpicture}[scale=1]

\draw[step=1cm,gray,very thin] (-3,-3) grid (3,3);
\draw[blue] (-2,1) -- (1,1);

\draw[blue]  (-2,1) -- (-2,-2);

\draw[blue]  (-2,-2) -- (1,-2);

\draw[blue]  (1,-2) -- (1,1);

\draw[blue]  (-2,-2) -- (0,2);

\draw[blue]  (-2,-2) -- (2,0);

\draw[blue]  (-2,1) -- (1,-2);

\draw[blue]  (-2,1) -- (0,2);

\draw[blue]  (1,-2) -- (2,0);

\draw[very thick, red]  (2,0) -- (0,2);


\filldraw [black] (2,0) circle (1.5pt);

\filldraw [black] (0,2) circle (1.5pt);

\filldraw [black] (1,-2) circle (1.5pt);

\filldraw [black] (-2,1) circle (1.5pt);

\filldraw [black] (1,1) circle (1.5pt);

\filldraw [black] (-2,-2) circle (1.5pt);

\node [below left] (p_1) at (-2,-2) {$p_0$};

\node [below right] (p_2) at (1,-2) {$p_1$};

\node [above left] (p_3) at (-2,1) {$p_2$};

\node [above right] (p_4) at (2,0) {$p_2'$};

\node [above right] (p_5) at (1.1,1.1) {$p_3$};

\node [above right] (p_6) at (0,2) {$p_4$};
\end{tikzpicture}

\caption{ The moment image of the invariant spheres in $W$. }

\label{wfig}
\end{figure}
\end{theorem}
\begin{proof}
 Up to conjugation, we can assume  $T^2$ to be the diagonal torus in $GL(2,\C)\subset Aut(W)$. It has a unique fixed point $p_0$ in 
$$
W\setminus R \simeq \C^4\simeq \mathbb{CP}^4\setminus H
$$
(cf. \cite[Lemma 5.7]{PZ}).
Up to a change in coordinates, the curve $\Gamma \subset \mathbb{CP}^4$ can be parametrized as
 \begin{equation}\label{eq:par}
 \phi([X:Y]) = [0:X^3:X^2Y : XY^2 : Y^3],
 \end{equation}
so that $H=\langle \Gamma \rangle$ is identified with the projective hyperplane $\{z_0=0\}$ (cf. \cite[Notation 5.1]{PZ}). In these coordinates, $p_0$ is identified with  $q_0=[1:0:0:0] \in \mathbb{CP}^4$ and $W\setminus R \simeq \C^4$ is identified with the set of points of the form $[1:z_1:z_2:z_3:z_4]$. As it is shown in \cite[Notation 5.1 and Lemma 5.7]{PZ}, the $T^2$-action on $W\setminus R$ is written in these coordinates as
\begin{equation}\label{eq:torusaction}
(\lambda, \mu)\cdot(z_1,z_2,z_3,z_4)=(\lambda^3z_1,\lambda^2\mu z_2, \lambda \mu^2 z_3,\mu^3 z_4).
\end{equation}
The corresponding action on $\mathbb{CP}^4\setminus E$ extends naturally to a $T^2$-action on $\mathbb{CP}^4$ with $5$ fixed points: $q_0$  and the points 
$$
q_1:=[0:1:0:0:0], \, q_2=[0:0:1:0:0],\, q_3=[0:0:0:1:0]\, \text{and} \, q_4=[0:0:0:0:1].
$$
Note that this action has  exactly two fixed points in $\Gamma$: the points $q_1$ and $q_4$. The weights of the $T^2$-action on $\mathbb{CP}^4$ at the five fixed points are shown in Table~\ref{table:weighsCP4}.
\begin{table}[h!]
\begin{tabular}{|c||l|}
\hline
      & \multicolumn{1}{c|}{Weights}       \\ \hline
$q_0$ & $\{(3,0),(2,1),(1,2),(0,3)\}$      \\ \hline
$q_1$ & $\{(-3,0),(-1,1),(-2,2),(-3,3)\}$  \\ \hline
$q_2$ & $\{(-2,-1),(1,-1),(-1,1),(-2,2)\}$ \\ \hline
$q_3$ & $\{(-1,-2),(2,-2),(1,-1),(-1,1)\}$ \\ \hline
$q_4$ & $\{(0,-3),(3,-3),(2,-2),(1,-1)\}$  \\ \hline 
\end{tabular} \caption{Weights for the $T^2$-action on $\mathbb{CP}^4$}
\label{table:weighsCP4}
\end{table}

Since $\Xi$ is $T^2$-invariant, we can perform the blow up equivariantly and obtain a natural $T^2$-action on the blow up $Bl_{\Xi}(W)$. Through the identification in \eqref{diagram}, we obtain a $T^2$-action on $Bl_{\Gamma}(\mathbb{CP}^4)$.

\begin{figure}[h]
\centering
\begin{tikzpicture} [scale=0.8]
\draw[blue] (0,0) -- (0,3);

\draw[blue] (0,0) -- (0,-4);

\draw[red] (0,0) -- (-3,-3);

\draw[blue] (0,0) -- (3,-3);

\filldraw [black] (0,0) circle (1.5pt);

\node [above left] (p_1) at (0,0) {$q_1$};

\filldraw [black] (0,3) circle (1.5pt);

\node [above right] (p_1) at (0,3) {$q_2$};

\filldraw [black] (3,-3) circle (1.5pt);

\node [above right] (p_1) at (3,-3) {$q_4$};

\filldraw [black] (-3,-3) circle (1.5pt);

\node [above left] (p_1) at (-3,-3) {$q_0$};

\filldraw [black] (0,-4) circle (1.5pt);

\node [above right] (p_1) at (0,-4) {$q_3$};

\node [right] (p_1) at (0,1.5) {$(-1,1)$};

\node [below right] (p_1) at (2,1) {$H$};

\node [below left] (p_1) at (1.6,-1.5) {$(-3,3)$};

\node [above left] (p_1) at (-1.5,-1.5) {$(-3,0)$};

\node [left] (p_1) at (0,-2) {$(-2,2)$};

  \draw[thick] (0,0) .. controls (0,3) and (2,-2) .. (3,-3);

\draw[blue] (10,0) -- (10,3);

\draw[blue] (10,0) -- (10,-4);

\draw[red] (10,0) -- (7,-3);

\draw[blue] (10,0) -- (13,-3);

\filldraw [black] (10,0) circle (1.5pt);

\filldraw [black] (10,3) circle (1.5pt);

\filldraw [black] (13,-3) circle (1.5pt);

\filldraw [black] (7,-3) circle (1.5pt);

\filldraw [black] (10,-4) circle (1.5pt);

\node [above left] (p_1) at (10,0) {$q_4$};

\node [above right] (p_1) at (10,3) {$q_3$};

\node [above right] (p_1) at (13,-3) {$q_1$};

\node [above left] (p_1) at (7,-3) {$q_0$};

\node [above right] (p_1) at (10,-4) {$q_2$};

\node [left] (above right) at (2,0) {$\Gamma$};

\node [right] (p_1) at (10,1.5) {$(1,-1)$};

\node [below left] (p_1) at (11.6,-1.5) {$(3,-3)$};

\node [above left] (p_1) at (8.5,-1.5) {$(0,-3)$};

\node [left] (p_1) at (10,-2) {$(2,-2)$};

\node [below right] (p_1) at (12,1) {$H$};

  \draw[thick] (10,0) .. controls (10,3) and (12,-2) .. (13,-3);

\node [left] (above right) at (12,0) {$\Gamma$};

\draw[red] (-2,-10) -- (2,-10);

\draw[red] (-2,-10) -- (0,-12.5);

\draw[blue] (2,-10) -- (0,-12.5);

\draw[blue] (0,-12.5) -- (0,-14.5);

\draw[red] (-2,-10) -- (-3.5,-12.5);

\draw[blue] (2,-10) -- (3.5,-12.5);

\draw[red] (-2,-10) -- (-2,-7);

\draw[blue] (2,-10) -- (2,-7);

\draw[blue,dashed] (0,-12.5) -- (0,-7);

\filldraw [black] (-2,-10) circle (1.5pt);
\filldraw [black] (2,-10) circle (1.5pt);
\filldraw [black] (0,-12.5) circle (1.5pt);
\filldraw [black] (0,-14.5) circle (1.5pt);
\filldraw [black] (-3.5,-12.5) circle (1.5pt);
\filldraw [black] (3.5,-12.5) circle (1.5pt);
\node [above left] (p_1) at (-3.5,-12.5) {$q_0$};
\node [above left] (p_1) at (0,-13) {$A$};
\node [above right] (p_1) at (2,-10) {$B$};
\node [above right] (p_1) at (0,-14.5) {$q_3$};
\node [above right] (p_1) at (3.5,-12.5) {$q_4$};

\node [left] (p_1) at (-2,-8.5) {$(-1,1)$};

\node [right] (p_1) at (2,-8.5) {$(-1,1)$};

\node [left] (p_1) at (0,-8.5) {$(-1,1)$};

\node [left] (p_1) at (-2.8,-11.2) {$(-3,0)$};

\node [right] (p_1) at (2.8,-11.2) {$(-3,3)$};

\node [left] (p_1) at (0,-13.5) {$(-2,2)$};

\node [above right] (p_1) at (0.5,-12.5) {$(-1,1)$};

\node [above left] (p_1) at (-0.5,-12.5) {$(-1,-2)$};

\node [above left] (p_1) at (2,-10) {$(0,-3)$};

\node [above right] (p_1) at (-2,-10) {$(0,3)$};

\node [below] (p_1) at (-1.8,-10.8) {$(1,2)$};

\node [below] (p_1) at (1.8,-10.8) {$(1,-1)$};

\node [above left] (p_1) at (-0.5,-12.5) {$(-1,-2)$};

\node  [right] (p_1) at (2.8,-10) {$\widetilde{H}$};

\filldraw [black] (0,-7) circle (1.5pt);

\filldraw [black] (-2,-7) circle (1.5pt);

\filldraw [black] (2,-7) circle (1.5pt);


\draw[red] (8,-10) -- (12,-10);

\draw[red] (8,-10) -- (10,-12.5);

\draw[blue] (12,-10) -- (10,-12.5);

\draw[blue] (10,-12.5) -- (10,-14.5);

\draw[red] (8,-10) -- (6.5,-12.5);

\draw[blue] (12,-10) -- (13.5,-12.5);

\draw[red] (8,-10) -- (8,-7);

\draw[blue] (12,-10) -- (12,-7);

\draw[blue,dashed] (10,-12.5) -- (10,-7);

\filldraw [black] (8,-10) circle (1.5pt);
\filldraw [black] (12,-10) circle (1.5pt);
\filldraw [black] (10,-12.5) circle (1.5pt);
\filldraw [black] (10,-14.5) circle (1.5pt);
\filldraw [black] (6.5,-12.5) circle (1.5pt);
\filldraw [black] (13.5,-12.5) circle (1.5pt);
\node [above left] (p_1) at (6.5,-12.5) {$q_0$};
\node [above left] (p_1) at (10,-13) {$C$};
\node [above right] (p_1) at (12,-10) {$D$};
\node [above right] (p_1) at (10,-14.5) {$q_2$};
\node [above right] (p_1) at (13.5,-12.5) {$q_1$};

\filldraw [black] (10,-7) circle (1.5pt);

\filldraw [black] (8,-7) circle (1.5pt);

\filldraw [black] (12,-7) circle (1.5pt);

\node [left] (p_1) at (8,-8.5) {$(1,-1)$};

\node [right] (p_1) at (12,-8.5) {$(1,-1)$};

\node [left] (p_1) at (10,-8.5) {$(1,-1)$};

\node [left] (p_1) at (7.2,-11.2) {$(0,-3)$};

\node [right] (p_1) at (12.8,-11.2) {$(3,-3)$};

\node  [right] (p_1) at (12.8,-10) {$\tilde{H}$};

\node [left] (p_1) at (10,-13.5) {$(2,-2)$};

\node [above right] (p_1) at (10.5,-12.5) {$(1,-1)$};

\node [above left] (p_1) at (9.5,-12.5) {$(-2,-1)$};

\node [above left] (p_1) at (12,-10) {$(-3,0)$};

\node [above right] (p_1) at (8,-10) {$(3,0)$};

\node [below] (p_1) at (8.3,-10.8) {$(2,1)$};

\node [below] (p_1) at (11.8,-10.9) {$(-1,1)$};

\end{tikzpicture}
\caption{Determining the weights on the blow-up $Bl_{\Gamma}(\mathbb{CP}^4)$ along $\Gamma$: The two fixed points  $q_1,q_4\in \mathbb{CP}^4$ that are in $\Gamma\subset \mathbb{CP}^4$  are shown on  top. The edges corresponding to $T^2$-invariant spheres in $H \simeq \mathbb{CP}^3$ are colored blue, and those that   are not in $H$ are colored red.  At the bottom, we have the weights at the fixed points in $Bl_{\Gamma}(\mathbb{CP}^4)$. The edges corresponding to invariant spheres in the proper transform $\widetilde{H}  \simeq Bl_{\Gamma}(\mathbb{CP}^3)$ are colored blue,  and those that are not in  $\widetilde{H}$  are colored red.}
\label{fig:weights3}
\end{figure}
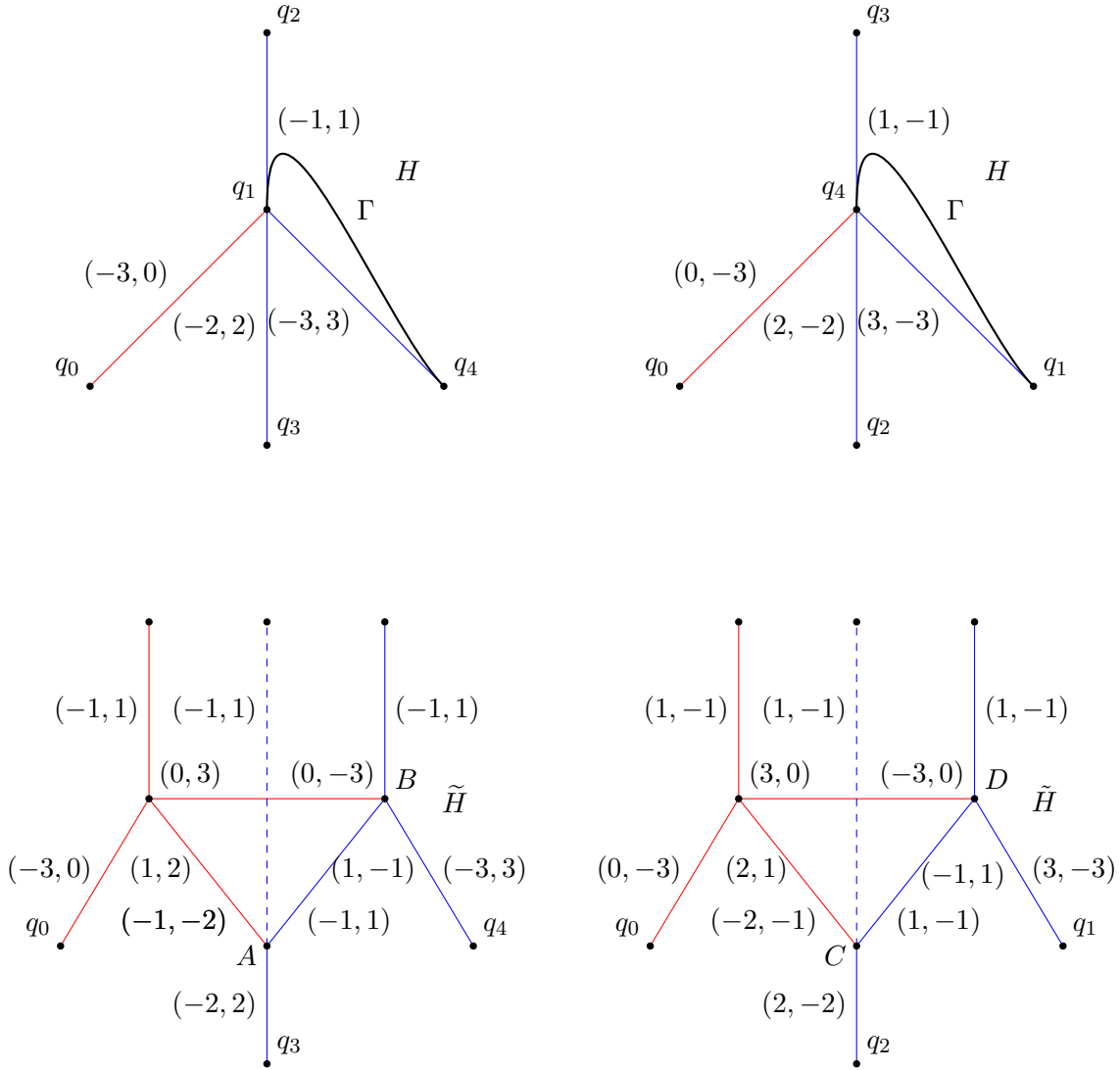

Note that this action has nine fixed points and that six of these points are on the exceptional divisor $\widetilde{R}=\varphi^{-1}(\Gamma)$. The weights at these fixed points are presented in Figure~\ref{fig:weights3}. Of these six points only the points represented by $A,B,C$ and $D$ in Figure~\ref{fig:weights3} are in the proper transform  $\widetilde{H}\simeq \rho^{-1}(\Xi)$ of $H$ and so  the points in $Bl_\Xi(W)$ corresponding to the remaining two fixed points in $\widetilde{R}$ are  not in the exceptional divisor of $Bl_\Xi(W)$. Consequently, they are also fixed points for the $T^2$-action on $W$ and their  weights  (for the $T^2$-action on $W$) remain the same after the blow down to $W$. Note that the same is true for $p_0$ (the fixed point in $W$ corresponding to $q_0$), since this point is not in the exceptional divisor $\varphi^{-1}(\Gamma)$, as $q_0\notin  \widetilde{H}$.

Hence, we already have  the weights of the $T^2$-action on $W$ at  three  fixed points. We will now determine  the weights  at the remaining fixed points (i.e. those that are in $\Xi$). These have to be obtained from the weights at the fixed points $A,B,C$ and $D$ by contracting the exceptional divisor.  By Lemma~\ref{lemma_blowup}, we just have to see which pairs of points in $\{A,B,C,D,q_2,q_3\}$ have two weights in common  and one weight that is symmetric. We can easily see that the only possibility is to identify  $A$ with $q_2$, $B$ with $D$ and $C$ with $q_3$ in the blow down. 

Again by Lemma~\ref{lemma_blowup} we see that the weights of the $T^2$-action at the resulting fixed points in $W$ are 
\begin{align*}
\{(-1,-2),(-2,-1),& (-1,1),(-2,2)\},  \quad \{(-3,0),(1,-1),(-1,1),(0,-3)\} \quad \text{and} \\ &  \{(-1,-2),(2,-2),(-2,-1),(1,-1)\}.
\end{align*}
We conclude that the weights at the fixed points for the $T^2$-action on $W$ are those in the statement of the theorem.

Finally, since $W$ is a Fano variety, we have by Lemma \ref{invariant kahler} that there exists a $T^2$-invariant K\"{a}hler form  $\omega$ for which the  $T^2$-action  is Hamiltonian and $c_{1} = [\omega]$. The corresponding moment map image is shown in Figure \ref{wfig}.
\end{proof}

\begin{lemma}\label{lemma_blowup} Let $M$ be an $8$-dimensional, symplectic manifold equipped with 
 a Hamiltonian $T^2$-action,  and let $\Sigma \subset M$ be a $T^2$-invariant $4$-dimensional submanifold of $M$. 
 Let $P\in \Sigma$ be a fixed point for the $T^2$-action on $M$ and let $w_1,w_2,w_3,w_4$ be the weights for this action at $P$, where $w_1,w_2$ are the weights  on $T_P\Sigma$ and $w_3,w_4$ are the weights on the normal bundle of $\Sigma$. 
 
If $Bl_\Sigma(M)$ is a $T^2$-equivariant blow up of $M$ along $\Sigma$ with blow up map $\pi:Bl_\Sigma(M) \to M$, then the weights of the induced $T^2$-action at the two fixed points in the fiber $\pi^{-1}(P)$ are 
 $$
 \{w_1,w_2,w_3,w_4-w_3\}\quad \text{and} \quad  \{w_1,w_2,w_4,w_3-w_4\}.
 $$
\end{lemma}
\begin{proof}
Since the exceptional divisor $\widetilde{\Sigma}$ is the projectivization of the normal bundle of $\Sigma$, the fiber $\pi^{-1}(P)$ is a $T^2$-invariant $\mathbb{CP}^1$ where $T^2$ acts by 
$$
\lambda \cdot [z_3:z_4] =  [\lambda^{w_3}z_3:\lambda^{w_4}z_4] =  [z_3: \lambda^{w_4-w_3 }z_4],\quad \lambda \in T^2.
$$
(Here, for $\lambda=(\lambda_1,\lambda_2)\in T^2$ and a weight $w=(w^1,w^2)$, we write $\lambda^{w} :=\lambda_1^{w^1}\lambda_2^{w^2}$). 
\end{proof}
\begin{figure}[h] \label{BU}
\centering
\includegraphics[width=0.5\textwidth]{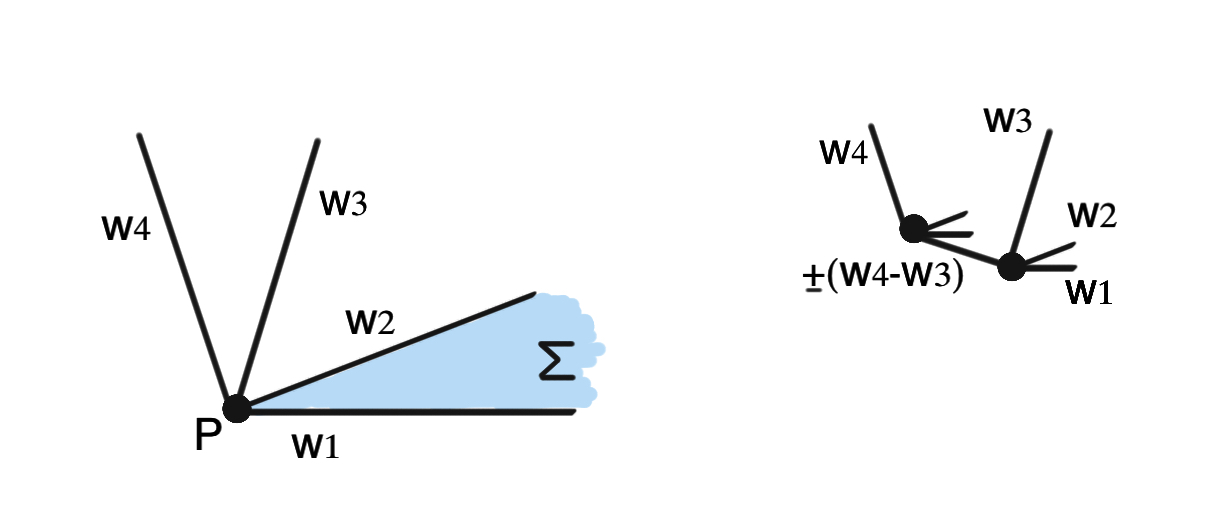}
\caption{Weights and Blow up along $\Sigma$}
\end{figure}
\color{black}

As it is easily seen, the action of the torus $T$ on $W$ is not effective. Indeed, the subgroup $\Z_3$ of the diagonal circle in $T$ acts trivially on $W$, as it acts trivially
on an open neighborhood of $p_0$. 
In the following we compute the weights of the effective action of $\mathbb{T}:=T/\Z_3$. \label{eff action}

Let $\xi_1:=(1,0)$ and $\xi_2:=(0,1)$ be a basis of the lattice $\ell$ of $T$ and let $x=(1,0)$, $y=(0,1)$ be the dual basis. In this notation, the $\Z_3$ that acts trivially is exactly $\exp\Z\Big\langle\displaystyle \frac{\xi_1+\xi_2}{3}\Big\rangle$.
The lattice of $\mathbb{T}$ is then given by $\widetilde{\ell}:=\Z\langle \xi_1, \frac{\xi_1+\xi_2}{3}\rangle$. A basis of the dual lattice $\widetilde{\ell}^*$ is given, for example, by 
$$\alpha:=-x+y\quad \text{and} \quad\beta:=3y\,.$$ Observe that, in this new notation, the weights of the $\mathbb{T}$-action at $p_0$ are given by 
$$
3x=-3\alpha+\beta, \;\;3y=\beta,\;\; 2x+y=-2\alpha+\beta, \;\; x+2y=-\alpha+\beta\,.
$$
These weights and those at the remaining fixed points are listed in Table~\ref{values c1T}.

\renewcommand{\arraystretch}{1.2}
\begin{table}[H]
\begin{center}
  \begin{tabular}{| c || c  c  c  c | c | }
    \hline
     &  &  & & & $c_1(P)=-\psi(P)$ \\ \hline \hline
    $p_0$ & $-3\alpha+\beta$, & $-2\alpha+\beta$, & $-\alpha+\beta$, & $\beta$ & $-6\alpha+4\beta$ \\ \hline
    $p_1$ & $3\alpha-\beta$, & $\alpha$, & $-\alpha+\beta$, & $\beta$ & $3\alpha+\beta$\\ \hline
    $p_2$ &  $-3\alpha+\beta$, &  $-\alpha$, & $-2\alpha+\beta$, & $-\beta$ & $-6\alpha+\beta$\\ \hline
   $p'_2$ & $\alpha-\beta$, & $2\alpha-\beta$, & $\alpha$, & $2\alpha$ & $6\alpha-2\beta$\\ \hline
  $p_3$ &$3\alpha-\beta$, & $-\alpha$, & $\alpha$, & $-\beta$ & $3\alpha-2\beta$\\ \hline
   $p_4$ & $\alpha-\beta$, & $-2\alpha$, & $2\alpha-\beta$, & $-\alpha$ & $-2\beta$ \\ 
    \hline
  \end{tabular}
 \end{center}
 \caption{Weights at the fixed points of the effective $\mathbb{T}$-action on $W$ and the corresponding value of $c_1^{\mathbb{T}}$.} \label{values c1T}
\end{table}

In what follows we prove the existence of canonical classes with respect to a certain choice of generic $\xi\in \mathfrak{t}$ (see Section \ref{section examples}). 
 
Choose a generic $\xi$ in the cone determined by the inequalities $\langle \beta, \xi \rangle>0$ and $\langle \alpha, \xi \rangle >0$, and so that the $\xi$-component of the moment
map $\varphi:=\psi^\xi$ satisfies
$p_2\prec p'_2$: It can be easily checked, that
this means choosing a $\xi$ in the above cone and on the right of the line through $p_0$ with direction $x+4y$.
These choices, as well as the corresponding orientation on the multigraph describing $W$, are  illustrated in Figure \ref{multigraph W oriented}.

\begin{figure}[H]
\centering
\begin{tikzpicture}[scale=1.4]

\draw [red, fill=blue!42,ultra thin] (-2,-2) -- (-2,2) -- (2,2) -- (-2,-2);

\draw[step=0.33333 cm,gray,very thin] (-3,-3) grid (3,3);

\draw [->, thick] (-2,-2) -- (-2,-0.5);

\draw [->, thick] (-2,-2) -- (-1,0);

\draw [->, thick] (-2,-2) -- (0,-1);

\draw [->, thick] (-2,-2) -- (-0.5,-2);

\draw [->, thick] (-2,-2) -- (-2,-0.5);

\draw [thick](-2,1) -- (1,1);

\draw [->, thick] (-2,1) -- (-1,1);

\draw [thick] (-2,-2) -- (-2,1);

\draw [->, very thick, blue] (-2,-2) -- (-2,-1);

\draw [thick] (-2,-2) -- (1,-2);

\draw [->, very thick, blue] (-2,-2) -- (-1,-2);

\draw [thick]  (1,-2) -- (1,1);

\draw [thick]  (1,-2) -- (1,1);

\draw [- >, thick]  (1,-2) -- (1,-0.7);

\draw [thick] (-2,-2) -- (0,2);

\draw [->,very thick, blue] (-2,-2) -- (-1.6666,-1.3333);

\draw  [thick] (-2,-2) -- (2,0);

\draw [->,very thick, blue] (-2,-2) -- (-1.3333,-1.6666);

\draw [thick]  (-2,1) -- (1,-2);

\draw [->, thick]  (1,-2) -- (-0.5,-0.5);

\draw [thick] (-2,1) -- (0,2);

\draw [->,thick] (-2,1) -- (-1,1.5);

\draw [thick] (1,-2) -- (2,0);

\draw [->,thick] (1,-2) -- (1.5,-1);

\draw [thick]  (2,0) -- (0,2);

\draw [->,thick] (2,0) -- (1.5,0.5);

\draw [->,thick] (1,1) -- (0.5,1.5);

\draw [->,very thick, red ] (-2,-2) -- (-1,-0.66);

\node [above right ] (p_1) at (-1,-0.66) {$\xi$};



\draw [->,very thick, blue] (-2,-2) -- (-2.3333, -1.6666);

\filldraw  (2,0) circle (1.5pt);

\filldraw  (0,0) circle (1.5pt);

\filldraw  (0,2) circle (1.5pt);

\filldraw  (1,-2) circle (1.5pt);

\filldraw  (-2,1) circle (1.5pt);

\filldraw  (1,1) circle (1.5pt);

\filldraw (-2,-2) circle (1.5pt);

\node [below right, scale = 0.8] (p_1) at (0,0) {$(0,0)$};

\draw [->] (0,0) -- (0.3333,0);

\draw [->] (0,0) -- (0,0.3333);

\node [left, scale = 0.6] (p_1) at (0,0.1666) {$y$};

\node [above, scale = 0.6] (p_1) at (0.17,0) {$x$};

\node [below left] (p_1) at (-2,-2) {$p_0$};

\node [below right] (p_2) at (1,-2) {$p_1$};

\node [above left] (p_3) at (-2,1) {$p_2$};

\node [above right] (p_4) at (2,0) {$p_2'$};

\node [above right] (p_5) at (1.1,1.1) {$p_3$};

\node [above right] (p_6) at (0,2) {$p_4$};

\node [below, scale = 0.8] (p_1) at (-1,-2) {$-3\alpha + \beta$};

\node [below, scale = 0.8] (p_1) at (-2.3333,-1.8) {$\alpha$};

\node [below right, scale = 0.8] (p_1) at (-1.3333,-1.6) {$-2 \alpha + \beta$};

\node [below right, scale = 0.8] (p_1) at (-2.1,-0.6) {$-\alpha + \beta$};

\node [left ,scale = 0.8] (p_1) at (-2,-1.5) {$\beta$};

\end{tikzpicture}

\caption{Oriented multigraph describing the action on $W$.}
\label{multigraph W oriented}

\end{figure}
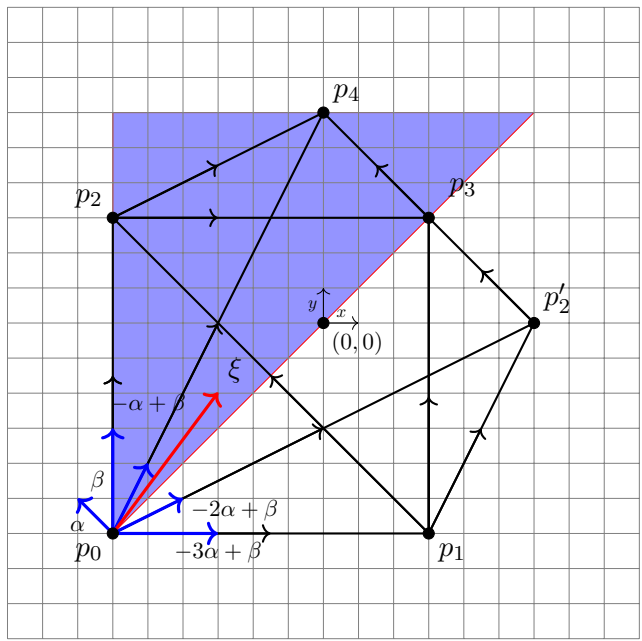
\begin{remark}
Note that, by Remark~\ref{rmk:indexmultigraph},  the index $k_0=3$ of $W$ can be retrieved from  Figure~\ref{multigraph W oriented}. Indeed, we can see in this figure that the edge connecting $p_0$ to $p_1$ has affine length equal to $3$. 
\end{remark}

\subsubsection{The canonical class $\tau_1\in H^2_{\mathbb{T}}(W;\Z)$:}
Using Lemma \ref{tau 1} and the values of $c_1^\mathbb{T}$ in Table \ref{values c1T}, we obtain that the canonical class $\tau_1$ of $W$ at $p_1$ (with respect to the chosen $\xi$) is given by
\begin{equation}\label{tau 1 W}
\tau_1=\frac{-6\alpha+4\beta - c_1^\mathbb{T}}{3}\,.
\end{equation}
Its values at all fixed points can therefore be computed from Table \ref{values c1T} and can be found in Table \ref{canonical classes W}.

\renewcommand{\arraystretch}{1.2}
\begin{table}[h]
\begin{center}
  \begin{tabular}{| c || c | c | c | c |}
    \hline
               & $\tau_1$ & $\tau_2$ & $\tau'_2$ & $\tau_3$  \\ \hline \hline
    $p_0$            &  0                             & 0 &  0 &0  \\ \hline
    $p_1$            &    $-3\alpha+\beta$   &  0&  0 & 0  \\ \hline
    $p_2$            &      $\beta$               &  $\alpha \beta$ & 0 & 0  \\ \hline
    $p'_2$           &  $-4\alpha + 2 \beta$ &  $0$ & $(-2\alpha+\beta)(-\alpha+\beta)$ &0   \\ \hline
    $p_3$            &  $-3\alpha+2\beta$ &  $\alpha \beta$ & $(-3\alpha+\beta)\beta$  & $\alpha\beta(-3\alpha+\beta)$  \\ \hline
    $p_4$            &  $-2\alpha+2\beta$ & $2\alpha(-\alpha+\beta)$ & $(-2\alpha+\beta)(-\alpha+\beta)$ &$2\alpha(-2\alpha+\beta)(-\alpha+\beta)$   \\ \hline

      \end{tabular}
 \end{center}

  \caption{
Restrictions to the fixed points of the ``non-trivial'' canonical classes on $W$.}
 \label{canonical classes W}
\end{table}

\subsubsection{The canonical classes $\tau_2$ and $\tau'_2$ in $H^4_{\mathbb{T}}(W;\Z)$:}
Consider Kirwan classes $\widetilde{\tau}_{2}$ at $p_2$ and $\widetilde{\tau}'_2$ at $p'_2$, which are elements of $H_\mathbb{T}^4(W;\Z)$. 
Since $\xi$ is chosen so that $p_2 \prec p'_2$, by the properties of the Kirwan classes we have:
 
 \begin{equation}\label{tau2W}
\begin{cases}
 \widetilde{\tau}_{2}(p_2)=\displaystyle\Lambda_{p_2}^-\\
  \widetilde{\tau}_{2}(p'_2) \text{ is not necessarily zero}
 \end{cases}
 \quad 
  \begin{cases}
 \widetilde{\tau}'_{2}(p'_2)=\displaystyle\Lambda_{p'_2}^-\\
  \widetilde{\tau}'_{2}(p_2) =0\,.
 \end{cases}
 \end{equation}
 Since the only points which, in the ordering induced by $\varphi$, lie above $p'_2$ are $p_3$ and $p_4$, we conclude that $\widetilde{\tau}'_{2}=:\tau'_2$ is the canonical class at $p'_2$. 
 What we want to prove is that, given a Kirwan class $\widetilde{\tau}_{2}$, we can modify it so that it is zero at $p'_2$, obtaining the canonical class $\tau_2$. 
 
 As $\widetilde{\tau}_{2}$ is a Kirwan class, $p_0\prec p_2$ and $p_1\prec p_2$, we have that $\widetilde{\tau}_{2}(p_0)=\widetilde{\tau}_{2}(p_1)=0$. At $p'_2$, there is only one weight of the $\mathbb{T}$-action multiple of $2\alpha-\beta$ (namely $2\alpha-\beta$), and only one weight multiple of $\alpha-\beta$ (namely $\alpha-\beta$, see Figure \ref{multigraph W oriented}).
 This implies that the connected components of the isotropy submanifolds $M^{2\alpha-\beta}$ and $M^{\alpha-\beta}$ containing $p'_2$ are spheres, which we call respectively $S_0^2$ and $S_1^2$.
The induced action of $\mathbb{T}$
 on $S^2_0$ is Hamiltonian and has exactly two fixed points, $p'_2$ and $q$. Since the image of the moment map for the $\mathbb{T}$-action on $S^2_0$ must be on the line
 $p'_2+\R\langle 2\alpha - \beta \rangle$, it is easy to see from Figure \ref{multigraph W oriented} that $q=p_0$.
 For degree reasons, the integral $\int_{S^2_0}\widetilde{\tau}_2$ must be a polynomial of degree one in $\Z[\alpha, \beta]$. 
 Therefore the ABBV formula applied to the submanifold $S^2_0$, together with the fact that $ \widetilde{\tau}_2(p_0)=0$, gives 
\begin{equation}\label{eq:*}
 \widetilde{\tau}_2(p'_2)=a (-2\alpha+\beta)\,\quad\text{for some  }a\in \Z[\alpha,\beta]\,.
\end{equation}
 By repeating the same argument for the invariant sphere $S^2_1$, we see that the $\mathbb{T}$-fixed points on $S^2_1$ are $p'_2$ and $p_1$, and that there exists an element $b$ of degree one in $\Z[\alpha, \beta ]$ such that 
\begin{equation}\label{eq:**}
 \widetilde{\tau}_2(p'_2)=b (-\alpha+\beta)\,.
\end{equation}
 Since $-2\alpha+\beta$ and $-\alpha+\beta$ are coprime, if regarded as polynomials in $\Z[\alpha,\beta]$,  we deduce, from  \eqref{eq:*} and \eqref{eq:**}, that there exists $k\in \Z$ such that 
 $$
 \widetilde{\tau}_2(p'_2)=k(-2\alpha+\beta)(-\alpha+\beta)=k \Lambda_{p'_2}^-\,.
 $$
 It is then easy to see that 
 $$
 \tau_2:=\widetilde{\tau}_2-k \widetilde{\tau}'_2
 $$
 is the canonical class at $p_2$. Let us now compute the restriction of $\tau_2$ at all the fixed points of $W$. Since $\tau_2$ is a canonical class, the values that we already know are:
 \begin{equation}\label{first values tau2}
 \tau_2(p_0)=\tau_2(p_1)=\tau_2(p'_2)=0\quad \text{and}\quad \tau_2(p_2)=\Lambda_{p_2}^-=\alpha \beta\,.
 \end{equation}
 What is left to compute is $\tau_2(p_3)$ and $\tau_2(p_4)$\,.
 Consider the connected component of $M^\beta$ containing $p_3$. This is an invariant sphere in $M$ and, by the same argument as above, since $\tau_2(p_1)=0$,
  we must have that $\tau_2(p_3)$ is a multiple of $\beta$. 
Considering the connected component of the isotropy submanifold $M^\alpha$ containing 
 $p'_2$, one can easily see that the only $\mathbb{T}$-fixed points on this component are $p'_2, p_3$ and $p_4$. Since the $\mathbb{T}$-action on $M^\alpha$ has three fixed points, the submanifold $M^\alpha$ must be equivariantly symplectomorphic to $\C P^2$ with a suitably rescaled Fubini-Study form and standard $S^1$-action (see \cite{Ka}).
 (Here the equivariance is with respect to the residual $\mathbb{T}/H_\alpha$-action). 
 In this $\C P^2$, consider the two invariant spheres containing $p'_2$ and $p_3$, and $p'_2$ and $p_4$. Applying arguments similar to those above, as well as the fact that $\tau_2(p'_2)=0$,
 we see that 
 $\tau_2(p_3)$ and $\tau_2(p_4)$ must be multiples of $\alpha$. Since $\alpha$ and $\beta$ are coprime
 when regarded as polynomials in $\Z[\alpha,\beta]$, we obtain that there exists $m_1\in \Z$ such that 
 $$
 \tau_2(p_3)=m_1\,\alpha\,\beta\,.
 $$
 
 Arguing as above for the connected component of the isotropy submanifold $M^{-\alpha+\beta}$ containing $p_0$ and $p_4$, we conclude that there exists $m_2\in \Z$ such that
 $$
 \tau_2(p_4)=m_2\, \alpha \,(-\alpha+\beta)\,.
 $$
 What is left to do is to determine the integers $m_1$ and $m_2$. In order to do so, we use two conditions. Consider the 4-dimensional isotropy submanifold $M^\alpha$ containing $p'_2,p_3$ and $p_4$, and the whole manifold $W$. Then, by degree reasons, we must have 
 \begin{equation}\label{first eq for tau2}
 \int_{M^\alpha}\tau_2 \in \Z\,\quad\text{and}\quad \int_W \tau_2=0\,.
 \end{equation}
 For both equations we use the ABBV formula, which for the first integral gives
 $$
 \int_{M^\alpha}\tau_2=\frac{m_1\alpha\beta}{\alpha(-\alpha)} + \frac{m_2\alpha(-\alpha+\beta)}{2\alpha^2}\in \Z\,.
$$
This is satisfied if and only if $m_2=2m_1$\,.
The second condition in \eqref{first eq for tau2}, together with the ABBV formula, gives that
$$
\int_W \tau_2=\frac{1}{(-2\alpha+\beta)(-3\alpha+\beta)}+\frac{m_1\alpha\beta}{\alpha(-\alpha)(-3\alpha+\beta)\beta} + \frac{2m_1\alpha(-\alpha+\beta)}{2\alpha^2(-2\alpha+\beta)(-\alpha+\beta)}=0
$$
which is satisfied if and only if $m_1=1$. This completes the computation of the values of $\tau_2$ at all the fixed points, which are summarized in Table \ref{canonical classes W}.

What is left to compute are the values of the canonical class $\tau'_2$ at the fixed points. We notice that the $\mathbb{T}$-action on the connected component $N$ of $M^\alpha$ containing $p'_2$ has fixed points $p'_2,p_3$ and $p_4$, which are exactly the fixed points at which the canonical class $\tau'_2$ can be nonzero. Therefore the equivariant Poincar\'e dual to $N$ in $W$ is the canonical class at $p'_2$, and its value at a fixed point $q\in N$ is exactly the product of the weights in the normal bundle of $N$ at $q$ in $W$, giving exactly the values described in Table  \ref{canonical classes W}. 

\subsubsection{The canonical class $\tau_3\in H^6_{\mathbb{T}}(W;\Z)$:} The only fixed points at which $\tau_3$ is not zero are $p_3$ and $p_4$. By definition of canonical class, we have
$\tau_3(p_3)=\alpha\,\beta(-3\alpha+\beta)$. The value of $\tau_3$ at $p_4$ can be found using
the ABBV formula, noting that  $\int_W \tau_3=0$ and  that the only points at which $\tau_3$ is not necessarily $0$ are $p_3$ and $p_4$.

We end this section with the following result.
\begin{theorem}\label{equiv cohomology and chern classes W}
Let $W$ be the quintic del Pezzo fourfold with the Hamiltonian $\mathbb{T}$-action described in Theorem \ref{Wweight} and page \pageref{eff action}.
Let $\alpha, \beta$ be a $\Z$-basis of the dual lattice $\widetilde{\ell}^*$.
Then there exists a basis of the equivariant cohomology ring $H_{\mathbb{T}}^*(W;\Z)$
given by canonical classes $\{\tau_i\}_{i=0}^4\cup \{\tau'_2\}$, with  $\tau_i\in H_{\mathbb{T}}^{2i}(W;\Z)$, $i=0,\ldots,4$, and  $\tau'_2\in H_{\mathbb{T}}^4(W;\Z)$, satisfying the following properties and relations:
\begin{align}
 \begin{aligned}\label{product W}
& \tau_1^2=(-3\alpha+\beta)\tau_1+3\tau_2+2\tau'_2, & \tau_1\tau_2 & =\beta \tau_2 + \tau_3,\\   
& \tau_1\tau'_2=(-4\alpha+2\beta)\tau'_2+\tau_3, & \tau_1\tau_3 & =(-3\alpha+2\beta)\tau_3+\tau_4,\\
& \tau_2^2=\alpha\beta \tau_2 + \tau_4, & \tau_2\tau'_2 & =\beta \tau_3 - \tau_4,\\\
& (\tau'_2)^2= (-2\alpha+\beta)(-\alpha+\beta)\tau'_2-2\alpha\,\tau_3+2\tau_4\,.\\
  \end{aligned}
 \end{align}
Moreover, the equivariant Chern classes $c_j^{\mathbb{T}}\in H_{\mathbb{T}}^{2j}(W;\Z)$ of the tangent bundle satisfy
{\small \begin{align}
 \begin{aligned}\label{equiv chern classes V}
 & c_1^{\mathbb{T}}= - 3 \tau_1  + (-6\alpha + 4 \beta) .\\
 & c_2^{\mathbb{T}}=  13 \tau_2 + 9\tau_2' +  (4\alpha-7\beta)  \tau_1 + (11\alpha^2 - 18 \alpha \beta + 6 \beta^2) .\\
 & c_3^{\mathbb{T}}= - 16 \tau_3  + 9 (2\beta - 3\alpha) \tau_2  +(7\alpha+6\beta) \tau_{2}'    - (\alpha-5\beta)(\alpha-\beta)\tau_1  -6\alpha^3 + 22 \alpha^2 \beta  - 18 \alpha \beta^2 + 4 \beta^3.\\
 & c_4^{\mathbb{T}}= 6 \tau_4  -4\beta\, \tau_3 + (12\alpha^2 -15\alpha \beta + 5 \beta^2 )\tau_2   + (2 \alpha^2 +  \beta \alpha + \beta^2) \tau_2'   +\beta(\beta-\alpha)(\alpha-\beta)\tau_1   \\
& \hspace{.8cm } +  (-3\alpha+\beta) (-2\alpha+\beta)(-\alpha+\beta) \beta .
 \end{aligned}
 \end{align}}
\end{theorem}

\begin{proof}
The proof of this theorem follows, \textit{mutatis mutandis}, that of Theorem \ref{equiv cohomology and chern classes V}.
\end{proof}

Corollary \ref{cohomology ring W} is a consequence of Theorem \ref{equiv cohomology and chern classes W} and the procedure explained in Section \ref{restrictions and structure constants}.

 \subsection{Index 4: The Grassmannian $Q=\text{Gr}(2,4)$ and its equivariant cohomology ring}\label{subsec Q}

Let $Q$ be the Grassmannian of complex 2-planes in $\C^4$. It is well-known that $Q$ can be realized as a degenerate coadjoint orbit of $SU(4)$, endowed with the action
of a $3$-dimensional torus $T$, with 6 fixed points. For details on how this construction works, as well as for the computation of the weights at the fixed points, we refer the reader to
\cite[Section 5.3.1, Example 5.25 (2)]{GHS}, and the references therein. 
The manifold $Q$ inherits a symplectic form preserved by the torus action with moment map $\psi\colon Q \to \mathfrak{t}^*$, and
the symplectic form can be chosen to be monotone and to satisfy $c_1=[\omega]$ (see Lemma \ref{invariant kahler} as well as \cite[Proposition 5.24]{GHS} for a combinatorial proof). 
Using the notation in \cite{GHS},
let 
$$\alpha=x_1-x_2, \quad \beta=x_2-x_3\quad \text{and} \quad \gamma=x_3-x_4$$ 
be a choice of simple roots, which are a basis of $\mathfrak{t}^*$, and choose 
$\xi\in \mathfrak{t}$ so that $(x_i-x_j)(\xi)>0$ for all $1\leq i < j \leq 4$ (for instance $\xi=\alpha^*+\beta^*+\gamma^*$, where $\alpha^*,\beta^*,\gamma^*$ is a basis of $\mathfrak{t}$
dual to $\alpha,\beta,\gamma$). 
The $T$-action on $Q$ is a GKM action with GKM graph given in Figure \ref{gkm Q}.
Moreover, the weights at each fixed point and the values of the first Chern class are listed in Table \ref{values c1T Q}.

\renewcommand{\arraystretch}{1.2}
\begin{table}[h]
\begin{center}
  \begin{tabular}{| c || c  c  c  c | c | }
    \hline
     &  &  & & & $c_1^T(q)=-\psi(q)$ \\ \hline \hline
    $q_0$ & $\beta$, & $\beta+\gamma$, & $\alpha+\beta+\gamma$, & $\alpha+\beta$ & $2(\alpha+2\beta+\gamma)$ \\ \hline
    $q_1$ & $-\beta$, & $\alpha$, & $\gamma$, & $\alpha+\beta+\gamma$ & $2(\alpha+\gamma)$\\ \hline
    $q_2$ &  $-\gamma$, &  $-\beta-\gamma$, & $\alpha$, & $\alpha+\beta$ & $2(\alpha-\gamma)$\\ \hline
   $q'_2$ & $-\alpha$, & $-\alpha-\beta$, & $\beta+\gamma$, & $\gamma$ & $-2(\alpha-\gamma)$\\ \hline
  $q_3$ &$-\alpha-\beta-\gamma$, & $-\alpha$, & $-\gamma$, & $\beta$ & $-2(\alpha+\gamma)$\\ \hline
   $q_4$ & $-\alpha-\beta-\gamma$, & $-\beta-\gamma$, & $-\alpha-\beta$, & $-\beta$ & $-2(\alpha+2\beta+\gamma)$ \\ 
    \hline
  \end{tabular}
 \end{center}
  \caption{
Weights at the fixed points of $T$-action on $Q$ and the corresponding value of $c_1^T$.}
 \label{values c1T Q}
\end{table}

\begin{remark}
Note that, by Remark~\ref{rmk:indexmultigraph},  the index $k_0=4$ of $Q$ can be retrieved from the    affine length of the edge connecting $q_0$ to $q_1$.
\end{remark}
  
Even if the equivariant cohomology ring of $Q$ is well-known, for completeness we carry on computations similar to those in the previous sections.
The first observation is that canonical classes exist, as the GKM graph of $Q$ is index-increasing (see \cite[Theorem 1.6]{GT}). Therefore, we compute their restrictions to the fixed point 
set which, in turn, determine the (equivariant) structure constants.

Let $\tau_i\in H^{2i}_T(Q;\Z)$ be the canonical class at $q_i$, for all $i=0,\ldots,4$ and $\tau'_2\in H^4_T(Q;\Z)$ the canonical class at $q'_2$. 
By Lemma \ref{tau 1}, $\tau_1$ is given by 
$$
\tau_1=\frac{2\alpha+4\beta+2\gamma-c_1^T}{4}\,,
$$ 
and its values are as in Table \ref{canonical classes Q}
\renewcommand{\arraystretch}{1.2}
\begin{table}[h]
\begin{center}
  \begin{tabular}{| c || c | c | c | c |}
    \hline
                          &  $\tau_1$                         & $\tau_2$                                                     & $\tau'_2$                                   & $\tau_3$  \\ \hline \hline
    $q_0$            &  0                                     & 0                                                                 &  0                                               & 0  \\ \hline
    $q_1$            &  $\beta$                           &  0                                                                &  0                                               & 0  \\ \hline
    $q_2$            &  $\beta+\gamma$            &  $\gamma(\beta+\gamma)$                       & 0                                                 & 0  \\ \hline
    $q'_2$           &  $\alpha+\beta$                &  $0$                                                           & $\alpha(\alpha+\beta)$               & 0   \\ \hline
    $q_3$            &  $\alpha+\beta+\gamma$ &  $\gamma(\alpha+\beta+\gamma)$           & $\alpha(\alpha+\beta+\gamma)$ & $\alpha\gamma(\alpha+\beta+\gamma)$  \\ \hline
    $q_4$            &  $\alpha+2\beta+\gamma$ & $(\alpha+\beta+\gamma)(\beta+\gamma)$ & $(\alpha+\beta)(\alpha+\beta+\gamma)$ &$(\alpha+\beta)(\beta+\gamma)(\alpha+\beta+\gamma)$   \\ \hline
\end{tabular}
 \end{center}
  \caption{
Restrictions to the fixed points of the ``non-trivial'' canonical classes on $Q$.}
 \label{canonical classes Q}
\end{table}

We now  compute the restrictions to the fixed points of the canonical classes $\tau_2$ and $\tau'_2$. In order to do this, we follow the same procedure as in Section \ref{subsec V}. For $\tau_2$, the value $\tau_2(q_2)$ at $q_2$ is given (by definition) by $\gamma(\beta+\gamma)$. In order to compute $\tau_2(q_3)$
it is sufficient to compute $\Theta(q_2,q_3)$ which, with the choice of $\xi$ given at the beginning of this section, is easily seen to be 1. Therefore \cite[(4.1)]{GT} gives
$\tau_2(q_3)=\gamma(\alpha+\beta+\gamma)$. In order to compute $\tau_2(q_4)$ we see that the only path in the path-formula of \cite[Theorem 1.6]{GT} from $q_2$ to $q_4$ is the one containing the edges $(q_2,q_3)$
and $(q_3,q_4)$.  Since $\Theta(q_3,q_4)=1$, the formula \cite[(1.3)]{GT} gives that $\tau_2(q_4)=(\alpha+\beta+\gamma)(\beta+\gamma)$.
Note that there is no path in the path formula from $q_2$ to $q'_2$, so by \cite[Theorem 1.6]{GT} it follows that $\tau_2(q'_2)=0$. 

Similar computations can be carried out for $\tau'_2$. Its values at the fixed points are listed in Table \ref{canonical classes Q}. 

As for $\tau_3$, it is sufficient to observe that the equivariant Poincar\'e dual to the $T$-invariant sphere connecting $q_3$ to $q_4$ satisfies the defining properties of the canonical class
at $q_3$. Therefore its values can immediately be computed, and are written in Table \ref{canonical classes Q}.  

Arguing as in the previous examples, we use Table \ref{canonical classes Q} and the procedure described in Section \ref{restrictions and structure constants} to compute the equivariant structure constants, obtaining the
following result.
\begin{theorem}\label{equiv cohomology and chern classes Q}
Let $Q$ be the Grassmannian $Gr(2,4)$ with the Hamiltonian $T$-action described in Section \ref{subsec Q}.
Let $\alpha, \beta, \gamma$ be a $\Z$-basis of the dual lattice $\widetilde{\ell}^*$.
Then, there exists a basis of the equivariant cohomology ring $H_T^*(Q;\Z)$
given by canonical classes $\{\tau_i\}_{i=0}^4\cup \{\tau'_2\}$, with $\tau_i\in H_{T}^{2i}(Q;\Z)$, for  $i=0,\ldots,4$, and $\tau'_2\in H_{T}^4(Q;\Z)$, satisfying the following properties and relations:
\begin{align}
 \begin{aligned}\label{product Q}
 & \tau_1^2= \beta\tau_1+\tau_2+\tau'_2, & \quad  \tau_1\tau_2 & = (\beta+\gamma)\tau_2+\tau_3,\\
 & \tau_1\tau'_2= (\alpha+\beta)\tau'_2+\tau_3, & \quad   \tau_1\tau_3 & = (\alpha+\beta+\gamma)\tau_3+\tau_4,\\
 & \tau_2^2=\gamma(\beta+\gamma)\tau_2+\gamma\,\tau_3+\tau_4, & \quad (\tau'_2)^2 & = \alpha(\alpha+\beta)\tau'_2+\alpha\tau_3 + \tau_4, \\
& \tau_2\tau'_2=(\alpha+\beta+\gamma)\tau_3\,.\\
 \end{aligned}
 \end{align}

Moreover, the equivariant Chern classes $c_j^T\in H_{T}^{2j}(Q;\Z)$ of the tangent bundle satisfy
{\small \begin{align}
 \begin{aligned}\label{equiv chern classes V}
 & c_1^T=   - 4\tau_1 + 2(\alpha+2\beta+\gamma). \\
 & c_2^T=   7\tau_2'  + 7\tau_2 - 7(\alpha +  \beta +  \gamma) \tau_1 +  (\alpha^2+6\alpha \beta + 3 \alpha \gamma + 6 \beta^2 + 6 \beta \gamma +\gamma^2). \\
 & c_3^T=   - 12 \tau_{3} +(9\alpha+6\beta+3\gamma)\tau_2 + (3\alpha +6\beta+ 9\gamma) \tau_{2}'  + (-3 \alpha^2- 7 \alpha \beta - 8 \alpha \gamma - 4 \beta^2 - 7 \beta \gamma -3 \gamma^2) \tau_1 \\
&  \hspace{.8cm} +  (\alpha + 2\beta + \gamma)(2\alpha \beta + \alpha \gamma +2\beta^2 + 2\beta \gamma)  .  \\
& c_4^T=    6 \tau_4  -3(\alpha+2\beta+\gamma)\tau_3 + (3\alpha^2+4\alpha \beta +2\alpha \gamma + \beta^2+\beta \gamma)\tau_2+(\alpha \beta+2\alpha \gamma+\beta^2+4\beta \gamma+3\gamma^2)\tau_2'\\
& \hspace{.8cm} -(\alpha+\beta+\gamma)(\alpha \beta+2 \alpha \gamma+\beta^2+\beta \gamma) \tau_1 + \beta (\beta+\gamma) (\alpha+\beta+\gamma)(\alpha+\beta) .
 \end{aligned}
 \end{align}}
\end{theorem}

\begin{proof}
The proof of this theorem follows, \textit{mutatis mutandis}, that of Theorem \ref{equiv cohomology and chern classes V}.
\end{proof}


We conclude this section with the following corollary of Theorem \ref{equiv cohomology and chern classes Q}, which gives the well-known structure constants of the cohomology ring
of $Q$,   obtained by using \eqref{structure constants}. Note that these correspond   (modulo sign)  to the
ordinary Schubert classes (see \cite[Proposition 6.5]{ST}).


\begin{cor}\label{cohomology ring Q}
Let $Q$ be the Grassmannian of complex planes in $\C^4$. Then there exists a basis $\{\zeta_i\}_{i=0,\ldots,4}\cup\{\zeta'_2\}$ of the cohomology ring $H^*(Q;\Z)$, with  $\zeta_i\in H^{2i}(Q;\Z)$,  $i=0,\ldots,4$, and $\zeta'_2\in H^4(Q;\Z)$, satisfying the
following properties and relations: 
\begin{equation}\label{product Q o}
  \zeta_1^2 = \zeta_2+\zeta'_2, \quad \zeta_1\zeta_3= \zeta_2^2=(\zeta'_2)^2 =\zeta_4, \quad 
  \zeta_1\zeta_2 =  \zeta_1\zeta'_2= \zeta_3, \quad   \zeta_2\zeta'_2=0 
 \end{equation}
 Moreover, the Chern classes are given by
 $$
 c_1=-4 \, \zeta_1,\quad c_2= 7\,\zeta_2 + 7\,\zeta_2^\prime, \quad c_3=-12\, \zeta_3 \quad \text{and} \quad c_4=6 \,\zeta_4.
 $$
\end{cor}


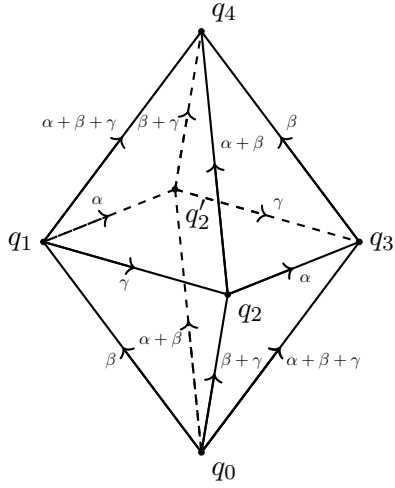
\begin{figure}[h] 

 \begin{tikzpicture}[scale = 0.7]

\filldraw  (-3,0) circle (1.5pt);

\filldraw  (3,0) circle (1.5pt);

\filldraw  (0.5,-1) circle (1.5pt);

\filldraw  (-0.5,1) circle (1.5pt);

\filldraw  (0,4) circle (1.5pt);

\filldraw  (0,-4) circle (1.5pt);

\draw [thick]  (-3,0) -- (0,4);

\draw [thick]  (-3,0) -- (0,-4);

\draw [thick]  (3,0) -- (0,4);

\draw [thick]  (3,0) -- (0,-4);

\draw [thick, dashed]  (-0.5,1) -- (0,4);

\draw [thick, dashed]  (-0.5,1) -- (0,-4);

\draw [thick]  (0.5,-1) -- (0,4);

\draw [thick]  (0.5,-1) -- (0,-4);

\draw [thick,dashed]  (-0.5,1) -- (3,0);

\draw [thick, dashed]  (-0.5,1) -- (-3,0);

\draw [thick]  (0.5,-1) -- (3,0);

\draw [thick]  (0.5,-1) -- (-3,0);

\draw [->, thick] (0,-4) -- (1.5,-2);

\draw [->, thick] (0,-4) -- (-1.5,-2);

\draw [->, thick] (0,-4) -- (0.25,-2.5);

\draw [->, thick, dashed] (0,-4) -- (-0.25,-1.5);

\draw [->, thick] (0.5,-1) -- (1.75,-0.5);

\draw [->, thick] (-3,0) -- (-1.25,-0.5);

\draw [->, thick] (0,-4) -- (1.5,-2);

\draw [->, thick] (-3,0) -- (-1.5,2);

\draw [->, thick] (3,0) -- (1.5,2);

\draw [->, thick] (0.5,-1) -- (0.25,1.5);

\draw [->, thick, dashed] (-0.5,1) -- (-0.25,2.5);

\draw [->, thick, dashed] (-0.5,1) -- (1.25,0.5);

\draw [->, thick, dashed] (-3,0) -- (-1.75,0.5);

\node [below right] (p_1) at (0,-4) {$q_0$};

\node [above right] (p_1) at (0,4) {$q_4$};

\node [right] (p_1) at (3,0) {$q_3$};

\node [left] (p_1) at (-3,0) {$q_1$};

\node [below right] (p_1) at (0.5,-1) {$q_2$};

\node [below right] (p_1) at (-0.5,1) {$q_2'$};

\node [below right, scale=0.6] (p_1) at (0.25,-2) {$\beta + \gamma$};

\node [below right, scale=0.6] (p_1) at (1.5,-2) {$\alpha + \beta + \gamma$};

\node [below left, scale=0.6] (p_1) at (-1.5,-2) {$ \beta $};

\node [below left, scale=0.6] (p_1) at (-0.25,-1.6) {$ \alpha +  \beta $};

\node [below right, scale=0.6] (p_1) at (1.75,-0.5) {$ \alpha $};

\node [below left, scale=0.6] (p_1) at (-1.25,-0.6) {$ \gamma $};

\node [above right, scale=0.6] (p_1) at (1.25,0.5) {$ \gamma $};

\node [above left, scale=0.6] (p_1) at (-1.75,0.6) {$ \alpha $};

\node [above right, scale=0.6] (p_1) at (1.5,2) {$ \beta $};

\node [above left, scale=0.6] (p_1) at (-0.3,2) {$\beta + \gamma$};

\node [above left, scale=0.6] (p_1) at (-1.5,2) {$\alpha + \beta + \gamma$};

\node [above right, scale=0.6] (p_1) at (0.25,1.6) {$ \alpha +  \beta $};

 (1.5,-2), (-1.5,2), (1.5,2), (0.25,1.5), (-0.25,2.5), (1.25,0.5),(-1.75,0.5)

\end{tikzpicture}

\caption{The oriented GKM graph of $Q$.}
\label{gkm Q}

\end{figure}

\subsection{Torus actions on $\mathbb{CP}^2 \times \mathbb{CP}^2$}

We conclude this section with the following well-known result, which will be useful
in Section~\ref{sec:index3_CP2}.

\begin{theorem}\label{thm:CP2CP2}
Consider the Fano $4$-fold $\left( X=\mathbb{CP}^2 \times \mathbb{CP}^2, \omega = 3 \,\omega_0 \oplus 3 \,\omega_0 \right)$, where $\omega_0$ is the Fubini-study form on $\mathbb{CP}^2$.
Then  $c_{1}(X) = [\omega]$ and  the weights for the  standard (toric) $T^4$-action on $X$ 
at the  $9$ fixed points
are 
{\small \begin{align*}
& \{(1,0,0,0),(0,1,0,0),(0,0,1,0),(0,0,0,1)\} ,  \quad  \{(1,0,0,0),(0,1,0,0),(0,0,-1,-1),(0,0,0,1)\} ,  \\  &  \{(1,0,0,0),(0,1,0,0),(0,0,1,0),(0,0,-1,-1)\} ,   \quad  \{(-1,-1,0,0),(0,1,0,0),(0,0,1,0),(0,0,0,1)\}  & \\  & \{(-1,-1,0,0),(0,1,0,0),(0,0,-1,-1),(0,0,0,1)\} ,  \,\,  \{(-1,-1,0,0),(0,1,0,0),(0,0,1,0),(0,0,-1,-1)\}  \\  & \{(1,0,0,0),(-1,-1,0,0),(0,0,1,0),(0,0,0,1)\} ,   \,\,  \{(1,0,0,0),(-1,-1,0,0), (0,0,-1,-1),(0,0,0,1)\},  \\ &  \{(1,0,0,0),(-1,-1,0,0) ,(0,0,1,0),(0,0,-1,-1)\}. 
\end{align*}}
The moment polytope is $\Delta^2 \subset \mathbb{R}^4$ where $\Delta \subset \mathbb{R}^2$ is the triangle with vertices $(0,0),(3,0),(0,3)$.
\end{theorem}

\section{The algorithm}\label{sec:algorithm}

In this section, we review the algorithm introduced in \cite{GS} that will be used in our classification results. Even though
this algorithm produces the possible  circle actions on a manifold with a certain data (e.g. with given Betti numbers), we can still use it to obtain classification results for actions of higher-dimensional tori.  The strategy to classify  possible torus actions on a manifold with a certain data  will be to obtain the possible circle actions on the manifold with the same data and then show that they are restrictions of a global torus action to generic subcircles.

Let $(M,\omega)$ be a closed, symplectic manifold of dimension $2n$ equipped with a circle action with isolated fixed points. In this case the weights of the action at a fixed point $p$ are nonzero integers  $w_1,\ldots, w_n$. Let $W^+$ and $W^-$ be the multisets of positive and negative weights at all fixed points of $M$
$$
W^+ = \displaystyle{\bigsqcup_{p_i\in M^{S^1}}}\{w_{ik}\mid w_{ik}>0\}\quad \text{and} \quad W^-= \displaystyle{\bigsqcup_{p_i\in M^{S^1}}}\{w_{ik}\mid w_{ik}<0\},
$$ 
where $\bigsqcup$ denotes the disjoint union of multisets (cf. Remark~\ref{rmk:circle2}).

By \cite[Proposition 2.11]{Ha} we have  $W_+=-W_-$, and so we can always choose a bijection $g:W^+\to W^-$ taking each $w_{ik}$ to some $w_{jl}$ with $w_{ik}=-w_{jl}$, leading to a multigraph $(V,E)$ as in Definition~\ref{def:multigraph} (see also Remark~\ref{rmk:circle2}). 

We can choose a direction on the multigraph as in Remark~\ref{rmk:circle2}. Moreover, we  label the (directed) edges with a map $\mathsf{w}:E\to \Z_{>0}$ given by $\mathsf{w}(e_{ik})=\lvert w_{ik} \rvert$ and work with \textbf{labeled multigraphs} $(V,E,\mathsf{w})$.


Jang and Tolman  \cite[Lemma 5]{JaTo} showed that, given a circle action on a closed almost complex manifold $M$ with isolated fixed points, there always exists a multigraph describing $M$ obtained from $W^+$ and $W^-$ with no self-loops. Therefore, we will  consider  the family $\mathcal{W}$ of all directed multigraphs determined by $W^+$ and $W^-$ that have no self-loops.

Given a multigraph $\Gamma=(V,E)$ describing $(M,\omega, \psi)$, we define its \textbf{magnitude function} $\mathsf{m}: E \to \Z$ as the evaluation of $c_1^{S^1}$ on the edges of $\Gamma$,
$$
\mathsf{m}(e):= c_1^{S^1}[e] 
$$
(cf. Definition~\ref{evaluation alpha}), where  $c_1^{S^1}$ is the first equivariant Chern class.
\begin{remark}\label{rem:algo}
Note that, for each fixed point $p\in M^{S^1}$, we have 
$$c_1^{S^1}(p)=\left(\sum_{j=1}^n w_{j}\right)x = \Sigma(p) \,x,$$ 
where $\{w_1,\ldots,w_n\}$ is the multiset of weights at $p$ and $x$ is  a degree-$2$ generator of $H_{S^1}({pt}, \Z) \simeq \Z[x]$.
\end{remark}

The sum of the values of the magnitude function over all the edges of a multigraph $\Gamma=(V,E)$ in $\mathcal{W}$  does not depend on $\Gamma$ (i.e. it is independent of the choice of bijection between $W^-$ and $W^+$).  More precisely, we have the following result.

\begin{proposition}\cite[Proposition 4.6]{GS}\label{PropGS}
Let $(M,\omega)$ be a closed symplectic manifold equipped with a Hamiltonian $S^1$-action with isolated fixed points and moment map $\psi:M\to \R$. Let $\mathcal{W}$ be the family of directed multigraphs describing $(M,\omega, \psi)$ with no self-loops. Then the magnitude $\mathsf{m}\colon E\to \Z$ associated to a multigraph  $\Gamma=(V,E)\in\mathcal{W}$  satisfies 
\begin{equation}\label{eq:cqcn-1}
\sum_{e\in E} \mathsf{m}(e) = c_1 c_{n-1}[M]=\sum_{i=0}^n b_{2i}(M)\left( 6\, i (i-1) + \frac{5 n -3n^2}{2}\right),
\end{equation}
where $b_{2i}(M)$ is the $2i$-th Betti number of $M$.
\end{proposition}
Given a labeled directed multigraph $\Gamma=(V,E,\mathsf{w})$ we define 
the map $\delta\colon V\times E\to \{-1,0,1\}$ given by
$$
\delta(p,e)= \left\{ \begin{array}{rl} 1,& \text{if $i(e)=p$} \\ -1,& \text{if $t(e)=p$} \\ 0, & \text{otherwise,} \end{array} \right.
$$
where $i(e)$ and $t(e)$ are the initial and terminal points of $e$.
It is easy to check that, from the definition of $c_1^{S^1}[e]$,  Remark~\ref{rem:algo}  and the definition of $\delta$, fixing an edge $\tilde{e}\in E$,  we always have 
\begin{equation}\label{eq:delta}
\sum_{e\in E} \big(\delta\left(\mathsf{i}(\tilde{e}),e\right)-\delta\left(\mathsf{t}(\tilde{e}),e\right) \big)\mathsf{w}(e)-\mathsf{w}(\tilde{e})\mathsf{m}(\tilde{e})=0.
\end{equation}
Hence, considering an ordering $E=\{e_1,\ldots,e_{\lvert E \rvert} \}$ on the edge set  and the  $(\lvert E\rvert \times \lvert E\rvert)$-matrices $A(\Gamma)=(a_{ij})$ and $\diag(\mathsf{m}(E))$ defined by
$$
a_{ij}:=\delta \left(\mathsf{i}(e_i),e_j\right)-\delta\left(\mathsf{t}(e_i),e_j\right)
$$ 
and
$$
\diag(\mathsf{m}(E)) :=\diag\left(\mathsf{m}(e_1),\ldots,\mathsf{m}(e_{\lvert E \rvert})\right),
$$
we have that  the weight vector  $\mathsf{w}(E):=(\mathsf{w}(e_1),\ldots,\mathsf{w}(e_{\lvert E \rvert})\in \Z_{>0}^{\lvert E\rvert }$ must be in the kernel of
\begin{equation}\label{eq:matrices}
A(\Gamma)-\diag(\mathsf{m}(E)).
\end{equation}

Therefore, given positive integers $n, N$ and integers $N_i\geq 0$ for $i=0,\ldots,n$, such that $\sum_{i=0}^n N_i=N$, the  possible isotropy weights for a Hamiltonian $S^1$-action on a closed symplectic  manifold  $(M,\omega)$ of dimension $n$ with $N$ isolated fixed points, such that the number of fixed points 
with $i$ negative weights is $N_i$ (i.e. such that $b_{2i}(M)=N_i$), must satisfy several  linear relations  determined by the following \textbf{algorithm} (see \cite{GS} for details): 
\begin{enumerate}
\item Consider all possible multigraphs $\Gamma=(V,E)$ with vertex set $V=\{p_1,\ldots,p_N\}$ and no self loops,  such that the number of vertices
which are the endpoints of $i$ edges is $N_i$. 
\item Determine the matrices $A(\Gamma)$ associated to the multigraphs $\Gamma$.
\item Look for possible functions $\mathsf{m}:E\to \Z$ satisfying \eqref{eq:cqcn-1}
and such that
\begin{equation}\label{algo}
\Big(  \ker \left(A(\Gamma)-\diag\left(\mathsf{m}(E)\right)\right)\Big)\cap \Z_{>0}^{\lvert E\rvert }\neq \emptyset,
\end{equation}
(these functions will be the possible magnitudes of the multigraphs). 
\item When they exist, the weight vectors  $\mathsf{w}(E)$ associated to the edges of the multigraphs $\Gamma$ must  be in the kernel of the  matrices
$$
A(\Gamma)-\diag(\mathsf{m}(E)).
$$
\end{enumerate}

If we assume the functions $\mathsf{m}:E\to \Z$ to be positive, they correspond to partitions of 
\begin{equation}\label{eq:int}
\sum_{i=0}^n N_i \left(6i(i-1)+\frac{5n-3n^2}{2}\right) \in \Z
\end{equation}
into $\lvert E\rvert$ positive integers $\mathsf{m}(e)$ and so there is only a finite number of possibilities for $\mathsf{m}$ and the algorithm ends.
\begin{remark}
One cannot expect that, in general, the magnitudes associated to multigraphs describing manifolds with circle actions are positive. For example, if the number of fixed points is large,  then $c_1c_{n-1}[M]$ can be negative (cf. Proposition~\ref{PropGS}). However, if for example $(M,\omega)$ is an $8$-dimensional, closed,  positive  monotone symplectic manifold, as is the case of $\chi(M)=6$, we can rescale the symplectic form and translate the moment map in such a way that $c_1^{S^1}=[\omega-\psi]$ (cf. Lemma~\ref{lemma monotone chi six}), and then Theorem~\ref{existence multigraph dim 8} gives conditions that ensure the existence of positive magnitudes. 
\end{remark}

\section{Proof of Main Theorem and other classification results}

In this section we give a proof of Main Theorem and of Theorem~\ref{thm:add}. The proofs proceed with the following steps:
\begin{enumerate}
\item[1)] Use the algorithm introduced in \cite{GS} and reviewed in Section~\ref{sec:algorithm} to classify the possible circle actions on a manifold satisfying the hypothesis of these theorems. 
\item[2)] Show that the circle actions obtained are restrictions of actions of higher-dimen\-sional tori  with the same isotropy data  as the ones described in the examples of Section~\ref{section examples}.
\item[3)] Observe that the isotropy data obtained completely determines the (equivariant) cohomology ring and Chern classes of the corresponding manifolds as it was proved in Section~\ref{section examples}.
\end{enumerate}

Note that the algorithm introduced in \cite{GS} can be used in the proof of Main Theorem because, by Theorem~\ref{existence multigraph dim 8} and Lemma~\ref{lemma monotone chi six}, we have the existence of a positive multigraph with respect to the first equivariant Chern class. In Theorem~\ref{thm:add},
the existence of a positive multigraph is an assumption.

We recall that all the accompanying files can be found at 
$$
 \href{https://www.math.tecnico.ulisboa.pt/~lgodin/CircleActions/}{\text{http://www.math.ist.utl.pt/$\sim$lgodin/CircleActions/}}
$$

\subsection{Classification results for $b_2(M)=1$ and $b_4(M)=2$:}\label{sec:class}
We begin with the following result that will be used throughout this section.
\begin{lemma}\label{lemma:mag}
Let $(M,\omega,\psi)$ be as in Main Theorem and consider a multigraph $\Gamma=(V,E)$ describing it as in Theorem~\ref{existence multigraph dim 8}. Then the index $k_0$ of $(M,\omega)$ satisfies
$$
k_0= \gcd\,\,  \{\mathsf{m}(e),\,\,\, e\in E \}.
$$
\end{lemma}
\begin{proof}
We first note that $k_0$ divides $\mathsf{m}(e)=c_1^{S^1}[e]$ for every $e\in E$ (see proof of Lemma~\ref{lemma betti numbers}).  Moreover, by Lemma~\ref{tau 1}, we know that $k_0$ is exactly the magnitude of the edges connecting $p_0$ to $p_1$ and $p_3$ to $p_4$.
\end{proof}

Given the properties of the magnitudes of Lemma~\ref{lemma:mag}, in the following sections we run the algorithm for each separate index $k_0=i$ with $i=1,\ldots,4$ (cf. Theorem~\ref{main theorem chern}).

We begin by first running Part $1$ of the file {\texttt CircleActions$\_$dim$8\_$nfpt$6\_$General.nb} to generate the list of all possible multigraphs $\Gamma=(E,V)$ describing a circle action on an $8$-dimensional manifold with $6$ fixed points (these multigraphs can be seen in the file {\texttt Multigraphs.pdf}). 

Then we run Part $2$ of the same file  to obtain the complete list of determinants of the matrices $A(\Gamma)-\diag(\mathsf{m}(E))$ in \eqref{eq:matrices} (listed in the file {\texttt Determinants.nb}).

After this, we run the  \texttt{c++} files  {\texttt Partitions-$\#$.cpp}, where $\#$ denotes the order number of the multigraph in the list above, to obtain partitions of  $c_1c_{3}[M]$ into $12=\lvert E\rvert = 4 \chi(M)/2$ positive integers $\mathsf{m}(e)$ (according to the algorithm described in Section~\ref{sec:algorithm}), for which there exists a determinant of a matrix $A(\Gamma_i)-\diag(\mathsf{m}(E_i))$ that is zero. 

We then run  Part $3$ of  the files  {\texttt Index\emph{i}.nb} (with $i=1,\ldots,4$)  to sort these partitions according to the rank of these singular matrices  \eqref{eq:matrices},  dividing them into two sets: those that produce singular matrices with maximal rank, and those that produce  matrices of lower rank. 
 
In the first case, we also select the partitions that produce matrices with null spaces containing vectors with positive components, according to \eqref{algo}. We store the first type of partitions in the files {\texttt Part$\#$toprk.txt} and the others in the files {\texttt Part$\#$smallrk.txt}.

Then, Part $4$ of  the files  {\texttt Index\emph{i}.nb} (with $i=1,\ldots,4$)   considers the first set of partitions,  producing a list of the corresponding isotropy weights, checking if they satisfy the polynomial equations in  \cite[Equation (1.1)]{GS} obtained from the Atiyah-Bott-Berline-Vergne Localization Theorem.

Finally, part $5$ of {\texttt Index\emph{i}.nb} (with $i=1,\ldots,4$) considers the second set of partitions and
\begin{itemize}
\item selects those that originate matrices with null spaces containing vectors with positive entries (see \eqref{algo});
\item produces the list of the corresponding isotropy weights;
\item selects those isotropy weights that satisfy several necessary properties such as
\begin{enumerate}
\item[(i)]\label{conditioniii} at each fixed point, the isotropy weights must be coprime integers; 
\item[(ii)] at a fixed point $P$ of index $2i$, the first $i$ weights are negative and the others are positive;
\item[(iii)] if there are three multiple edges between two vertices, then the weights corresponding to these edges must be coprime \cite[Lemma 8.2(1) ]{GS}.
\end{enumerate}
\end{itemize}
At each step, the resulting lists of isotropy weights are saved in different files so that it is easy to verify which isotropy weights are discarded with each test performed.

\subsubsection{Index 1}\label{secresults:index1}

If we assume that the index of the manifold $(M,\omega)$ is $k_0=1$, the greatest common divisor of all the $12$ elements of the partitions considered must be $1$ and the first and last elements of the partitions are always equal to $1$ (cf. Lemma~\ref{lemma:mag}) since these correspond to the magnitudes of the edges connecting $p_0$ to $p_1$ and $p_3$ to $p_4$.
The possible partitions can be found in the file {\texttt PartitionsIndex1.zip}.

In the end, the program eliminates all possibilities. Indeed, no potential multiset of isotropy weights satisfies all the necessary conditions and so we conclude {\bf that there are no possible circle actions in the index-$1$ case}.

\subsubsection{Index 2}
When the index is $2$, the greatest common divisor of the $12$ elements of the partitions considered must be  $2$ and the first and last elements of the partitions are always $2$ (cf. Lemma~\ref{lemma:mag}).
The possible partitions can be found in the file {\texttt PartitionsIndex2.zip}.
%
%
The isotropy weights and the corresponding multigraphs that satisfy all the necessary conditions can be found in the file ResultsIndex2.pdf and are shown in Figures~\ref{fig:Index2} and \ref{fig:Index2Weights}.

\begin{figure}[h!]
\begin{center}
\includegraphics[scale=0.33] {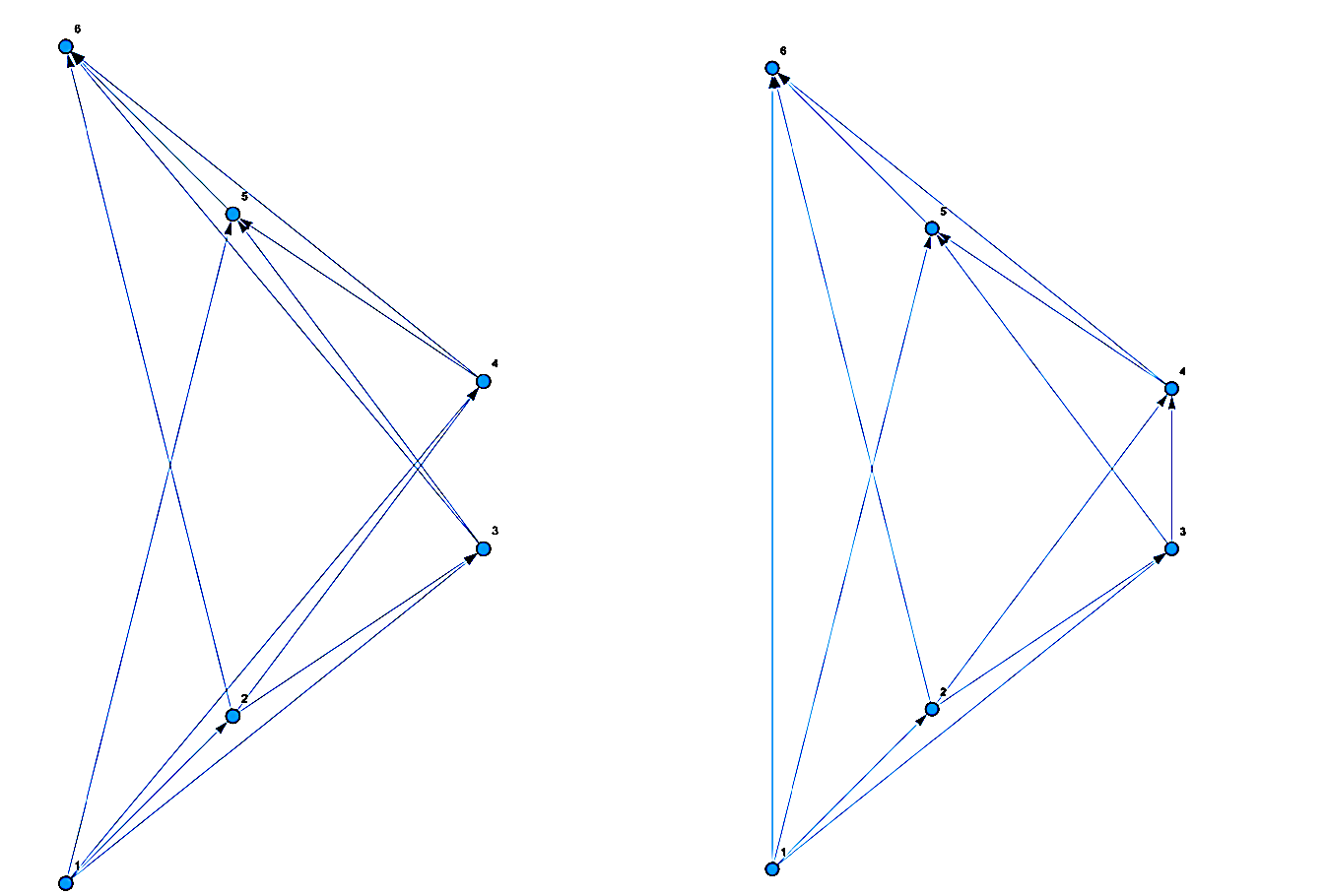}
\caption{Possible multigraphs for Index $2$}\label{fig:Index2}
\end{center}
\end{figure}

\begin{figure}[h!]
\begin{center}
\includegraphics[scale=0.7] {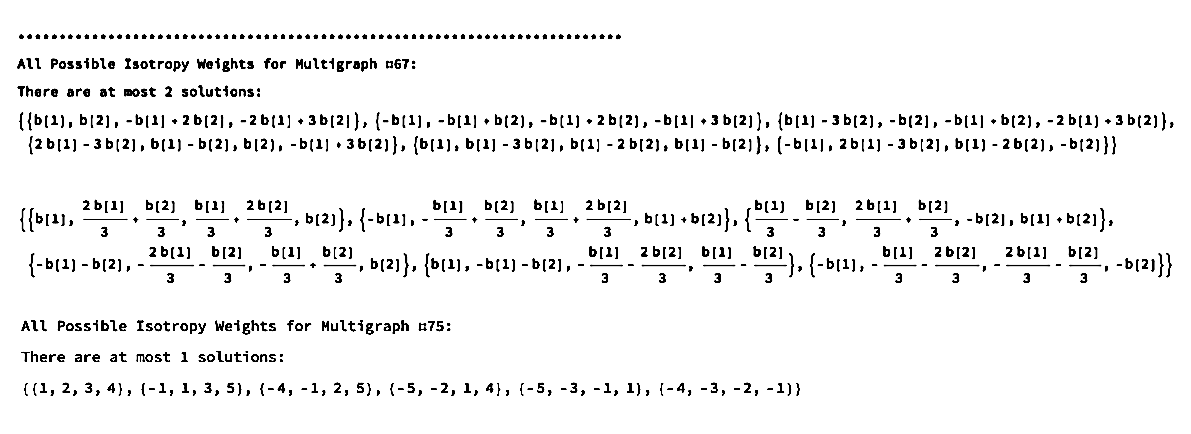}
\caption{Possible isotropy weights for Index $2$}\label{fig:Index2Weights}
\end{center}
\end{figure}
In all cases, the Chern numbers obtained from the isotropy weights in Figure~\ref{fig:Index2Weights} are
\begin{align*}
c_1^4[M] = 288, \quad c_1^2c_2[M]=168 \quad \text{and} \quad c_2^2[M]=98.
\end{align*}
The first two multisets of weights give us a $2$-dimensional family of circle actions that generate the  $T$-action on the Fano-Mukai $4$-fold $V$ described in Section~\ref{subsec V}.\\

\noindent  $\bullet$ The first $2$-dimensional family of weights listed in Figure~\ref{fig:Index2Weights} generates the effective $T$-action on the Fano-Mukai $4$-fold $V$. Indeed, if we consider the basis ${\emph b}_1 =2x+y=(2,1)$ and ${\emph b}_2=x+y=(1,1)$ of the dual lattice $\ell^*$ of $T$, where $x$ and $y$ are as in Section~\ref{sec:5.1.2}, then the weights for the action of $T$ listed in Table~\ref{values c1T V} become those in Table~\ref{weigthsV}.

\renewcommand{\arraystretch}{1.2}
\begin{table}[h!]
\begin{center}
  \begin{tabular}{| c || c  c  c  c |  }
    \hline
     &  &  & & \\ \hline \hline
    $q_0$ & $-2\bb_1+3\bb_2$, & $-\bb_1+2\bb_2$, & $\bb_2$, & $\bb_1$  \\ \hline
    $q_1$ & $2\bb_1-3\bb_2$, & $\bb_1-\bb_2$, & $\bb_2$, & $-\bb_1+3\bb_2$ \\ \hline
    $q_2$ &  $-\bb_1$, &  $-\bb_1+\bb_2$, & $-\bb_1+2\bb_2$, & $-\bb_1+3\bb_2$ \\ \hline
   $q'_2$ & $\bb_1-3\bb_2$, & $\bb_1-2\bb_2$, & $\bb_1-\bb_2$, & $\bb_1$\\ \hline
  $q_3$ &$\bb_1-3\bb_2$, & $-\bb_2$, & $-\bb_1+\bb_2$, & $-2\bb_1+3\bb_2$\\ \hline
   $q_4$ & $-\bb_1$, & $-\bb_2$, & $\bb_1-2\bb_2$, & $2\bb_1-3\bb_2$  \\ 
    \hline
  \end{tabular}
 \end{center}
  \caption{
Weights for the effective $T$-action on $V$ for the new basis of $\ell^*$.}
 \label{weigthsV}
\end{table}
 The multigraphs obtained for these circle actions from the projection of the $T$-action to a generic subcircle  are always as the one on the left of Figure~\ref{fig:Index2}. Moreover, considering  the corresponding moment maps, we obtain the following ordering of the fixed point set as defined in \eqref{ordering fixed points}
$$
q_0 \prec q_2 \prec q_3 \prec q_1 \prec q_2^\prime \prec q_4.
$$
Note that this ordering is different from the one chosen in Section~\ref{subsec V} to compute the equivariant cohomology for the $T$-action on $V$.
\vspace{.2cm}


\noindent  $\bullet$ The second $2$-dimensional family of weights listed in Figure~\ref{fig:Index2Weights}  with $\text{b[1]},\text{b[2]} \in \Z_{>0}$  satisfying $\text{b[2]} >   \text{b[1]}$,  $2\text{b[1]}+ \text{b[2]}= 0 \pmod 3$  and $\text{b[1]}+ 2\text{b[2]}=0 \pmod 3$ with 
$$\gcd \left(\text{b[1]},\text{b[2]},\frac{2\text{b[1]}+ \text{b[2]}}{3},\frac{\text{b[1]}+ 2\text{b[2]}}{3}\right)=1,$$
also generates the effective $T$-action on the Fano-Mukai $4$-fold $V$ since, writing $u=b_1 +b_2 \in \Z$ and $v=\frac{b_1+2b_2}{3}\in \Z$, we recover the first  family of weights\footnote{These conditions ensure that the weights are integers and that the corresponding circles actions are effective.} . The multigraphs obtained from the corresponding projections of the $T$-action are always as the one on the left of Figure~\ref{fig:Index2} and, considering the ordering on the fixed point set given by the associated moment maps, we obtain that 
$$
q_1 \prec q_0\prec q_2^\prime \prec q_2 \prec q_4 \prec q_3.
$$

\begin{remark}
The third multiset of weights is a particular case of the first with $b[1]=1$ and $b[2]=2$. The multigraph for this action (on the right of Figure~\ref{fig:Index2}) is different from the one of the previous action only for the edges that are labeled  $1$. These edges correspond to free gradient spheres which we know depend on the metric considered. Therefore, we can still obtain this multigraph and this multiset of weights from a subcircle of $T$.
\end{remark}
%

\subsubsection{Index 3}\label{subsec:index3}
When the index is $3$, the greatest common divisor of the $12$ elements of the partitions considered must be $3$ and the first and last elements of the partitions are always $3$ (cf. Lemma~\ref{lemma:mag}).
The possible partitions can be found in the file {\texttt PartitionsIndex3.zip}.
%

We proceed as in the previous cases but now, for index-3, we further select the isotropy weights that originate a possible set of Chern numbers according to Table~\ref{non spin}.

The isotropy weights  that satisfy all the necessary conditions can be found in the file ResultsIndex3.pdf and are presented in Figure~\ref{fig:Index3Weights}. The corresponding multigraphs are shown in Figure~\ref{fig:Index3}, where the first two multisets of weights  correspond to the multigraph on the left, while the other four possibilities correspond to the multigraph on the right. In all possible cases the Chern numbers obtained from the  weights are
\begin{align*}
c_1^4[M] = 405, \quad c_1^2c_2[M]=198 \quad \text{and} \quad c_2^2[M]=97.
\end{align*}

\begin{figure}[h!]
\begin{center}
\includegraphics[scale=0.33,trim={1cm .2cm 4cm .2cm},clip] {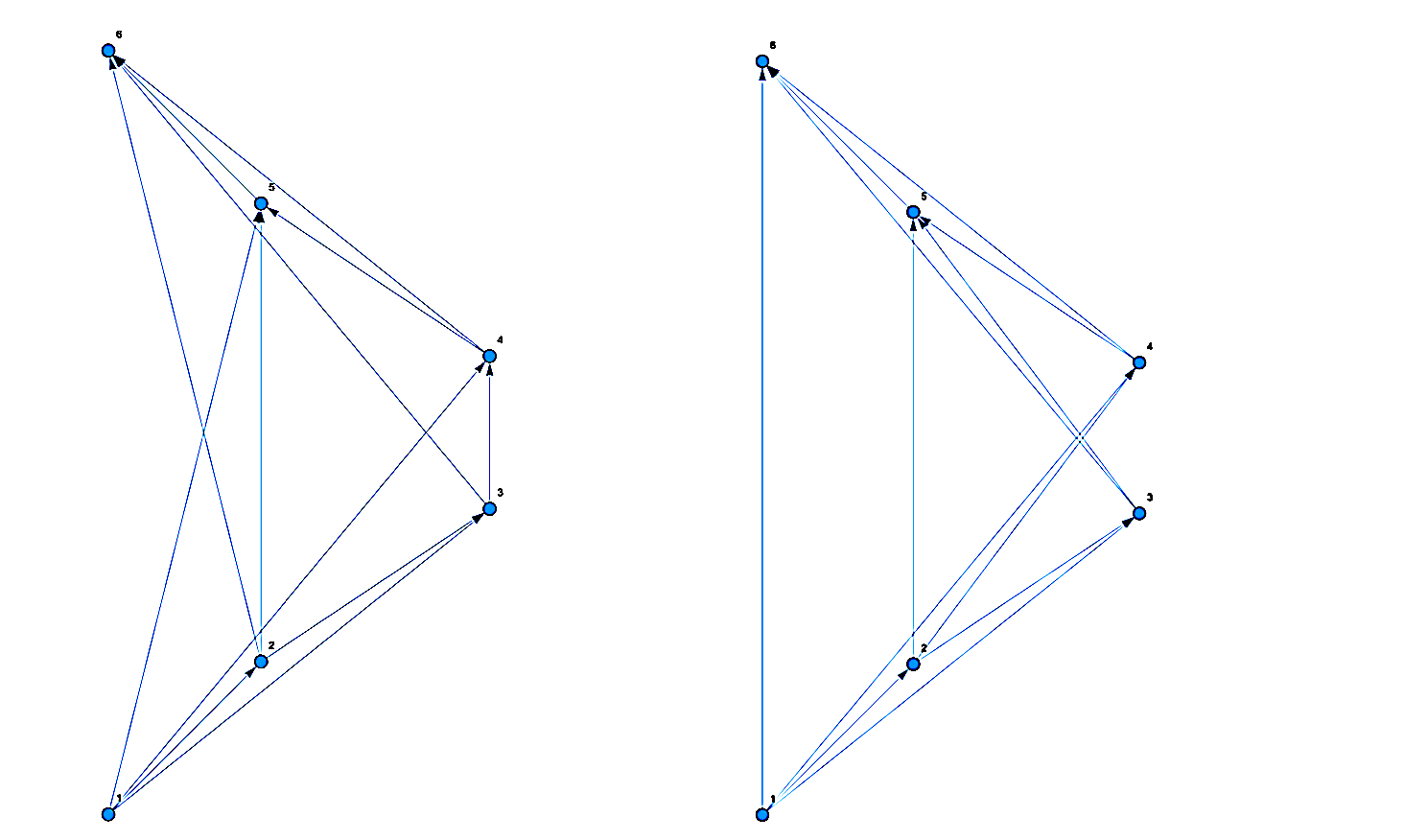}
\caption{Multigraphs for Index $3$}\label{fig:Index3}
\end{center}
\end{figure}

\begin{figure}[h!]
\begin{center}
\includegraphics[scale=0.5,trim={4cm 3cm 5cm 1cm},clip] {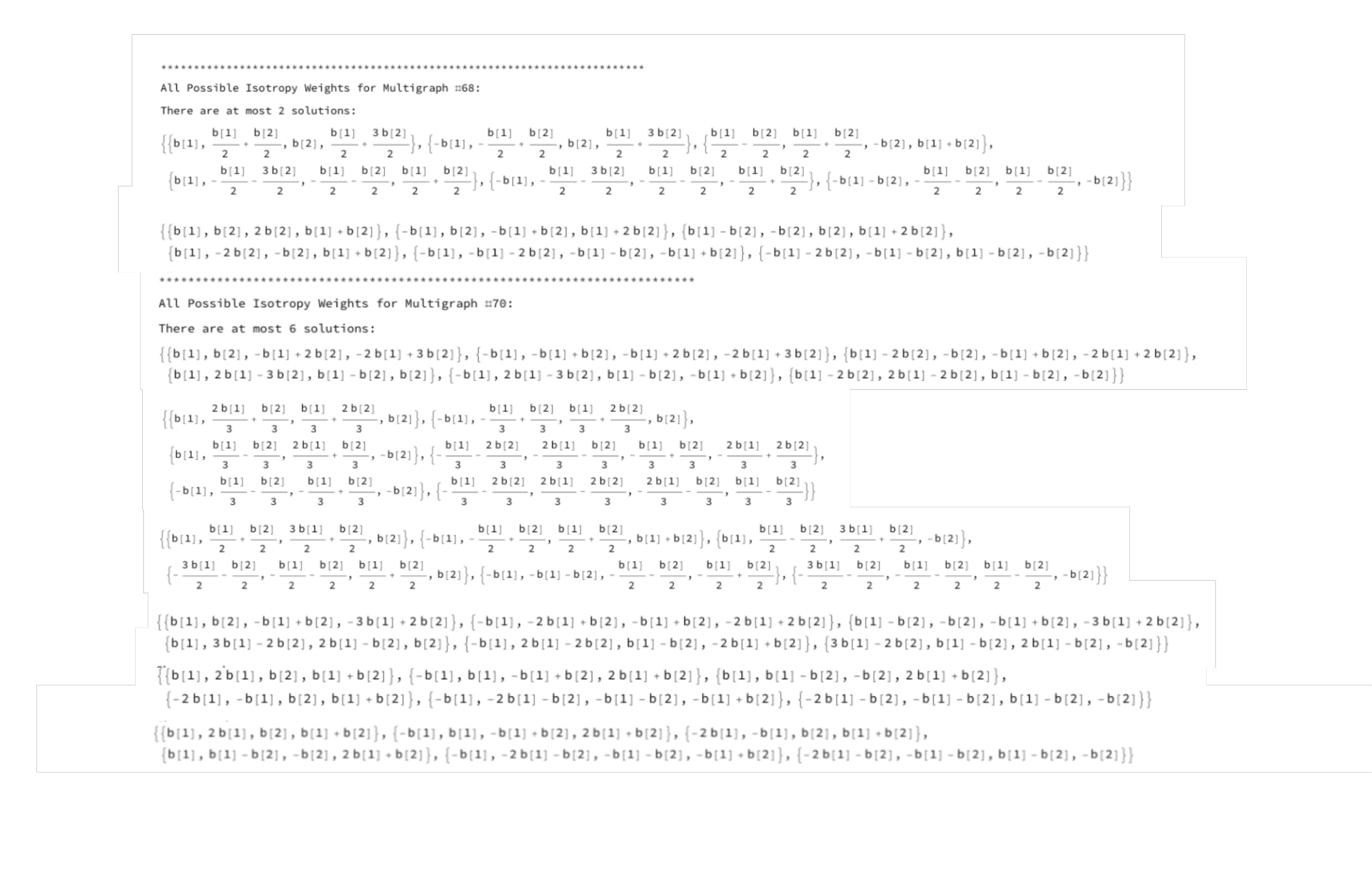}
\caption{Isotropy weights for Index $3$}\label{fig:Index3Weights}
\end{center}
\end{figure}

All the multisets of weights obtained give us a $2$-dimensional family of circle actions that generate the effective $\mathbb{T}$-action on the quintic del Pezzo fourfold $W$ described in Section~\ref{subsec W}.


\noindent  $\bullet$  The first $2$-dimensional family of weights listed in Figure~\ref{fig:Index3Weights}  with $\text{b[1]},\text{b[2]} \in \Z_{>0}$ satisfying  $\text{b[2]} >   \text{b[1]}$, $ \text{b[1]} = \text{b[2]} \mod 2$ and 
$$
\gcd\left(\text{b[1]},\text{b[2]} ,\frac{\text{b[1]}+\text{b[2]}}{2},\frac{\text{b[1]}+3\text{b[2]} }{2}\right) = 1,
$$ 
generates the effective $\mathbb{T}$-action on the quintic del Pezzo fourfold $W$ (cf. Section~\ref{subsec W}) since,
writing $\hat{\alpha}=\frac{b_1 +b_2}{2} \in \Z$ and $\hat{\beta}=\frac{b_1+3b_2}{2}\in \Z$, the $2$-dimensional family of weights becomes
\begin{align*}
\left\{ \{3\ha - \hb, \ha, -\ha+\hb, \hb\},\{-3\ha + \hb, -2\ha+\hb, -\ha + \hb,\hb\} , \{ 2\ha-\hb,\ha, \ha - \hb,2\ha\}, \right. \\ \left. \{3\ha - \hb, -\hb, -\ha, \ha\} , \{-3\ha+\hb, -\hb, -\ha, -2\ha+\hb\} , \{-2\ha, -\ha, 2\ha-\hb, \ha-\hb\} \right\},
\end{align*}
which is the same as that of Table~\ref{values c1T}.
The multigraphs obtained from the projections of the $\mathbb{T}$-action to a generic subcircle defined by $\ha \,v_1 + \hb \,v_2 \in \widetilde{\ell}$ are always as the one on the left of Figure~\ref{fig:Index3} and, considering the ordering on the fixed point set given by the associated moment maps, we obtain that 
$$
p_1 \prec p_0\prec p_2^\prime \prec  p_3  \prec p_2 \prec p_4.
$$


\noindent  $\bullet$ The second $2$-dimensional family of weights listed in Figure~\ref{fig:Index3Weights} also generates the effective $\mathbb{T}$-action on  $W$ since, if we consider the basis ${\emph b}_1 =\alpha -\beta=(-1,-2)$ and ${\emph b}_2=\alpha=(-1,1)$ of the dual lattice $\tilde{\ell}^*$ of $\mathbb{T}$ given in Section~\ref{subsec W}, the $\mathbb{T}$-weights  listed in Table~\ref{values c1T} become those in Table~\ref{NewweigthsW}.

\renewcommand{\arraystretch}{1.2}
\begin{table}[h!]
\begin{center}
  \begin{tabular}{| c || c  c  c  c |  }
    \hline
    $p_0$ & $-\bb_1 - 2 \bb_2$, & $-\bb_1-\bb_2$, & $-\bb_1$, & $-\bb_1+\bb_2$  \\ \hline
    $p_1$ & \,\,\,$\bb_1+2 \bb_2$, & $\bb_2$, &  $-\bb_1$, & $-\bb_1+\bb_2$  \\ \hline
    $p_2$ &  $-\bb_1-2 \bb_2$, &   -$\bb_2$, & $-\bb_1-\bb_2$, & $\bb_1-\bb_2$ \\ \hline
   $p'_2$ & $\bb_1$, & $\bb_1+\bb_2$, & $\bb_2$, & $2\bb_2$\\ \hline
  $p_3$ &$\bb_1+2\bb_2$, & $-\bb_2$, &  $\bb_2$, & $\bb_1-\bb_2$\\ \hline
   $p_4$ & $\bb_1$, & $-2\bb_2$, & $\bb_1+\bb_2$, & $-\bb_2$  \\ 
    \hline
  \end{tabular}
 \end{center}
  \caption{
Weights for the effective $\mathbb{T}$-action on $W$ for the new basis of $\tilde{\ell}^*$.}
 \label{NewweigthsW}
\end{table}
The multigraphs obtained for these circle actions from the projection of the  $\mathbb{T}$-action to a generic subcircle are always as the one on the left of Figure~\ref{fig:Index2}. Moreover, considering  the corresponding moment maps, we obtain the following ordering of the fixed point set as defined in \eqref{ordering fixed points}
$$
p_2^\prime \prec p_1\prec p_3 \prec p_4 \prec p_0 \prec p_2.
$$

 
\noindent  $\bullet$ Similarly, the third $2$-dimensional family of weights listed in Figure~\ref{fig:Index3Weights} generates  the $\mathbb{T}$-action on  $W$ since, if we consider the basis ${\emph b}_1 = \beta$ and ${\emph b}_2=-\alpha+\beta$ of the dual lattice $\tilde{\ell}^*$ of $\mathbb{T}$,  the weights for the action of $\mathbb{T}$ listed in Table~\ref{values c1T} become those in Table~\ref{NewweigthsW_3}.
\renewcommand{\arraystretch}{1.2}
\begin{table}[h!]
\begin{center}
  \begin{tabular}{| c || c  c  c  c |  }
    \hline
    $p_0$ & $-2\bb_1 +3  \bb_2$, & $-\bb_1+2\bb_2$, & $\bb_2$, & $\bb_1$  \\ \hline
    $p_1$ & \,\,\,$2\bb_1 -3  \bb_2$, & $\bb_1-\bb_2$, &  $\bb_2$, & $\bb_1$ \\ \hline
    $p_2$ &  $-2\bb_1+3 \bb_2$, &   $-\bb_1+\bb_2$, & $-\bb_1+2\bb_2$, & $-\bb_1$ \\ \hline
   $p'_2$ & $-\bb_2$, & $\bb_1-2\bb_2$, & $\bb_1-\bb_2$, & $2\bb_1-2\bb_2$\\ \hline
  $p_3$ &$2\bb_1-3\bb_2$, & $-\bb_1+\bb_2$, &  $\bb_1-\bb_2$, & $-\bb_1$\\ \hline
   $p_4$ & $-\bb_2$, & $-2\bb_1+2\bb_2$, & $\bb_1-2\bb_2$, & $-\bb_1+\bb_2$  \\ 
    \hline
  \end{tabular}
 \end{center}
 \caption{
Weights for the effective $\mathbb{T}$-action on $W$ for the new basis of $\tilde{\ell}^*$.}
 \label{NewweigthsW_3}
\end{table}
The multigraphs obtained for these circle actions from the projection of the  $\mathbb{T}$-action to a generic subcircle are always as the one on the right of Figure~\ref{fig:Index3} and the ordering on the fixed point set obtained from the corresponding moment maps is
$$
p_0 \prec p_2\prec p_4 \prec p_1 \prec p_3 \prec p_2^\prime.
$$

\noindent  $\bullet$  The forth $2$-dimensional family of weights  in Figure~\ref{fig:Index3Weights} 
with $\text{b[1]},\text{b[2]} \in \Z_{>0}$ such that  $\text{b[2]} >   \text{b[1]}$, $ \text{b[1]} = \text{b[2]} \mod 2$ and 
$$
\gcd\left(\text{b[1]},\text{b[2]} ,\frac{\text{b[1]}+\text{b[2]}}{2},\frac{3\text{b[1]}+ \text{b[2]} }{2}\right) = 1,
$$ 
also generates the  $\mathbb{T}$-action on $W$.  In particular,
writing $u=\frac{2b_1 +b_2}{3} \in \Z$ and $v=\frac{b_1+2b_2}{3}\in \Z$ and considering 
$\hat{\alpha}=u-v$ and $\hat{\beta}=2u-v$, the $2$-dimensional family of weights becomes
\begin{align*}
\left\{\{\hb,  -\ha + \hb,-2\ha+\hb,  -3\ha + \hb, \} , \{ -\hb, -\ha, -2\ha+\hb,-3\ha+\hb\} , \{\hb,  \ha, -\ha+\hb,  3\ha - \hb \}, \right.  \\ \left. \{2\ha-\hb, \ha-\hb, -\ha, -2\ha,\} , \{-\hb,\ha,  -\ha, 3\ha - \hb \} , \{ 2\ha-\hb,2\ha, \ha, \ha - \hb \}\right\},
\end{align*}
which is the same as that of Table~\ref{values c1T}.
The multigraphs obtained from the projection of the $\mathbb{T}$-action to the corresponding  generic subcircles are always as the one on the right of Figure~\ref{fig:Index3} and the ordering of the fixed point set obtained  is 
$$
p_0 \prec p_2\prec p_1 \prec p_4 \prec p_3\prec p_2^\prime. 
$$

\noindent  $\bullet$  The fifth $2$-dimensional family of weights  in Figure~\ref{fig:Index3Weights}  with $ \text{b[2]} >  \text{b[1]}>0$ and 
$$\text{b[1]} = \text{b[2]} \pmod{2}$$
also generates the $\mathbb{T}$-action on $W$.  In particular,
writing $u=\frac{-b_1 +b_2}{2} \in \Z$ and $v=\frac{b_1+b_2}{2}\in \Z$ and considering 
$\hat{\alpha}=v$ and $\hat{\beta}=-u+2v$, the $2$-dimensional family of weights becomes
\begin{align*}
\left\{\{-\ha+\hb,  \ha,  \hb, 3\ha-\hb \} , \{ \ha-\hb, 2\ha-\hb, \ha,2\ha \} , \{-\ha+\hb,  -2\ha + \hb, \hb, - 3\ha + \hb \}, \right.  \\ \left.  \{ -\hb, -\ha, \ha, 3\ha - \hb \}, \{\ha-\hb, -2\ha,-\ha,  2\ha - \hb\} , \{-\hb,-\ha,  -2\ha+\hb, -3\ha + \hb \} \right\}.
\end{align*}
The multigraphs obtained from the projection of the $\mathbb{T}$-action to the generic subcircles are always as the one on the right of Figure~\ref{fig:Index3} and the ordering of the fixed point set obtained  is 
$$
p_1 \prec p_2^\prime \prec p_0 \prec p_3 \prec p_4\prec p_2. 
$$

\noindent  $\bullet$ The sixth $2$-dimensional family of weights listed in Figure~\ref{fig:Index3Weights} with $ \text{b[2]} >  2\text{b[1]}>0$ and $\gcd (\text{b[1]},\text{b[2]})=1$
generates  the $\mathbb{T}$-action on  $W$ since, if we consider the basis ${\emph b}_1 = - 2\alpha + \beta$ and ${\emph b}_2= - 3 \alpha+\beta$ of the dual lattice $\tilde{\ell}^*$ of $\mathbb{T}$,  the weights for the action of $\mathbb{T}$ listed in Table~\ref{values c1T} become those in Table~\ref{NewweigthsW_6}.
\renewcommand{\arraystretch}{1.2}
\begin{table}[h!]
\begin{center}
  \begin{tabular}{| c || c  c  c  c |  }
    \hline
    $p_0$ & $ \bb_2$, & $\bb_1$, & $2\bb_1-\bb_2$, & $3 \bb_1-2 \bb_2$  \\ \hline
    $p_1$ & \,\,\,$- \bb_2$, & $\bb_1-\bb_2$, &  $2\bb_1-\bb_2$, & $3\bb_1-2\bb_2$ \\ \hline
    $p_2$ &  $ \bb_2$, &   $-\bb_1+\bb_2$, & $\bb_1$, & $-3\bb_1+2\bb_2$ \\ \hline
   $p'_2$ & $-2\bb_1+\bb_2$, & $-\bb_1$, & $\bb_1-\bb_2$, & $2\bb_1-2\bb_2$\\ \hline
  $p_3$ &$-\bb_2$, & $-\bb_1+\bb_2$, &  $\bb_1-\bb_2$, & $-3\bb_1+2\bb_2$\\ \hline
   $p_4$ & $-2\bb_1 + \bb_2$, & $-2\bb_1+2\bb_2$, & $-\bb_1$, & $-\bb_1+\bb_2$  \\ 
    \hline
  \end{tabular}
 \end{center}
  \caption{
Weights for the effective $\mathbb{T}$-action on $W$ for the new basis of $\tilde{\ell}^*$.}
 \label{NewweigthsW_6}
\end{table}
The multigraphs obtained for these circle actions from the projection of the  $\mathbb{T}$-action to a generic subcircle are always as the one on the right of Figure~\ref{fig:Index3} and the ordering on the fixed point set obtained from the corresponding moment maps is
$$
p_2 \prec p_4\prec p_3 \prec p_0 \prec p_2^\prime \prec p_1.
$$


$\bullet$ The seventh $2$-dimensional family of weights listed in Figure~\ref{fig:Index3Weights} with $ \text{b[2]} >  \text{b[1]}>0$ and $\gcd (\text{b[1]},\text{b[2]})=1$
generates  the $\mathbb{T}$-action on  $W$ as we can easily see, if we consider the basis ${\emph b}_1 = - \alpha$ and ${\emph b}_2= 2\alpha - \beta$ of the dual lattice $\tilde{\ell}^*$ of $\mathbb{T}$.
The multigraphs obtained from the projection of the $\mathbb{T}$-moment polytope are always as the one on the right of Figure~\ref{fig:Index3}
and the ordering of the fixed point set obtained from the corresponding moment maps is 
$$
p_4 \prec p_3\prec p_2 \prec p_2^\prime \prec p_1 \prec p_0.
$$

$\bullet$ The eighth $2$-dimensional family of weights listed in Figure~\ref{fig:Index3Weights} with $ \text{b[2]} >  \text{b[1]}>0$ and $\gcd (\text{b[1]},\text{b[2]})=1$
also generates  the $\mathbb{T}$-action on  $W$. We  can easily see this again by considering the basis ${\emph b}_1 = - \alpha$ and ${\emph b}_2= 2\alpha - \beta$ of the dual lattice $\tilde{\ell}^*$ of $\mathbb{T}$ as in the previous example.
The multigraphs obtained from the projection of the $\mathbb{T}$-moment polytope are again as the one on the right of Figure~\ref{fig:Index3}
and the only thing different from the previous family is the ordering of the fixed point set obtained from the corresponding moment maps, which now interchanges $p_2$ and $p_2^\prime$:
$$
p_4 \prec p_3\prec p_2^\prime \prec p_2 \prec p_1 \prec p_0.
$$

\subsubsection{Index 4}
For index  $4$, since the greatest common divisor of the elements of the partitions considered is $4$ there is only one possible partition of $48$: the one with all elements equal to $4$.
We proceed as in the previous indices  and, in the end, the only multiset of  isotropy weights that satisfies all the necessary conditions is listed in Figure~\ref{fig:Index4Weights}. The corresponding multigraph is the one in Figure~\ref{fig:Index4}. The Chern numbers obtained from these isotropy weights are
\begin{align*}
c_1^4[M] = 512, \quad c_1^2c_2[M]=224 \quad \text{and} \quad c_2^2[M]=98.
\end{align*}
\begin{figure}[h!]
\begin{center}
\includegraphics[scale=0.33] {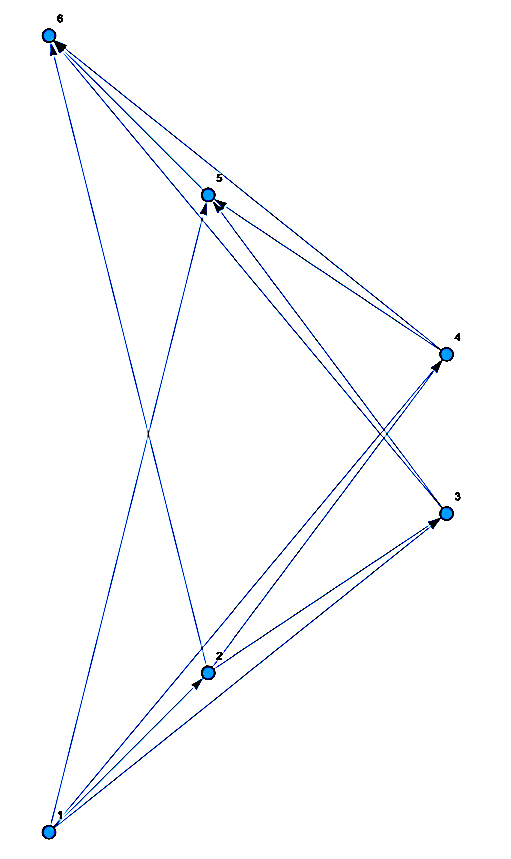}
\caption{Multigraph for Index $4$}\label{fig:Index4}
\end{center}
\end{figure}

\begin{figure}[h!]
\begin{center}
\includegraphics[scale=0.8] {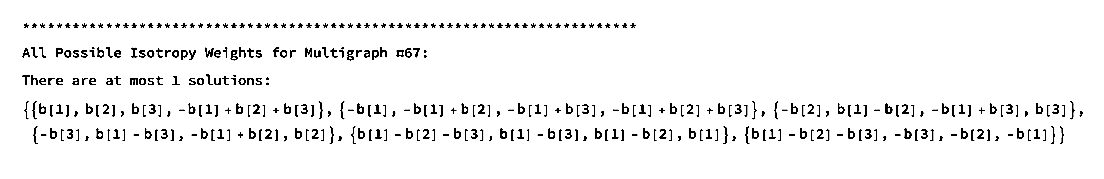}
\caption{Isotropy weights for Index $4$}\label{fig:Index4Weights}
\end{center}
\end{figure}
The $3$-dimensional family of weights  in Figure~\ref{fig:Index4Weights}  with $\text{b[1]},\text{b[2]},\text{b[3]} \in \Z_{>0}$ satisfying  $\text{b[2]} >   \text{b[1]}$, $ \text{b[3]} > \text{b[1]}$ and 
$$
\gcd\left(\text{b[1]},\text{b[2]} ,\text{b[3]} \right) = 1,
$$ 
generates the action of the $3$-dimensional torus $\mathbb{T}$ on the Grassmannian $Q= Gr(2,4)$ described in Section~\ref{subsec Q}.  Indeed,
if we consider the basis ${\emph b}_1 =  \beta$, ${\emph b}_2= \beta + \gamma$  and ${\emph b}_3= \alpha+ \beta$  of  $\frak{t}^*$,  the weights for the action of $T$ listed in Table~\ref{values c1T Q} become those in Table~\ref{NewweigthsQ}.
\renewcommand{\arraystretch}{1.2}
\begin{table}[h!]
\begin{center}
  \begin{tabular}{| c || c  c  c  c |  }
    \hline
    $q_0$ & $ \bb_1$, & $\bb_2$, & $-\bb_1+\bb_2+\bb_3$, & $ \bb_3$  \\ \hline
    $q_1$ & \,\,\,$- \bb_1$, & $-\bb_1+\bb_3$, &  $-\bb_1+\bb_2$, & $-\bb_1+\bb_2+\bb_3$ \\ \hline
    $q_2$ &  $ \bb_1-\bb_2$, &   $-\bb_2$, & $-\bb_1+\bb_3$, & $\bb_3$ \\ \hline
   $q'_2$ & $\bb_1-\bb_3$, & $-\bb_3$, & $\bb_2$, & $-\bb_1+\bb_2$\\ \hline
  $q_3$ &$\bb_1-\bb_2-\bb_3$, & $\bb_1-\bb_3$, &  $\bb_1-\bb_2$, & $\bb_1$\\ \hline
   $q_4$ & $\bb_1-\bb_2-\bb_3$, & $-\bb_2$, & $-\bb_3$, & $-\bb_1$  \\ 
    \hline
  \end{tabular}
 \end{center}
  \caption{
Weights for the effective $T$-action on $Q$ for the new basis of $\frak{t}^*$.}
 \label{NewweigthsQ}
\end{table}
The multigraphs obtained for these circle actions from the projection of the  $T$-action to a generic subcircle are always as the one in Figure~\ref{fig:Index4} and the ordering on the fixed point set obtained from the corresponding moment maps is
$$
q_0 \prec q_1\prec q_2 \prec q_2^\prime \prec q_3 \prec q_4.
$$
\subsection{Additional Results}
In this section we prove Theorem~\ref{thm:add}. For this, we consider the additional cases for the Betti numbers under the conditions of Proposition~\ref{index betti} and rule out some of the possibilities using the algorithm in \cite{GS}.  

For the remaining of this section $(M,\omega)$ denotes  a  closed, symplectic manifold admitting a Hamiltonian $T^2$-action with isolated fixed points.
\subsubsection{Index 2, $b_2(M)=1$ and $3\leq b_4(M)\leq 6$} In this section, we consider the case where the index $k_0$ of $(M,\omega)$ is $2$ and we assume $b_2(M)=1$, running Part 1 of the file {\texttt index2b2\_1\_b4\_i.nb}, with $i=3,4,5,6$ to generate the list of multigraphs $\Gamma=(E,V)$ describing a circle action when $3\leq b_4(M)\leq 6$. 
Note that, from the classification results in \cite{GS}, if $b_2(M)=b_4(M)=1$, then the index must be $k_0=5$ and that, if $b_2(M)=1$ and $b_2(M)=2$, the classification was completed in Main Theorem.

By Theorem~\ref{thm:ChernNumbers}, the number of fixed points is, in these cases, equal to $4+b_4(M)$ and that $c_1c_{3}[M]=52 - 2\,b_4(M)$. 
Then, for each value of $b_4(M)$, we run Part 2 of the file {\texttt index2b2\_1\_b4\_i.nb}, considering partitions $\mathsf{m}(e)$ of $52 -2 b_4(M)$ into 
$$\lvert E \rvert = 4(4+b_4(M))/2= 8+2b_4(M)$$ 
positive even integers, such that the first and last elements of the partitions are always equal to $2$ (cf. Lemma~\ref{tau 1}), checking which  of these annihilate a determinant of a matrix $A(\Gamma_i) - \text{diag}(\mathsf{m}(E_i))$, (according to the algorithm described in Section~\ref{sec:algorithm}). These  partitions can be found in the files {\texttt graph\#.txt} stored in the folder {\texttt b4\_i} for $i=3,4,5,6$, that can be accessed through the download buttons ``Additional Files".

We run Part 3 of the file {\texttt index2b2\_1\_b4\_i.nb} to sort these partitions according to the rank of the matrices, dividing them into two sets: those that originate singular matrices with maximal rank, and those that originate matrices of lower rank. In the first case, we also select the partitions that originate matrices with null spaces containing vectors with positive entries. We store the first type of partitions in the files Part$\#$toprk.txt and the others in the files Part$\#$smallrk.txt.

Then, Part 4 of the file {\texttt index2b2\_1\_b4\_i.nb} considers the first set of partitions, producing a list of the corresponding isotropy weights, checking if they satisfy the polynomial equations in  \cite[Equation (1.1)]{GS} (which are obtained from the ABBV localization formula). 

Finally, Part $5$ of {\texttt index2b2\_1\_b4\_i.nb} considers the second set of partitions, selects those that originate matrices with null spaces containing vectors with positive entries,  produces the list of the corresponding isotropy weights and then selects those that satisfy  \cite[Equation (1.1)]{GS} as well as the necessary conditions $(i)$ to $(iii)$ listed in Section~\ref{secresults:index1}.
At each step, the resulting lists of isotropy weights are saved in different files so that it is easy to verify which isotropy weights are discarded with each test performed. All these files can be found in the folders {\texttt b4\_i} for $i=3,4,5,6$.

In the end, the program eliminates all possibilities. Indeed, no potential multiset of isotropy weights satisfies all the necessary conditions and so, in particular,  we conclude {\bf that there are no possible $T^2$-actions} with isolated fixed points in the index-$2$ case when $b_2(M)=1$ and $3\leq b_4(M)\leq 6$.

\subsubsection{Index $3$}\label{sec:index3_CP2} If the index of $(M,\omega)$ is $3$, then we know, from Proposition~\ref{index betti}, that $(b_2(M),b_4(M))$  is either $(1,2)$ or $(2,3)$ and the first case was already studied in Section~\ref{subsec:index3}. Let us  know describe the results obtained in the second case.  For this, we run
Part 1 of the file {\texttt Index3b2\_2\_b4\_3.nb} to generate the list of multigraphs $\Gamma=(E,V)$ describing a circle action when $b_2(M)=2$ and $b_4(M)=3$. Note that the number of fixed points is, in this case, equal to $9$ and that $c_1c_{3}[M]=54$. We obtain the multigraphs listed in the file {\texttt Multigraphsb2\_2\_b4\_3.zip}.

Since there is only one partition $\mathsf{m}(e)$ of $54$ into  $18 = \lvert E \rvert = 4 (6+b_4(M))/2$ positive integers that are multiples of $3$, namely the partition $L$ with all $18$ elements equal to $3$ (cf. Lemma~\ref{tau 1}), we run Part 2 of  {\texttt Index3b2\_2\_b4\_3.nb} to determine which  matrices $A(\Gamma_i) - \text{diag}(L)$ are singular, (according to the algorithm described in Section~\ref{sec:algorithm}). The indices of the multigraphs that produce these matrices are stored in the file {\texttt numberofgraphs.txt}.

We run Part 3 of the file {\texttt Index3b2\_2\_b4\_3.nb} to sort the singular matrices according to their rank, dividing them into two sets. The indices of the matrices with maximal rank are stored in the file {\texttt Matricestoprk.txt} and the remaining ones in the file {\texttt Matricessmallrk.txt}.

Then, Part 4 of the file {\texttt Index3b2\_2\_b4\_3.nb} considers the first set of matrices, producing a list of the corresponding isotropy weights, checking if they satisfy the polynomial equations in  \cite[Equation (1.1)]{GS}.

Finally, Part $5$ of  {\texttt Index3b2\_2\_b4\_3.nb}  considers the second set of matrices and runs the necessary tests as in the previous sections.
At each step, the resulting lists of isotropy weights are saved in different files so that it is easy to verify which isotropy weights are discarded with each test performed. All these files can be found in the folder {\texttt IsotropyWeights} that can be accessed through the download button ``Additional Files" at the bottom of the  \href{https://www.math.tecnico.ulisboa.pt/~lgodin/CircleActions/page6.html}{Additional Results} webpage.

Up to a permutation of the vertices corresponding to fixed points of index $4$, the $2$-dimensional and $4$-dimensional families of isotropy weights that satisfy all the necessary conditions  as in previous sections can be found in the file {\texttt Conclusion.pdf} (accessed through the download button ``Results" at the bottom of the webpage)  and are presented in Figure~\ref{fig:weightsCP2CP2}. 
\begin{figure}[h!]
\begin{center}
\includegraphics[scale=0.5,trim={0.9cm 9.7cm 1.3cm 0.3cm},clip] {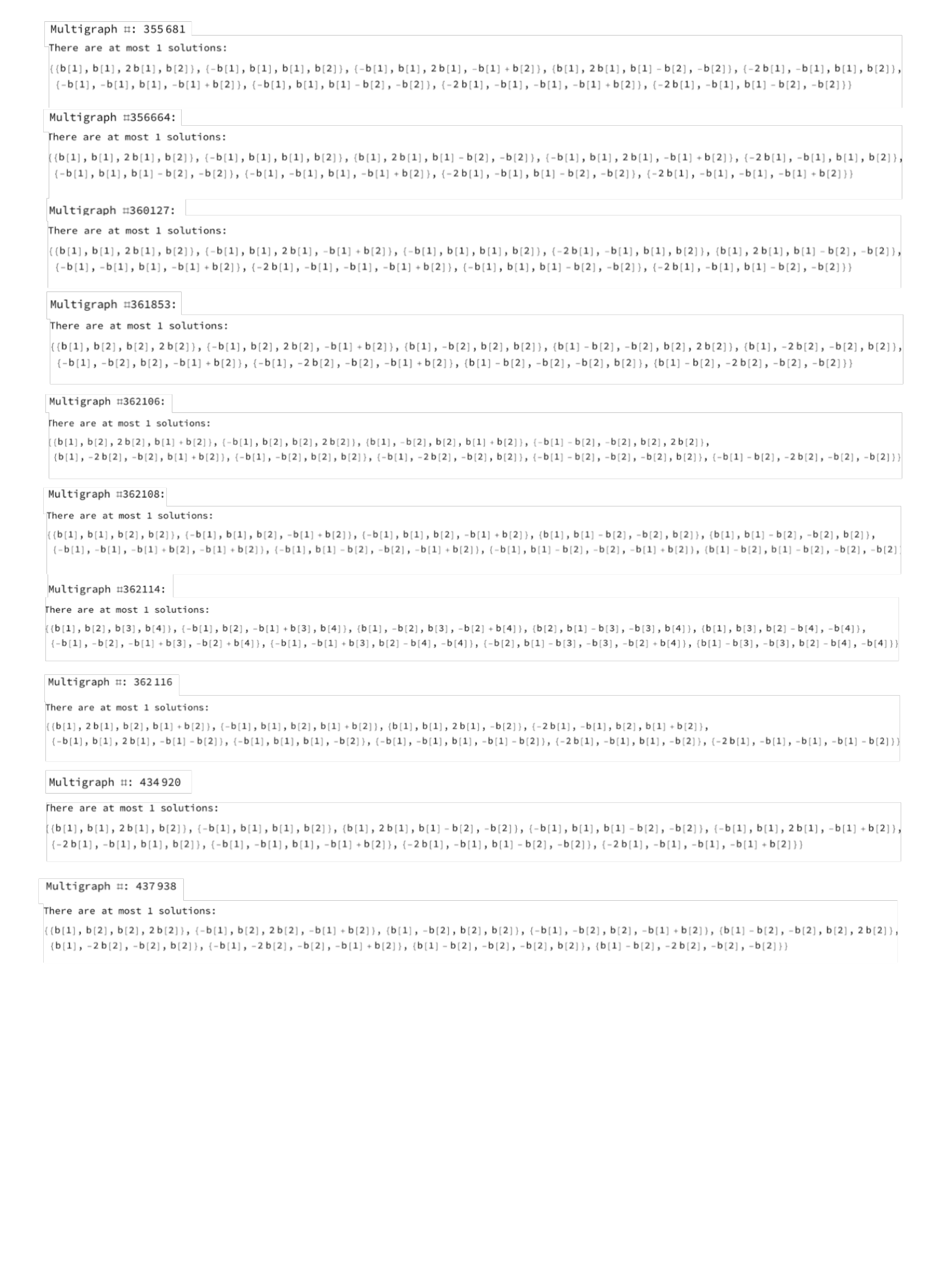}
\caption{Weights}\label{fig:weightsCP2CP2}
\end{center}
\end{figure}
The corresponding multigraphs are  shown in Figure~\ref{fig:CP2CP2}. In all cases the Chern numbers obtained are 
\begin{align*}
c_1^4[M] = 486, \quad c_1^2c_2[M]=216 \quad \text{and} \quad c_2^2[M]=99.
\end{align*}

It is easy to check that the $4$-dimensional family of isotropy weights corresponding to the seventh multigraph in Figure~\ref{fig:CP2CP2} generates an action of a $4$-torus whose weights are the same as those of the toric action on $\C P^2\times \C P^2$ described in Theorem~\ref{thm:CP2CP2}.
Moreover, the remaining $2$-dimensional families of weights listed in Figure~\ref{fig:weightsCP2CP2} generate actions of $2$-subtori of this $4$-dimensional torus. For example, the first $2$-dimensional family of weights of Figure~\ref{fig:CP2CP2} can be obtained from the seventh by taking ${\texttt b[2]=b[1]}$  and  ${\texttt b[4]= 2 b[1]}$. 

\begin{figure}[h!]
\begin{center}
\includegraphics[scale=0.33,trim={0cm 2.5cm 0cm 0cm},clip] {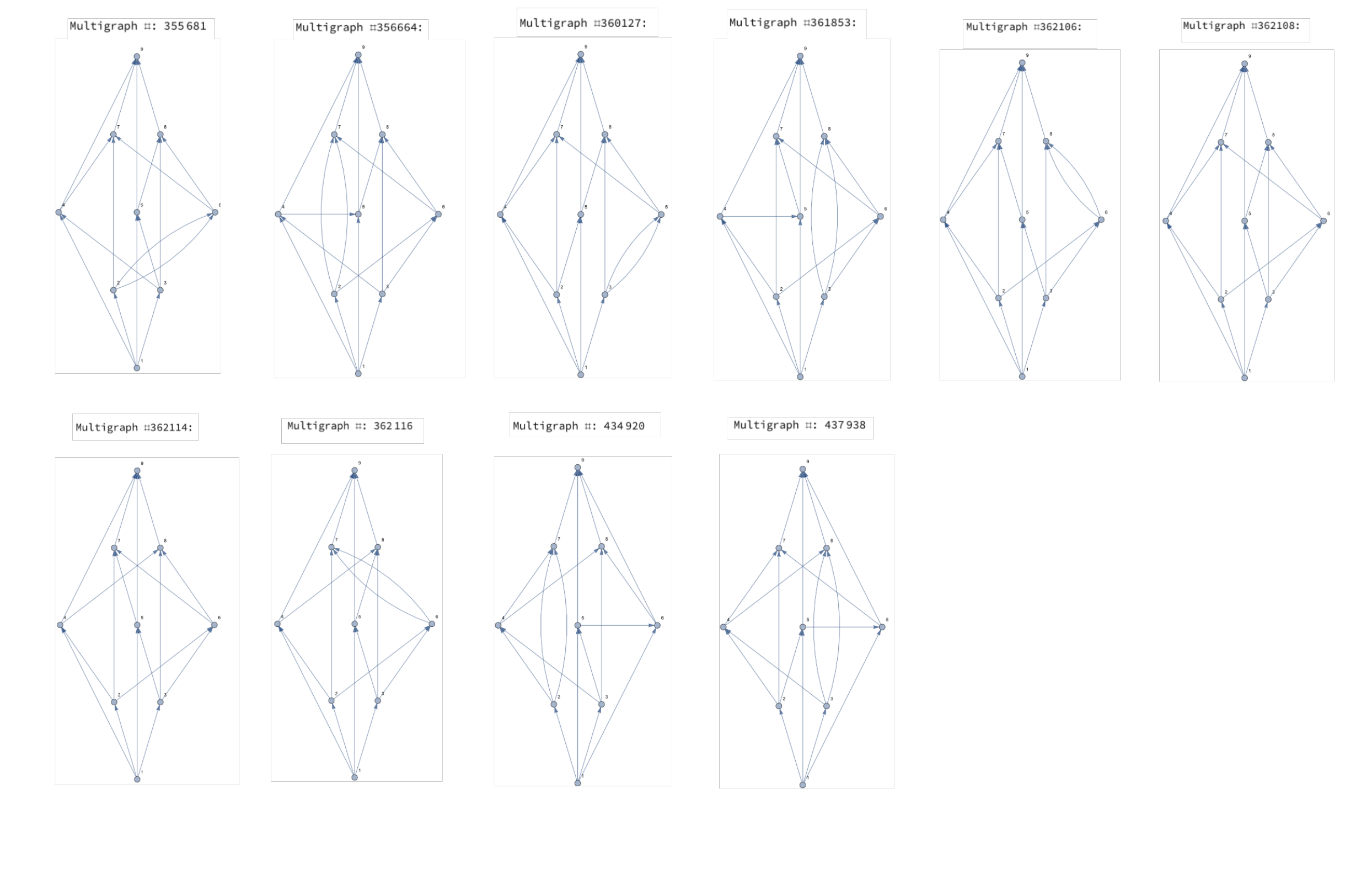}
\caption{Multigraphs}\label{fig:CP2CP2}
\end{center}
\end{figure}

\bibliographystyle{plain}
\bibliography{Fano_fourfold}

\end{document}